\documentclass[a4paper,12pt,oneside]{article}
\pdfoutput=1

\usepackage[english]{babel}
\usepackage[latin1]{inputenc}
\usepackage[T1]{fontenc}

\usepackage{amsmath,amssymb,amsthm,mathtools}

\usepackage[shortlabels]{enumitem}
\setlist{nosep}

\usepackage{titling}
\usepackage{fancyhdr}
\usepackage{hyperref}

\title{Geometry of geodesics}
\author{Joonas Ilmavirta\\\texttt{joonas.ilmavirta@jyu.fi}}
\date{July 2020}

\theoremstyle{plain}
\newtheorem{theorem}{Theorem}[section]
\newtheorem{lemma}[theorem]{Lemma}
\newtheorem{corollary}[theorem]{Corollary}
\newtheorem{proposition}[theorem]{Proposition}

\theoremstyle{definition}
\newtheorem{definition}[theorem]{Definition}

\theoremstyle{remark}
\newtheorem{remark}[theorem]{Remark}
\newtheorem{exx}{Exercise}[section]

\newtheorem{iexx}[exx]{Important exercise}

\reversemarginpar
\newcommand{\tahti}{${}$\marginpar{{\Large$\qquad\star$}}}

\newcommand\Xqed[1]{\leavevmode\unskip\penalty9999 \hbox{}\nobreak\hfill\quad\hbox{#1}}
\newenvironment{ex}{\exx}{\Xqed{$\bigcirc$}\endexx}

\newenvironment{iex}{\iexx\tahti}{\Xqed{$\bigcirc$}\endiexx}
\newcommand{\qa}{\begin{iex}Do you have any questions or comments regarding section~\thesection? Was something confusing or unclear? Were there mistakes?\end{iex}}

\newcommand{\R}{\mathbb{R}}

\newcommand{\N}{\mathbb{N}}
\newcommand{\ip}[2]{\left\langle#1,#2\right\rangle}
\newcommand{\iip}[2]{\left(#1,#2\right)}
\newcommand{\der}{\mathrm{d}}
\newcommand{\Der}[1]{\frac{\der}{\der #1}}
\newcommand{\pDer}[1]{\frac{\partial}{\partial #1}}
\newcommand{\dd}{\,\der}
\newcommand{\eps}{\varepsilon}
\renewcommand{\phi}{\varphi}
\newcommand{\abs}[1]{\left\lvert #1 \right\rvert}
\newcommand{\aabs}[1]{\left\| #1 \right\|}
\DeclareMathOperator{\spt}{spt}
\newcommand{\xrt}{\mathcal{I}}
\DeclareMathOperator{\id}{id}
\DeclareMathOperator{\im}{im}

\newcommand{\inwb}{{\partial_{\text{in}}(SM)}}
\newcommand{\dum}{{\,\cdot\,}}
\newcommand{\Cs}[3]{\Gamma^{#1}_{\phantom{#1}#2#3}}
\newcommand{\IF}{I}
\newcommand{\NVF}{NV\!F}
\newcommand{\Order}{\mathcal O}

\def \vd{\overset{\smash{\tt{v}}}{\nabla}}
\def \hd{\overset{\smash{\tt{h}}}{\nabla}}

\def \vdiv{\overset{\smash{\tt{v}}}{\mbox{\rm div}}}
\def \hdiv{\overset{\smash{\tt{h}}}{\mbox{\rm div}}}
\DeclareMathOperator{\dive}{div}

\newcommand{\LC}{Levi-Civita}

\newcommand{\nts}[1]{}

\begin{document}

\maketitle

\noindent
These are lecture notes for the course ``MATS4120 Geometry of geodesics'' given at the University of Jyv\"askyl\"a in Spring 2020. Basic differential geometry or Riemannian geometry is useful background but is not strictly necessary. Exercise problems are included, and problems marked important should be solved as you read to ensure that you are able to follow.

Previous feedback has been very useful and new feedback is welcome.


{
\renewcommand{\baselinestretch}{-1.5}\normalsize
\tableofcontents
}

\clearpage

\section{Riemannian manifolds}

\subsection{A look on geometry}

A central concept in Euclidean geometry is the Euclidean inner product, although its importance is somewhat hidden in elementary treatises.
We will relax its rigidity to allow for a certain kind of variable inner product.
This provides a rich geometrical framework --- Riemannian geometry --- and shines new light on the nature of Euclidean geometry as well.

There is much to be studied beyond Riemannian geometry, but we will not go there.
Neither will we study all of Riemannian geometry; we shall focus on the geometry of geodesics.
Gaps will be left, especially early on, and may be filled in by more general courses or textbooks on Riemannian geometry.

Yet another thing we will not be concerned with is regularity.
There are interesting phenomena in various spaces of low regularity, but even those are best understood if one has background knowledge of the simplest possible situation.
All the structures in this course will be smooth, by which we mean $C^\infty$.
Many --- but not all --- of the resulting functions will be smooth as well, and we will take some care to show how smoothness of structure implies smoothness of derived structure.

We will do local Riemannian geometry in the sense that we will implicitly be working in a single coordinate patch.
Even when a more global treatment would be needed using a partition of unity or some such tool, we will pretend that everything is still in a single patch.
This promotes the structures essential for this course.
A reader with more prior familiarity with manifolds is invited to globalize the proofs presented here in a more honest fashion.

Differential geometry can often be done in a local coordinate formalism or using invariant concepts.
We prefer an invariant approach, but the coordinate description will always be given as well so as to give more concrete and calculable definitions.

Some readers may find these notes vague or lacking in detail, but that is entirely purposeful.
The goal is to focus on a certain set of phenomena and not to be held back by technicalities.
One does not need to manually craft every atom to obtain a coherent big picture, and one might even argue that orientation to details can harm by causing the focus to drift away from the ideas that are important for the present goal.

\subsection{Smooth manifolds}

Let $n\in\N$.
A topological $n$-dimensional manifold $M$ is a topoplogical space which is second-countable\footnote{A first-countable space has a countable neighborhood base at each point, whereas a second-countable space has a countable base for the whole topology.}, Hausdorff\footnote{The Hausdorff condition is also known as the separation axiom T2. It means that any two distinct points have disjoint neighborhoods.} and ``looks locally like $\R^n$''.
The last bit in quotes means that any point $x\in M$ has a neighborhood $U\subset M$ for which there exists a homeomorphism $\phi\colon U\to\phi(U)\subset\R^n$.
Such a local homeomorphism is known as a coordinate chart as it gives Euclidean coordinates in an open subset of the manifold.

The conditions above define a topological manifold.
To make it smooth, we introduce more structure.
As $M$ itself is just an abstract space, there is no way to differentiate on it.
All derivatives will have to be considered in local Euclidean coordinates given by a chart, but on a single chart there is nothing to differentiate.

Consider two charts $\phi_i\colon U_i\to\phi_i(U_i)$ with $i=1,2$.
If the domains $U_1$ and $U_2$ intersect, we get a map between the two local coordinate systems.
Specifically, if $U\coloneqq U_1\cap U_2$, the map $\psi\colon\phi_1(U)\to\phi_2(U)$ defined by $\psi\circ\phi_1=\phi_2$ is a map between two open sets in $\R^n$.
This map is called the transition function between the two coordinate charts.

\begin{ex}
Show that the transition function $\psi$ is a homeomorphism.
\end{ex}

We say that the two coordinate charts $\phi_i$ are smoothly compatible if the map $\psi$ is a diffeomorphism.
To either satisfy or irritate the reader, we observe that if the two open sets $U_i$ do not meet, then $\psi$ is the unique map from the empty subset of $\R^n$ to itself and is vacuously smooth; this ensures that checking for compatibility only makes a difference if the two domains meet.

\begin{ex}
Is smooth compatibility an equivalence relation in the set of coordinate charts on a manifold $M$?
\end{ex}

An atlas is a collection of coordinate charts $(U_\alpha,\phi_\alpha)_{\alpha\in A}$ so that they cover the whole manifold: $\bigcup_{\alpha\in A}U_\alpha=M$.
An atlas is smooth if all pairs of coordinate charts are smoothly compatible.
A smooth atlas is maximal if no new coordinate chart can be added to it without breaking smoothness.
A maximal smooth atlas is sometimes called a smooth structure.

\begin{ex}
Show that every atlas is contained in a unique maximal atlas.
\end{ex}

\begin{definition}[Smooth manifold]
A smooth $n$-dimensional manifold is a topological $n$-manifold with a maximal smooth atlas.
\end{definition}

All regularity matters are always defined in terms of the local coordinates given by a fixed atlas.
A function $f\colon M\to\R$ on a smooth manifold is defined to be smooth when $f\circ\phi^{-1}$ is a smooth Euclidean function for any local coordinate map $\phi$.

\begin{iex}
Define what it should mean for a function $f\colon M\to N$ between two smooth manifolds of any dimension to be smooth.
\end{iex}

The Euclidean space $\R^n$ is an $n$-dimensional smooth manifold.
An atlas is given by any open cover (e.g. the singleton of the space itself) and identity maps.

\begin{remark}
\label{rmk:coordinates-after-structure}
Once we have fixed a smooth structure, a valid coordinate chart is precisely a smooth diffeomorphism $\phi\colon U\to\phi(U)\subset\R^n$ from an open set $U\subset M$.
This cannot be taken as a starting point, since before the smooth structure and its charts we do not know what smoothness of such a map would mean.
This only becomes useful later when deciding whether a given map gives valid coordinates.
\end{remark}

\subsection{Curves, vectors and differentials}
\label{sec:curve-vector-differential}

A smooth curve is a smooth map from an interval $I\subset\R$ to our smooth manifold $M$.
The velocity of a curve $\gamma\colon I\to M$ at any given time $t\in I$ is a tangent vector in the tangent space $T_{\gamma(t)}M$.
Indeed, the tangent space can be defined using velocities of curves\footnote{One says that two curves $\gamma_i$ are equivalent if in a fixed local coordinate system the Euclidean curves $\phi\circ\gamma_i$ have the same velocity at the reference point. Then a tangent vector is an equivalence class of curves. To get a coordinate invariant definition, one can also take the equivalence class over systems of coordinates.}, but it is not the only possible approach.
Different points of view are useful, and we will be free to change perspectives as convenient.
It is unimportant for us which approach one chooses to define tangent spaces.

In terms of local coordinates the tangent space $T_xM$ at $x\in M$ can be understood\footnote{It is hopefully evident that any local coordinate chart gives an identification of the tangent space $T_xM$ at $x$ with $\R^n$ with the curve approach of the preceding paragraph.} to be just $\R^n$.
A typical approach is to define a tangent vector as a derivation, a certain kind of a differential operator.
This is related to the curve-based definition as follows:
A tangent vector $W\in T_xM$ can be thought of as a differential operator or as the velocity of a curve $\gamma$ at $t=0$.
A smooth function $f\colon M\to\R$ is differentiated by $Wf=\Der{t}f(\gamma(t))|_{t=0}$.

The same object can function as the velocity of a curve or as a derivation.
It would be possible to give different incarnations of tangent vectors different names and introduce canonical isomorphisms between them, but we will leave any such identifications out.

An important feature of a tangent space is that it is a vector space.
For any $x$ on an $n$-dimensional smooth manifold $M$, the tangent space $T_xM$ is an $n$-dimensional real vector space.
It is therefore isomorphic to $\R^n$, but not in a canonical way.
Any local coordinates give a natural way to identify $T_xM\cong\R^n$, but the many possible coordinate charts in neighborhoods of $x$ give different isomorphisms\footnote{Indeed, all isomorphisms between the two vector spaces can be realized through a coordinate chart of a maximal atlas.}.

The dual vector space $T_xM$ is called the cotangent space and denoted by $T_x^*M$.
One could also define $T_x^*M$ first and then define $T_xM$ by duality.
The most important example of a covector is the differential of a function $f\colon M\to\R$.
The differential at $x\in M$ is $\der f_x\in T_x^*M$ and the duality pairing is defined by
\begin{equation}
\der f_x(W)
=
Wf
\end{equation}
for any $W\in T_xM$, considered as a derivation.
Be careful to call this the differential, not the gradient, of a function.

We shall study vectors and covectors in more detail later, but the very basics are best learned from introductory material to differential geometry.

\subsection{Algebraic constructions on the tangent bundle}
\label{sec:alg-TM}

All of the tangent spaces of a manifold together make up the tangent bundle.
That is, one can define the tangent bundle of our smooth manifold $M$ to be the disjoint union
\begin{equation}
TM
=
\coprod_{x\in M}T_xM.
\end{equation}
This is a union of vector spaces, and many operations are done tangent space by tangent space.\footnote{The tangent bundle is also a smooth manifold itself, and we shall make heavy use of that later on. But for now it is merely a collection of tangent spaces. Treating it as a manifold opens new doors, but we will not open them yet.}

In general, a bundle is a disjoint union of spaces of some kind attached to each point.
(The tangent bundle is a union of tangent spaces.)
These spaces, called the fibers of the bundle, are isomorphic to each other but not necessarily in a canonical way.
(Since $T_xM\cong\R^n$ for all $x\in M$, the tangent spaces are indeed isomorphic, but not canonically.)

A section of the tangent bundle $TM$ is a map $W\colon M\to TM$ so that $W(x)\in T_xM$ for all $x\in M$.
A section of the tangent bundle is called a vector field.
The section of any other bundle is defined in a similar fashion.
We will define later what smoothness of a section means.
This will be done twice, in local coordinates (section~\ref{sec:tensor-in-coordinates}) and in an invariant fashion (section~\ref{sec:TM}).

Any vector space operation can be perform for the tangent bundle (or any vector bundle for that matter).
For example the dual of the tangent bundle is the cotangent bundle, where the dual is taken fiber by fiber.
The cotangent bundle $T^*M$ is the disjoint union of the cotangent spaces $T_x^*M$.

Similarly, one can take the tensor product $TM\otimes TM$, which is a bundle whose fiber at $x$ is $T_xM\otimes T_xM$.
Tensor products of the tangent and cotangent bundles give rise to many of the bundles one encounters in differential geometry.
For example, the Riemann curvature tensor $R$ is a section of the bundle $TM\otimes T^*M\otimes T^*M\otimes T^*M$.
In other words, for any $x\in M$ we have a multilinear map
\begin{equation}
R(x)
\colon
T_x^*M\times T_xM\times T_xM\times T_xM
\to
\R.
\end{equation}
It is a $1$-contravariant and $3$-covariant tensor field, also called a tensor field of type $(1,3)$.

A vector field is a tensor field of type $(1,0)$ and covectors have type $(0,1)$.
A scalar has type $(0,0)$.

For another example of a tensor field, recall that a linear maps $T_xM\to T_xM$ can be thought of as elements of the tensor product $T_xM\otimes T_x^*M$.
The bundle with these fibers is $TM\otimes T^*M$.
Sections of this bundle are ``matrix fields'' in the sense that at each point $x\in M$ it provides a linear map $T_xM\to T_xM$.
These are tensor fields of type $(1,1)$.
(This is the endomorphism bundle of $TM$, so called because the value of a section at each point is an endomorphism of the relevant tangent space.)

Tensor products can be understood as spaces of multilinear maps.
First, the dual of a real vector space $E$ is the space of linear maps $E\to\R$.
We can write $E^*=ML(E;\R)$, so it is a multilinear map of one variable --- which is a complicated way to say ``linear''.
We also have $E=ML(E^*;\R)$ using the natural identification $E=(E^*)^*$ of finite-dimensional spaces.
Now we can proceed to tensor products.
We have
$E^*\otimes E^*=ML(E\times E;\R)$
and
$E\otimes E\otimes E^*=ML(E^*\times E^*\times E;\R)$.
Using associativity of tensor products we can also see $E^*\otimes E$ as $ML(E;E)$, and this particular interpretation is studied in exercise~\ref{ex:linear-maps-tensor}.
This allows us to see the Riemann curvature tensor as a multilinear map $(T_xM)^3\to T_xM$.

\begin{ex}
\label{ex:linear-maps-tensor}
Let $E,F$ be two finite-dimensional real vector spaces.
There is a natural mapping $\Phi$ from the space $L(E;F)$ of linear maps $E\to F$ to the tensor product $F\otimes E^*$.
Describe this map in formulas (either for itself or an inverse) or in words or in pictures --- or a combination thereof.
\end{ex}

The idea of bundles is necessarily a little vague here as our focus is elsewhere.
The hope is that these first impressions make it easier to pick up ideas along the way and make the reader motivated and well equipped to treat general bundles later on.
We will return to the structure of bundles in section~\ref{sec:TM}.

\subsection{Coordinate representations of tensor fields}
\label{sec:tensor-in-coordinates}

Consider now a single coordinate patch $U\subset M$.
Identifying $U$ with $\phi(U)\subset\R^n$, we can use Euclidean coordinates\footnote{The index is up. This is just a convention, but life is much easier when one sticks to it.} $x^i$ on this subset of $M$.
Let us consider the tangent and cotangent spaces at a point $x\in U$.
Both can be identified with $\R^n$, but it is good to choose a specific identification.

A natural basis for the Euclidean space $\R^n$ consists of the standard unit vectors.
However, when considering tangent vectors as derivations (first order differential operators), it is most natural to let the basis vectors be\footnote{When we differentiate with respect to something that has an upper index, we get a lower index. In time this hopefully makes sense.}
\begin{equation}
\partial_i
\coloneqq
\left.
\pDer{x^i}
\right|_{x}
\in T_xM
.
\end{equation}
Evaluation at the point $x$ and indeed the dependence on $x$ is left implicit in the notation $\partial_i$.
The notation would quickly become unwieldy with everything spelled out, which is why we have chosen to abbreviate the notation of the basis vectors.

The corresponding dual basis consists of the vectors $\der x^i\in T_x^*M$.
Just as in regular linear algebra, the dual basis is defined by
\begin{equation}
\der x^i(\partial_j)
=
\delta^i_j.
\end{equation}
The Kronecker delta $\delta^i_j$ tends to have one index up and another one down.
In fact, the $i$th component of the local coordinates $\phi\colon U\to\R^n$ can be seen as a map $x^i\colon U\to\R$, and the differential of this map $\der x^i$ is the dual basis element.
This justifies the notation.

A vector $W\in T_xM$ and a covector $\alpha\in T_x^*M$ can now be expressed in these bases:
\begin{equation}
\begin{split}
W
&=
W^i\partial_i,
\qquad\text{and}\\
\alpha
&=
\alpha_i\der x^i.
\end{split}
\end{equation}
Observe that the basis and the components have indices in the opposite places.

Here we have for the first time employed the Einstein summation convention:
\begin{equation}
\begin{split}
W^i\partial_i
&\coloneqq
\sum_{i=1}^n W^i\partial_i,
\qquad\text{and}\\
\alpha_i\der x^i
&\coloneqq
\sum_{i=1}^n \alpha_i\der x^i.
\end{split}
\end{equation}
That is, when an index appears once up and once down, all possible values are summed over.
If an index appears more than twice or both occurrences are up or both down, there is an issue.\footnote{This is a non-issue in Euclidean geometry.}

\begin{iex}
Show that
\begin{equation}
W^i
=
\der x^i(W)
\end{equation}
and
\begin{equation}
\alpha_i
=
\alpha(\partial_i).
\end{equation}
This gives us a way to find the components of a vector or a covector in a given basis.

As often, dependence on $x$ was left implicit.
\end{iex}

These basis elements on the tangent and cotangent spaces are crucial for building the smooth structure of the tangent and cotangent bundles in section~\ref{sec:TM}.

Consider then a tensor field $a$ of type $(1,1)$.
As discussed in section~\ref{sec:alg-TM}, $a(x)\colon T_xM\to T_xM$ is a linear map.
As any linear map, $a(x)$ can be expressed as a matrix once a basis is given.
Indeed,
\begin{equation}
a(x)
=
a^i_j(x)\partial_i\der x^j.
\end{equation}
The component $a^i_j$ describes how the $j$th component of the input contributes to the $i$th component of the output.
The component can be extracted from $a(x)$ using
\begin{equation}
\label{eq:(1,1)-component}
a^i_j(x)
=
\der x^i(a(x)\partial_j).
\end{equation}
The general method is the same: operate with the tensor field on the basis vector field(s) and then use the basis covector field(s) to evaluate the component(s).

Smoothness of a tensor field means that all component functions are smooth.
Given some local coordinates, each component of a tensor field is a real-valued function.
The derivative of the component $a^i_j$ with respect to the coordinate $x^k$ is denoted by $a^i_{j,k}$.
Such derivatives do not behave well enough under changes of coordinates, so the coordinate derivatives are not generally the components of a tensor field.

\begin{ex}
Find the components $R^i_{\phantom{i}jkl}$ of a type $(1,3)$ tensor field $R$ using the basis vectors and covectors.
\end{ex}

As we only use a single coordinate system, we need not study how the tensor fields transform when coordinates are changed.

\subsection{A new look at Euclidean linear algebra}

Consider the manifold $M=\R^n$ and in particular its tangent space $T_0M\cong\R^n$.
The basis vectors are given by $e_1=(1,0,\dots,0)$ and the other standard basis vectors $e_i$.
In our Riemannian notation $e_i=\partial_i$.
A vector is written in terms of the basis as $V=V^ie_i$.

It is natural to think of a vector as a column vector.
A row vector corresponds to a covector, $\alpha=\alpha_ie^i$, where $e^i$ are the dual basis vectors to $e_i$.
There is a natural identification of the two bases, given by
\begin{equation}
e^i(W)
=
\ip{e_i}{W}.
\end{equation}
If we map $e_i\mapsto e^i$ and extend linearly, we get a linear map $(\R^n)^*\to\R^n$.
This identification is based on the inner product.
In general, inner products are a way to identify a space with its dual.

The $i$th component of a vector $W$ is found by
\begin{equation}
W^i
=
e^i(W)
=
\ip{e_i}{W}
\end{equation}
as familiar.

\begin{iex}
Given a linear map $L\colon\R^n\to\R^m$, how can you find its matrix elements with respect to some bases on the two spaces?
Compare to~\eqref{eq:(1,1)-component}.
\end{iex}

By the identification of the bases we can identify column vectors with row vectors.
This corresponds exactly to transposition.
The duality pairing $\alpha(W)$ is just the matrix product of a row vector and a column vector.
The inner product of two column vectors can be obtained by transposing one of them and then multiplying as matrices.
The concept of transpose is based on the inner product and changes if the inner product is changed.
And we will change it.

\subsection{Riemannian metric}

A Riemannian metric is a smooth tensor field $g$ of type $(0,2)$ that satisfies a positivity condition and a symmetry condition.
As a tensor field of this type, $g(x)$ is a bilinear map $T_xM\times T_xM\to\R$.
The positivity condition is that
\begin{equation}
\label{eq:g-positive}
g(x)(v,v)>0
\end{equation}
whenever $v\in T_xM$ is non-zero.
The symmetry condition is that
\begin{equation}
\label{eq:g-symmetric}
g(x)(v,w)
=
g(x)(w,v)
\end{equation}
for all $v,w\in T_xM$.
This gives rise to a rich geometric structure.

The convention in the sequel is as follows:
$M$ is always a smooth manifold of dimension $n$, and it has a fixed Riemannian metric $g$.
In other words, $(M,g)$ is a Riemannian manifold.
We assume $M$ to be connected.\footnote{If $M$ is disconnected, the different connected components have completely independent lives. We lose awkward situations but no generality in assuming connectedness.}
Unless otherwise mentioned, we will be working in a single coordinate chart so as to avoid unnecessary complications.

\qa

\section{Distance and geodesics}

\subsection{An inner product}
\label{sec:music}

A Riemannian metric gives an inner product on the tangent space.
Namely, the inner product of two vectors $v,w\in T_xM$ is given simply by
\begin{equation}
\ip{v}{w}
\coloneqq
g(v,w).
\end{equation}
We will often leave the dependence of the metric tensor on the base point $x$ implicit.

\begin{ex}
Expand objects in terms of their components and show that $\ip{v}{w}=g_{ij}(x)v^iw^j$.
\end{ex}

As described in the Euclidean setting, an inner product gives a canonical way to identify vectors with covectors.
In fact, one can consider $g$ as a linear map $T_xM\to T_x^*M$ given by
\begin{equation}
v
\mapsto
g(v,\dum).
\end{equation}
Written in terms of components, the vector with components $v^i$ is mapped to the covector with components $g_{ij}v^j$.
This covector is denoted by $v^\flat$ and called ``$v$ flat''.

\begin{iex}
Show that the map $v\mapsto v^\flat$ is bijective.
You will need the positivity condition \eqref{eq:g-positive}.
\end{iex}

The inverse of the map $v\mapsto v^\flat$ maps a covector $\alpha$ to the vector $\alpha^\sharp$, called ``$\alpha$ sharp''.
These are the musical isomorphisms and they satisfy $v=(v^\flat)^\sharp$ and $\alpha=(\alpha^\sharp)^\flat$.

Given the canonical bases on $T_xM$ and $T_x^*M$, the matrix of the ``flat map'' is $g_{ij}$ itself.
The matrix of the inverse map, the ``sharp map'', is denoted by $g^{ij}$ and is the inverse of this matrix --- it satisfies $g^{ij}g_{jk}=\delta^i_k$.
Invariantly, this can be denoted as $g^{-1}$.

\begin{ex}
Show that $g^{ij}((v^\flat)_i,(w^\flat)_j)=\ip{v}{w}$.
\end{ex}

\begin{ex}
Show that $g^{ij}$ defines an inner product on $T_x^*M$ and the musical isomorphisms preserve the inner product.
\end{ex}

The inner products give us natural definitions of norms for the tangent and cotangent spaces:
$\abs{v}=\ip{v}{v}^{1/2}$ and $\abs{\alpha}=\ip{\alpha}{\alpha}^{1/2}$ using the relevant inner products.
The musical isomorphisms are isometries.
The (co)tangent space $T_x^{(*)}M$ is also isometric to $\R^n$, as are all $n$-dimensional real inner product spaces.

Due to the way the musical isomorphisms work in coordinates --- $(v^\flat)_i=g_{ij}v^j$ and $(\alpha^\sharp)^i=g^{ij}\alpha_j$ --- they are sometimes called lowering and raising indices.

Recall that the differential $\der f$ of a scalar function $f\colon M\to\R$ is a covector field.
The corresponding vector field is called its gradient: $\nabla f=(\der f)^\sharp$.

One would obtain much more general structures by taking a norm on the tangent space that does not correspond to an inner product.
This would lead to Finsler geometry.

\subsection{On computations in local coordinates}

Let us consider the flat map as an example.
If $v$ is a vector field, then $\alpha=v^\flat$ is given in local coordinates as $\alpha_i=g_{ij}v^j$.
Including the variable and the sum explicitly, this means
\begin{equation}
\label{eq:flat-in-coordinates}
\alpha_i(x)=\sum_jg_{ij}(x)v^j(x).
\end{equation}
If we need to compute a derivative like $\partial_k\alpha_{i}$ in these local coordinates, it can be helpful to look at~\eqref{eq:flat-in-coordinates}.
Each $\alpha_i(x)$ is a real valued function of $x\in\R^n$ (or rather only in the set $\phi(U)\subset\R^n$), and so are $g_{ij}(x)$ and $v^j(x)$.
Each component is just a real-valued function --- coordinate expressions are almost always expressions containing sums and products of real numbers nothing more elaborate.
When you differentiate, the normal product rule applies without any changes.

If a specific computation confuses you, please bring it up in the end-of-section exercise or otherwise.
Future versions of the notes benefit from all feedback.

\subsection{Length of curve}

Recall that the length of a smooth curve $\gamma\colon[a,b]\to\R^n$ is defined by
\begin{equation}
\ell(\gamma)
=
\int_a^b\abs{\dot\gamma(t)}\dd t.
\end{equation}
We define the length of a smooth curve $\gamma\colon[a,b]\to M$ by the same formula.

To properly do so, we must know what $\dot\gamma(t)$ is.
As discussed in section~\ref{sec:curve-vector-differential}, velocities of curves are one way to define tangent vectors in the first place, so $\dot\gamma(t)$ should be an element of $T_{\gamma(t)}M$.

In local coordinates one can write $\dot\gamma(t)=\dot\gamma^i(t)\partial_i$.
The length of $\dot\gamma(t)$ is given by the metric tensor.
Notice how the norm used to measure the length of $\dot\gamma(t)$ is different for different values of $t$.

Everything is defined so that the length of a curve is independent of the choice of coordinates and parametrization.

\subsection{Distance between points}

Let $p,q\in M$ be any two points.
As $M$ is connected, there is a smooth path between the two points.
We define the distance between them to be
\begin{equation}
d(p,q)
=
\inf
\{\ell(\gamma);\gamma\colon[0,1]\to M,\gamma(0)=p,\gamma(1)=q\}.
\end{equation}
It is typical to choose the curve family so that $\gamma$ is piecewise smooth, but smooth will work as well.

\begin{ex}
\label{ex:cut-corner}
Explain with a picture or maybe even a proof why minimizing length of piecewise smooth curves will lead to the same infimum as minimizing over smooth curves.
\end{ex}

This concept of distance defines a metric in the sense of metric spaces.
But we will restrict the word ``metric'' to the metric tensor and call this $d$ the distance.

\begin{ex}
Give an example of two points in a Euclidean domain where a minimizing curve does not exist within the domain.
The same issue can occur on manifolds, so existence of minimizers requires assumptions.
(A local result is given in exercise~\ref{ex:local-geodesic-short}.)
\end{ex}

\begin{proposition}
The manifold $M$ with the distance $d$ satisfies all the axioms of a metric space.
Its topology coincides with that of the topological manifold $M$.
\end{proposition}

The proof of coincidence of the two topologies can be found in many introductory treatises of Riemannian geometry.
It suffices to prove such equivalence within a chart, and that follows from the distance being bi-Lipschitz to the underlying Euclidean metric where the coordinates live.
See exercise~\ref{ex:d=0}.

\begin{iex}
Explain why $d$ is symmetric and satisfies the triangle inequality.
\end{iex}

\begin{ex}
\label{ex:d=0}
Show that if $d(x,y)=0$, then $x=y$.
You can work in local coordinates near $x$.
Argue by continuity that $C^{-1}\abs{v}_{\R^n}\leq\abs{v}\leq C\abs{v}_{\R^n}$ for all $v\in TU$ for a small neighborhood $U$ of $x$ (in those local coordinates) and for some constant $C>1$.
Using that estimate find a lower bound on the length of any smooth curve joining $x$ and $y$.
\end{ex}

\subsection{First variation of length}

We want to find the shortest curve between two points.
We do so using smooth calculus of variations.
The aim is to find the Euler--Lagrange equation and later show that its solutions are actually minimal.

Let $\Gamma\colon[0,1]\times(-\eps,\eps)\to M$ be a smooth map.
We understand $\Gamma(t,s)$ to be a family of curves so that each $\Gamma(\dum,s)$ is a curve.
We want to differentiate
\begin{equation}
\ell(\Gamma(\dum,s))
=
\int_0^1\abs{\partial_t\Gamma(t,s)}\dd t
\end{equation}
at $s=0$.
Let us work in local coordinates again.

\begin{ex}
\label{ex:d-abs-G}
Show that
\begin{equation}
\begin{split}
&\partial_s
[
g(\Gamma(t,s))(\partial_t\Gamma(t,s),\partial_t\Gamma(t,s))
]^{1/2}
\\&=
\frac1{2\abs{\partial_t\Gamma}}
(
g_{ij,k}(\Gamma)\partial_s\Gamma^k\partial_t\Gamma^i\partial_t\Gamma^k
+
2g_{ij}(\Gamma)\partial_t\Gamma^i\partial_t\partial_s\Gamma^j
),
\end{split}
\end{equation}
where the argument $(t,s)$ of $\Gamma$ has been left out for clarity.
Here we used the derivative notation $g_{ij,k}\coloneqq \partial_kg_{ij}$ again.
\end{ex}

We are now ready to compute the variation of length of a family of constant speed curves from $p$ to $q$.
Recall that reparametrization does not change length so we are free to do so.
This reparametrization preserves smoothness as long as $\partial_t\Gamma\neq0$.

\begin{proposition}
\label{prop:1st-variation}
Let $\Gamma\colon[0,1]\times(-\eps,\eps)\to M$ be a smooth map so that
\begin{itemize}
\item
$\abs{\partial_t\Gamma(t,0)}$ is constant,
\item
$\Gamma(0,s)=p$ for all $s$, and
\item
$\Gamma(1,s)=q$ for all $s$.
\end{itemize}
Denoting\footnote{Notice that the second order derivatives are computed in local coordinates. We do not yet have proper tools to handle them invariantly. We will later, and the formula simplifies considerably; see~\eqref{eq:1st-variation-invariantly}.}
$\gamma(t)\coloneqq\Gamma(t,0)$,
$\dot\gamma(t)\coloneqq\partial_t\Gamma(t,0)$, 
$\ddot\gamma(t)\coloneqq\partial_t^2\Gamma(t,0)$,
and
$V(t)\coloneqq\partial_s\Gamma(t,s)$,
we have
\begin{equation}
\left.
\partial_s
\ell(\Gamma(\dum,s))
\right|_{s=0}
=
\int_0^1
\frac1{\abs{\dot\gamma}}
V^k
\left[
\frac12g_{ij,k}\dot\gamma^i\dot\gamma^j
-g_{ik,j}\dot\gamma^i\dot\gamma^j
-g_{ik}\ddot\gamma^i
\right]
\dd t.
\end{equation}
\end{proposition}

\begin{proof}
Exercise~\ref{ex:d-abs-G} shows that the derivative in question is
\begin{equation}
\int_0^1
\frac1{\abs{\dot\gamma}}
\left[
\frac12g_{ij,k}V^k\dot\gamma^i\dot\gamma^j
+g_{ij}\dot\gamma^i\partial_tV^k
\right]
\dd t
.
\end{equation}
We integrate by parts in the second term to take the $\partial_t$ away from $V^k$.
As $\abs{\dot\gamma}$ is independent of $t$ and $V(0)=0$ and $V(1)=0$, we find the desired form of the derivative.
\end{proof}

If the curve $\gamma(t)=\Gamma(t,0)$ is to be minimizing within this family, this derivative should vanish for any variation field $V(t)$.
This inspires us to define a geodesic to be a constant speed curve which satisfies
\begin{equation}
\label{eq:ge-1}
\frac12g_{ij,k}\dot\gamma^i\dot\gamma^j
-g_{ik,j}\dot\gamma^i\dot\gamma^j
-g_{ik}\ddot\gamma^i
=
0.
\end{equation}
In fact, it turns out that solutions to this equation automatically have constant speed; see corollary~\ref{cor:geod-constant-speed}.

It is important to read this result the right way.
We have shown that a smooth minimizing curve is a geodesic --- which means satisfying the geodesic equation.
We have not shown that minimizers exist or that they are smooth.
That will come much later.

\subsection{The Christoffel symbol}

The Christoffel symbol is a gadget that looks a bit like a type $(1,2)$ tensor field --- but is not due to the derivatives --- is defined in local coordinates as
\begin{equation}
\Cs ijk
=
\frac12
g^{il}
(g_{lj,k}+g_{lk,j}-g_{jk,l}).
\end{equation}
This symbol will appear often in coordinate formulas.
We immediately point out the symmetry property:
\begin{equation}
\Cs ijk
=
\Cs ikj.
\end{equation}

\begin{ex}
Show that equation~\eqref{eq:ge-1} is equivalent with
\begin{equation}
\label{eq:ge-2}
\ddot\gamma^i+\Cs ijk\dot\gamma^j\dot\gamma^k=0.
\end{equation}
This is called the geodesic equation.
\end{ex}

Observe that in Euclidean geometry where $g_{ij}(x)$ is independent of the base point $x$ the Christoffel symbol vanishes.
On more general manifolds its appearance is inevitable, but it will disappear in an invariant treatment.
In fact, it is what helps make derivatives invariant.

If one does a non-inertial change of coordinates in classical mechanics, one introduces pseudoforces such as the centrifugal force.
The Christoffel symbol can be seen as a pseudoforce term: a geodesic wouold continue at constant speed ($\ddot\gamma^i=0$) without its effect.
A typical Riemannian manifold does not admit ``inertial coordinates'' and the Christoffel symbol appears.
(They can be made vanish at a single point as in exercise~\ref{ex:CS-normal}.)
We will also find an invariant form of the geodesic equation which in a sense remove the pseudoforces from the picture.

\subsection{The geodesic equation}

A solution to the geodesic equation is called a geodesic.
It follows from standard ODE theory that for any $x\in M$ and any $v\in T_xM$ there is a unique geodesic $\gamma\colon(-\eps,\eps)\to M$ so that $\gamma(0)=x$ and $\dot\gamma(0)=v$.
Existence for long times is not guaranteed unless additional structure is introduced.\footnote{If you are interested, look up geodesic completeness and the Hopf--Rinow theorem.}

\begin{ex}
\label{ex:ge-e/u}
Use this result:
\begin{quote}
If $F\colon\R^N\to\R^N$ is Lipschitz, then the ODE $u'(t)=F(u(t))$ has a unique local $C^1$ solution for any given initial conditions $u(0)=u_0\in\R^N$.
\end{quote}
Prove the local existence and uniqueness result for the geodesic equation.
\end{ex}

\begin{ex}
\label{ex:geod-smooth}
Consider the quoted ODE result of the previous exercise.
Show that if $F$ is smooth, so is $u$.
This proves that geodesics are necessarily smooth.
\end{ex}

We stress that we define a geodesic to be a solution to the geodesic equation.
(The equation will have a couple of equivalent forms.)
That geodesics actually minimize length is not entirely trivial, so we shall prove it later.

Existence of minimizers has not been established yet either.
The \AA{} theorem can be used to produce a minimizer, but often of very low regularity.
We will use smooth tools instead.

\qa

\section{Connections and covariant differentiation}

\subsection{Connections in general}

It is not always obvious what differentiation should mean.
For a function $M\to\R$ we can assign a differential as a covector (a cotangent vector).
The derivative of a function $\R\to M$ (a curve) can be treated as a vector (a tangent vector).
These behave well under changes of coordinates, and indeed these derivatives can be used to define vectors and covectors in the first place.

Differentiation of vectors does not make sense equally simply.
Consider a vector field $W(x)$.
What does it mean for $W(x)$ to stay constant as $x$ changes?
Each $W(x)$ belongs to $T_xM$, so the underlying space changes.
We need a way to compare tangent vectors on nearby tangent spaces.

The same issue arises with all kinds of bundles.
The analogue of a vector field or a tensor field on a general bundle is called a section.
A consistent method of differentiating a section of a bundle is called a connection.
A connection for vector fields is called an affine connection.

\begin{definition}
An affine connection $\nabla$ on a manifold $M$ is a bilinear map that maps a pair $(X,Y)$ of vector fields into a vector field $\nabla_XY$ so that the following conditions hold for any smooth function $f\colon M\to\R$:
\begin{itemize}
\item
$\nabla_{f X}Y=f\nabla_XY$
\item
$\nabla_X(fY)=f\nabla_XY+X(f)Y$.
\end{itemize}
\end{definition}

These conditions describe the linearity when the vector fields are multiplied by a scalar function instead of a single number.
(A reader familiar with more abstract linear algebra may enjoy the observation that vector fields constitute a module over the ring $C^\infty(M;\R)$ of smooth functions.)

One can read $\nabla_XY$ as ``the derivative of the vector field $Y$ in the direction of the vector field $X$''.
If $X,Y\colon\R^n\to\R^n$ are smooth vector fields, the standard affine connection of Euclidean geometry is given by
\begin{equation}
(\nabla_XY)^j
=
X^i Y^j_{\phantom{j},i}
\end{equation}
using the usual coordinates of $\R^n$.

\begin{ex}
\label{ex:Eucl-connection}
Show that the Euclidean connection defined above is indeed an affine connection on the space $\R^n$.
You will see the familiar Leibniz rule take a new form.
\end{ex}

\subsection{The \LC{} connection}

There are a great many connections on a smooth manifold.
The definition of a connection had nothing to do with a metric tensor.
We would of course like the concept of differentiation to be somehow compatible with the metric.

Before giving a definition of such a good connection, we need to recall the concept of a commutator.
The commutator of two linear operators $A$ and $B$ is $[A,B]\coloneqq AB-BA$.
The commutator of two differential operators of orders $k$ and $m$ is a differential operator of order $k+m-1$.
In particular, the commutator of two derivations (first order differential operators) is another derivation.

Therefore the commutator of two vector fields is a vector field.
One can define it explicitly as $[X,Y]f=X(Yf)-Y(Xf)$, where the vector fields turn scalar fields to scalar fields.

\begin{ex}
\label{ex:commutator-coordinates}
Let $X$ and $Y$ be two vector fields.
Show that their commutator is a vector field has the components
\begin{equation}
[X,Y]^i
=
X^jY^i_{\phantom{i},j}
-
Y^jX^i_{\phantom{i},j}.
\end{equation}
This shows that the commutator as differential operator has only first order terms and is therefore a vector field.
\end{ex}

\begin{definition}
An affine connection $\nabla$ on a Riemannian manifold $(M,g)$ is called a symmetric\footnote{This adjective was missing from the first version. Usually the metric condition only refers to the first property. Symmetry or lack of torsion is the second one.} metric connection if
\begin{itemize}
\item
$Xg(Y,Z)=g(\nabla_XY,Z)+g(Y,\nabla_XZ)$ and
\item
$\nabla_XY-\nabla_YX=[X,Y]$.
\end{itemize}
\end{definition}

The first condition is a Leibniz rule for the inner product; a Leibniz rule of a different nature was included in the definition of an affine connection.
The point is that although $g(Y,Z)$ contains three tensor fields (the metric tensor and the two vector fields), there are no derivatives of the metric tensor in the formula.
We will see in a moment that indeed the covariant derivative of the metric tensor is zero.

The second condition has nothing to do with the metric.
Instead, it states that something called the torsion of the connection vanishes.
The torsion measures how the tangent spaces twist as one moves from one base point to another.
A rough heuristic way to see the condition is that we want the tangent spaces to rotate but not twist.

Every Riemannian manifold has a unique symmetric metric connection\footnote{We will not prove this theorem.}, and it is called the \LC{} connection\footnote{This is named after Tullio \LC, a single person. Therefore the connection is called the \LC{} connection instead of the Levi--Civita connection.}.
The connection is defined so that for two vector fields $X(x)$ and $Y(x)$ we have
\begin{equation}
(\nabla_XY)^i
=
X^jY^i_{\phantom{i},j}
+
\Cs ijkX^jY^k.
\end{equation}
It is not apparent as we have not bothered with changing coordinates, but $\nabla_XY$ is indeed a valid vector field.

\begin{ex}
Prove that the \LC{} connection is an affine connection.
\end{ex}

\begin{ex}
Prove that the \LC{} connection is a symmetric metric connection.
\end{ex}

\subsection{Covariant differentiation}

We would like to be able to differentiate tensor fields of all kinds.
We continue to use $\nabla$ for this purpose, but in the sequel we will rarely need to differentiate very complicated tensor fields.
For any tensor field $T$ of any type $(k,l)$ and a vector field $X$, we would like to be able to compute $\nabla_XT$, the covariant derivative of $T$ in the direction of $X$.
This should all be defined so that $\nabla_XT$ is also a tensor field of type $(k,l)$ and thus behaves under coordinate changes as a tensor field should.
As $\nabla_XT$ is linear in $X$, we may regard $\nabla T$ as a tensor field of type $(k,l+1)$.

Any affine connection gives rise to such a way, as long as we require the following:
\begin{itemize}
\item
On scalar functions the covariant derivative is simply the derivative by a vector field: $\nabla_Xf=Xf$.
\item
On vector fields we have the original connection.
\item
Tensor products satisfy the Leibniz rule
\begin{equation}
\nabla_X(T\otimes R)
=
\nabla_XT\otimes R
+
T\otimes\nabla_X R.
\end{equation}
\item
The covariant derivative commutes with any contraction or trace.\footnote{We have not introduced this concept nor will we use it explicitly. This statement is here for completeness.}
\end{itemize}
The \LC{} connection has an additional property that neatly describes the metric compatibility:
\begin{equation}
\label{eq:Dg=0}
\nabla g=0.
\end{equation}
That is, the concept of differentiation is defined so that the metric tensor $g$ is ``constant''.
(A more appropriate technical term is ``parallel''.)

Recall the differential of a smooth function $f\colon M\to\R$ as a cotangent vector.
If tangent vectors are seen as derivations, then $\der f(X)=Xf$.
The covariant derivative of $f$ in the direction of a vector field $X$ was just defined so that $\nabla_Xf=Xf$.
Therefore $\der f(X)=\nabla_Xf$.
As $f$ is a tensor field of type $(0,0)$, its covariant derivative $\nabla f$ as defined above is a tensor field of type $(0,1)$ --- a covector field.
This covector field should satisfy $(\nabla f)(X)=\nabla_Xf$ for any vector field $X$, so we conclude that the covariant derivative $\nabla f$ is exactly $\der f$, the differential of $f$.

We mentioned in section~\ref{sec:music} that the gradient of a function $f$ can be defined as the vector field $(\der f)^\sharp$ corresponding to the covector field $\der f$.
The gradient vector field is usually denoted by $\nabla f$.
This is confusing with the covariant derivative, but fortunately the musical isomorphisms send the two objects denoted by $\nabla f$ to each other in a canonical way.
We shall denote the differential (and therefore the covariant derivative) of a scalar function by $\der f$, although some more consistency with other covariant derivatives would be achieved by different notation.

To get all of this on a more concrete footing, let us see how to covariantly differentiate a tensor field given in terms of components in some local coordinates.
For a vector field $Y$ we have directly the formula of the \LC{} connection:
\begin{equation}
\label{eq:cd-vf}
(\nabla_XY)^i
=
X^jY^i_{\phantom{i},j}+\Cs ijk X^jY^k.
\end{equation}

\begin{iex}
The coordinate vector fields $\partial_i$ are of course valid vector fields within their coordinate patch.
What is $\der x^i(\nabla_{\partial_j}\partial_k)$?
Describe in words what it means and give a formula.
\end{iex}

We would then like to find a similar expression for $(\nabla_X\alpha)_ i$ for a covector field $\alpha$.

\begin{ex}
Starting with the covariant derivative of a vector field and the Leibniz rule
\begin{equation}
X(\alpha(Y))
=
(\nabla_X\alpha)(Y)
+
\alpha(\nabla_XY)
\end{equation}
(which follows from the tensor product rule and the trace rule stipulated above),
show that
\begin{equation}
(\nabla_X\alpha)_ i
=
X^j\alpha_{i,j}-\Cs jik X^j\alpha_k.
\end{equation}
This is the covariant differentiation rule of covector fields.
\end{ex}

A tensor field of any type can be differentiated in a similar fashion.
For every upper index we add a term like we had for vectors and for all lower indices we add a term like for covectors.
For example, the covariant derivative of a type $(1,1)$ tensor $a$ is given by
\begin{equation}
(\nabla_Xa)^i_j
=
X^ka^i_{j,k}
+
\Cs ikl a^k_jX^l
-
\Cs klj a^i_kX^l.
\end{equation}

\begin{iex}
What is the coordinate expression for $\nabla_Xg$ for a type $(0,2)$-tensor $g$?
\end{iex}

\begin{ex}
Show directly using the formula of the previous exercise that $\nabla_Xg=0$ when $g$ is the metric tensor.
\end{ex}

\subsection{On notation}

There are various different notations in use in differential geometry.
Different conventions are convenient in different situations, and the different ways to express the same thing offer new points of view.

For example, the derivative of a scalar function $f\colon M\to\R$ in the directions of a vector field $X$ on $M$ can be written as
\begin{equation}
\nabla_Xf
=
Xf
=
\der f(X)
=
\ip{\nabla f}{X}
=
\ip{\der f}{X},
\end{equation}
where the last inner product is the duality pairing between $T_xM$ and $T_x^*M$.
And this list is not exhaustive; for example, in some cases it is convenient to denote $\der f$ by $f^*$ and call it the pushforward. 
The same object can also be expressed in local coordinates as $X^i\partial_if$ or $X^if_{,i}$.

Componentwise notations also vary somewhat.
It is customary to have all indices ``in sequence'' whether up or down, so that a gap is left where an index is in the other place.
This means writing, for example, $T^{i\phantom{j}k}_{\phantom{i}j\phantom{k}l}$ instead of $T^{ik}_{jl}$.
This only really becomes crucial when raising and lowering indices by the musical isomorphisms (which extends to tensor fields), so this convention is not always followed.

In Riemannian geometry one can naturally identify tangent vectors with cotangent vectors using the musical isomorphisms.
It is possible to leave the isomorphisms implicit and just let indices wander around freely.
However, it is instructive to keep track at least of vectors and covectors.
There are situations where a Riemannian metric is not available for music and often the natural kind of object sits most comfortably in any computation.

We have seen two types of differentiation.
The simplest kind is coordinate differentiation.
For example, the coordinate derivative of a vector field $V^i$ would be
\begin{equation}
\pDer{x^j}V^i(x)
=
\partial_j V^i
=
V^i_{\phantom{i},j}.
\end{equation}
This is an object with one index up and another down, but it is not a tensor field of type $(1,1)$ due to the issue of coordinate invariance which we have kept mysterious.

The covariant derivative of $V$ in the direction of the vector field $Y$ is $\nabla_YV$.
Its components are given by~\eqref{eq:cd-vf}.
One can write this in local coordinates as
\begin{equation}
(\nabla_YV)^i
=
Y^jV^i_{\phantom{i};j}
\end{equation}
by introducing the notation
\begin{equation}
V^i_{\phantom{i};j}
=
V^i_{\phantom{i},j}
+
\Cs ijk V^k.
\end{equation}
These are precisely the components of the $(1,1)$-type tensor field $\nabla V$.
The comma is used for coordinate differentiation and semicolon for covariant differentiation.

The Christoffel symbols are used as correction terms to make differentiation behave well.

\qa

\section{Fields along a curve}

\subsection{Vector fields along a curve}

Let $\gamma\colon I\to M$ be a smooth curve defined on an interval $I\subset\R$.
We would like to give a natural space for the velocity vector $\dot\gamma(t)$ to live in.
Each $\dot\gamma(t)$ is in $T_{\gamma(t)}M$, but this is not a vector field as previously described.
It is only defined on a subset of the manifold, namely the trace $\gamma(I)$.
And what if the curve intersects itself or even stops?

We define a vector field along the curve $\gamma$ to be a smooth map $V\colon I\to TM$ that satisfies $V(t)\in T_{\gamma(t)}M$ for all $t\in I$.
There are two important examples:
\begin{itemize}
\item
$\dot\gamma(t)$ is a vector field along $\gamma$.
\item
If $V$ is a vector field on $M$, then $V(\gamma(t))$ is a vector field along $\gamma$.
\end{itemize}
If $\dot\gamma\neq0$, then at least locally any vector field along $\gamma$ can be extended to its neighborhood and considered like the second example.
But it is best to treat objects so that they require no artificial extensions; a vector field along a curve should only exist on the curve.

It is probably worth pointing out that a vector field along a curve need not point along the curve.
It only has to be defined along the curve.

\subsection{Covariant differentiation along a curve}

In local coordinates we define the covariant derivative of $V(t)$ along $\gamma(t)$ with respect to $t$ to be
\begin{equation}
(D_tV(t))^i
=
\dot V^i(t)
+
\Cs ijkl V^j(t)\dot\gamma(l).
\end{equation}
This is a derivative with respect to the time parameter $t$, but as before, a naive coordinate derivative is invalid.

\begin{ex}
Suppose that $\gamma$ is the integral curve of a vector field $X$ on $M$.
This means that $\dot\gamma(t)=X(\gamma(t))$ for all $t$.
(We will return to integral curves in section~\ref{sec:dyn-sys}.)
Let $V$ be any vector field on $M$.
Show that\footnote{We defined covariant differentiation along a curve so that this holds. There is only one definition that makes this work.}
\begin{equation}
D_tV
=
\nabla_XV.
\end{equation}
Where does this equation make sense?
\end{ex}

The velocity of a curve $\gamma$ is $\dot\gamma$.
Its natural time derivative is $D_t\dot\gamma$, the ``covariant acceleration''.
In Euclidean geometry it makes sense to say that a curve is straight if its acceleration vanishes.
We can now do the same:
we can say that a curve is straight when $D_t\dot\gamma(t)=0$ for all $t$.

\begin{iex}
Show that a smooth curve $\gamma$ is straight if and only if it is a geodesic.
\end{iex}

We have found a familiar fact: The shortest curves are straight.
But, unlike in Euclidean geometry, a straight curve is not necessarily the shortest one between its endpoints.

We have found yet another form of the geodesic equation, this time an invariant one:
\begin{equation}
\label{eq:ge-3}
D_t\dot\gamma(t)
=
0.
\end{equation}
Compare this to the previous versions~\eqref{eq:ge-1} and~\eqref{eq:ge-2}.

The first derivative of the curve $\gamma$ is often denoted by $\dot\gamma$.
Sometimes it is good to write it as $\partial_t\gamma$ for clarity.
And as before, we can define covariant differentiation of the simplest objects to agree with the usual derivative, so that we may well write
\begin{equation}
\dot\gamma
=
\partial_t\gamma
=
D_t\gamma.
\end{equation}
This is only a matter of notation, but its benefit will come clear soon.
The geodesic equation gets yet another form:
\begin{equation}
\label{eq:ge-4}
D_t^2\gamma=0.
\end{equation}
This version is both neat and useful.
We will see it soon in section~\ref{sec:JF} when studying Jacobi fields.

The covariant derivative along a curve is also compatible with the metric as one might expect.
The following two rules establish the natural Leibniz rules for vector fields $V$ and $W$ and a scalar field $f$ along $\gamma$.
(A scalar field along a curve is simply a real-valued function defined on the interval where the curve is parametrized.)
The time derivative of a scalar $f$ could be written as $D_tf$ as well, but $\partial_tf$ highlights that we are only differentiating a number.

\begin{ex}
Show that $D_t(fV)=(\partial_tf)V+fD_tV$.
\end{ex}

\begin{ex}
\label{ex:Dt-ip}
Show that $\partial_t\ip{V}{W}=\ip{D_tV}{W}+\ip{V}{D_tW}$.
\end{ex}

\subsection{Parallel transport}

\begin{definition}
A vector field $V$ along a curve $\gamma$ is said to be parallel if $D_tV=0$.
\end{definition}

A parallel vector field is the closest we can get to a constant vector field.

Any vector at any point along a curve can be parallel transported along it.

\begin{ex}
Let $\gamma\colon I\to\R$ be a curve.
Given any $t_0\in I$ and $V_0\in T_{\gamma(t_0)}M$, show that there is a unique parallel vector field $V$ along $\gamma$ with $V(t_0)=V_0$.

This is what it means to parallel transport $V_0$ from a single tangent space along the curve.
\end{ex}

Beware that parallel transport happens along a curve, not just between two points.
Even if a curve intersects itself, parallel transport around a loop rarely preserves the vector.
But it does preserve something:

\begin{proposition}
\label{prop:par-ip}
If $V$ and $W$ are parallel vector fields along a curve $\gamma$, then their inner product $\ip{V}{W}$ is constant.
In particular, a parallel vector field has constant norm.
\end{proposition}

\begin{proof}
As $D_tV=D_tW=0$, exercise~\ref{ex:Dt-ip} implies that $\partial_t\ip{V}{W}$.
The second claim is found by letting $V=W$.
\end{proof}

\begin{corollary}
\label{cor:geod-constant-speed}
A geodesic has constant speed.
\end{corollary}

\begin{remark}
When we did our calculus of variations to find the geodesic equation, we required that $\abs{\dot\gamma}$ is constant.
It should therefore be no surprise that a solution to the equation has constant speed.
If we are free to reparametrize as we like, geodesics will certainly not be unique anymore.
If we drop constant speed parametrization, we can describe geodesics to be those smooth curves $\gamma\colon I\to M$ for which $\dot\gamma(t)\neq0$ and $D_t\dot\gamma(t)=f(t)\dot\gamma(t)$ for some smooth function $f\colon I\to\R$.
This can be interpreted so that the acceleration of the curve must be along the curve.
This is similar to describing Euclidean geodesics as $\gamma(t)=x+h(t)v$ for a function $h$ with non-vanishing derivative; in its case $f(t)=h''(t)/h'(t)$.
\end{remark}

We have found that a minimizing curve must be a geodesic.
Now we know that geodesics are as straight as a curve on a Riemannian manifold can be and that they have constant speed\footnote{Although the length functional is parametrization independent, we did make use of constant speed parametrization to find the variation of length.}.
What we have not discovered yet is whether a geodesic is always minimizing and whether one always exists between any two points.
We will prove these statements later, but only locally as they are not generally globally true.

\subsection{Orthonormal bases}

The Riemannian metric makes each tangent space $T_xM$ into an inner product space of dimension $n$.
Therefore there is an orthonormal basis $e_1,\dots,e_n$.
As in Euclidean geometry, working within such a basis is convenient.

Now consider a smooth curve $\gamma$ on $M$.
We can take an orthonormal basis in the tangent space at any point and then parallel transport each\footnote{The index of $e_\alpha$ is not a coordinate index, so we try to reduce confusion by using a different kind of letter.} $e_\alpha$ along the curve.
This gives rise to vector fields $e_\alpha(t)$ along $\gamma$.

Such a collection of vectors is called an orthonormal parallel frame along $\gamma$.
It provides a consistent basis throughout the curve.
By proposition~\ref{prop:par-ip} the vectors $e_\alpha(t)\in T_{\gamma(t)}M$ are orthonormal for all values of $t$.

It is common to choose one of the basis vectors to be $\dot\gamma(t)$ itself.
It is indeed parallel and has unit length if $\gamma$ is a unit speed geodesic.
However, for a general curve $\dot\gamma$ is not parallel.

In a parallel frame computations appear more Euclidean.

\begin{ex}
Any vector field $V(t)$ along $\gamma$ can be expressed in the orthonormal parallel frame as
\begin{equation}
V(t)
=
\sum_{\alpha=1}^nV_\alpha(t)e_\alpha(t).
\end{equation}
Show that $V$ is parallel if and only if each $V_\alpha(t)$ is constant.
What is the norm of $V(t)$?
\end{ex}

Parallel frames exist along curves, but not on the whole manifold.
It is extremely rare that there would be even one non-zero vector field in a small open subset of the manifold which would be parallel along all curves.

\begin{ex}
Euclidean geometry is far more rigid than general Riemannian geometry.
Give an example of a non-zero vector field on $\R^n$ which is parallel transported along any curve.

Are there $n$ such vectors that could make an orthonormal frame?

Using local coordinates on any Riemannian manifold $M$ makes $U\subset M$ look Euclidean.
You can then choose a parallel field of this kind in the local coordinates.
Why is it not a parallel field defined in $U\subset M$?
\end{ex}

Given a basis of a vector space, there is a corresponding dual basis on the dual space.
The dual basis of an orthonormal parallel frame is an orthonormal parallel coframe.
The same properties of preserved inner products hold with the dual inner product on $T_x^*M$.

\subsection{The variation field of a family of geodesics}
\label{sec:pre-JF}

We used a family of curves when we studied variations of length.
Let us return to studying such a family $\Gamma(t,s)$.
Such a family appeared in proposition~\ref{prop:1st-variation}.
The proposition can be rephrased using our new tools:

\begin{quote}
Let $\Gamma\colon[0,1]\times(-\eps,\eps)\to M$ be a smooth map for which $\Gamma(0,s)=p$ and $\Gamma(1,s)=q$ for all $s$.
Denote $\gamma(t)=\Gamma(t,0)$ and $V(t)=\partial_s\Gamma(t,0)$.
Then
\begin{equation}
\label{eq:1st-variation-invariantly}
\partial_s\ell(\Gamma(\dum,s))|_{s=0}
=
-
\int_0^1
\frac1{\abs{\dot\gamma(t)}}
\ip{V}{D_t\dot\gamma}
\dd t.
\end{equation}
\end{quote}

In this form it is more transparent that the geodesic equation is $D_t\dot\gamma=0$.

\begin{ex}
\label{ex:1st-variation-sign}
Let us explain the negative sign in~\eqref{eq:1st-variation-invariantly}.
Suppose $\gamma$ is a unit speed curve in $\R^2$.
Draw a picture of a non-geodesic curve $\gamma$ in the plane and draw a nearby shorter curve with the same endpoints.
Draw the variation field $V$ and the second derivative $\ddot\gamma$ in a couple of points along the curve.
Explain the negative sign in the formula based on this example.
\end{ex}

The variation field of the family, $V(t)$, is a vector field along $\gamma$.

Every $\Gamma(\dum,s)$ is assumed to be a geodesic.
We have in fact already used the vector field $V(t)=\partial_s\Gamma(t,s)|_{s=0}$ in our variational calculations.
This is a vector field along the reference geodesic $\gamma=\Gamma(\dum,0)$.
This field describes first order variations of the curve family, and it is far simpler to study the behaviour of this variation vector field than the whole family of geodesics.

The variation field may be extended to all geodesics in the family by letting $V(t,s)=\partial_s\Gamma(t,s)$.
In fact, this is the velocity vector field of the curve $\Gamma(t,\dum)$, where now $t$ is fixed.
It is important to be able to differentiate with respect to both variables $t$ and $s$ --- also covariantly.

Of course one can study variations of any curve family, but more structure emerges when one studies a family of geodesics.
Comparison of nearby geodesics is not trivial; geodesics that start nearby can diverge and later converge and maybe even intersect.
Nothing similar can happen in Euclidean geometry.

\qa

\section{Jacobi fields}
\label{sec:JF}

\subsection{Commutators of covariant derivatives}

Consider two vector fields $X$ and $Y$ and a scalar field $f$ on $M$.
One can differentiate $f$ with $X$ and $Y$ in two different orders.
Their difference is $XYf-YXf=[X,Y]f$.
This is the commutator of two vector fields, and it is another vector field; see exercise~\ref{ex:commutator-coordinates}.

Consider then three vector fields $X,Y,Z$ on $M$.
Again, one can differentiate $Z$ covariantly with $X$ and $Y$ in the two directions.
The difference between the two orders is
\begin{equation}
[\nabla_X,\nabla_Y]Z.
\end{equation}

\begin{ex}
\label{ex:Eucl-cd-commutator}
Return to the Euclidean connection of exercise~\ref{ex:Eucl-connection}.
(This is the \LC{} connection of $\R^n$ as a Riemannian manifold.)
Show that
\begin{equation}
[\nabla_X,\nabla_Y]Z
=
\nabla_{[X,Y]}Z.
\end{equation}
This is exactly what we had for scalar fields on a general Riemannian manifold.
\end{ex}

Based on this observation we rephrase our question:
What is
\begin{equation}
[\nabla_X,\nabla_Y]Z
-\nabla_{[X,Y]}Z
?
\end{equation}

\begin{proposition}
\label{prop:cd-commutator}
There is a smooth tensor field $R$ of type\footnote{A multilinear map $T_x^*M\times T_xM\times T_xM\times T_xM\to\R$ can also be seen as a multilinear map $T_xM\times T_xM\times T_xM\to T_xM$. We take this interpretation here.} $(1,3)$ for which
\begin{equation}
R(X,Y,Z)
=
[\nabla_X,\nabla_Y]Z
-\nabla_{[X,Y]}Z.
\end{equation}
\end{proposition}

This tensor is often denoted as $R(X,Y)Z$ instead so that $R(X,Y)$ is seen as a linear map $T_xM\to T_xM$.
A tensor field often admits many different ways to view it.
This tensor is called the Riemann curvature tensor\footnote{Much more could be said about the meaning of curvature than is said in these notes. That would be a detour for our purposes.}.

\begin{proof}[Proof of proposition~\ref{prop:cd-commutator}]
It is clear that $R(X,Y)Z$ as given by the formula is linear in the three vector fields.
What is not trivial is that it does not depend on any derivatives but only on the values of the three vector fields at a point.
This can be verified by calculation.
\end{proof}

\begin{ex}
\label{ex:R-components}
Find a local coordinate expression for $[\nabla_X,\nabla_Y]Z-\nabla_{[X,Y]}Z$.
If the $i$th component of the vector $R(X,Y)Z$ is $R^i_{\phantom{i}jkl}X^jY^kZ^l$, find an expression for the components $R^i_{\phantom{i}jkl}$ of the Riemann curvature tensor.
Second order derivatives of the metric should appear.
You may also choose to use first order derivatives of Christoffel symbols.
\end{ex}

We will need analogous results for vector fields along curves.
First let $\Gamma\colon[0,1]\times(-\eps,\eps)\to M$ be any smooth map.
We have the natural vector fields $\partial_s\Gamma$ and $\partial_t\Gamma$ and they are well defined for any values of the two parameters.

\begin{lemma}
\label{lma:cd-commute-1}
The covariant derivatives of $\Gamma$ satisfy the commutator relationship
\begin{equation}
D_t\partial_s\Gamma
=
D_s\partial_t\Gamma.
\end{equation}
\end{lemma}

\begin{ex}
Prove the lemma.
\end{ex}

\begin{lemma}
\label{lma:cd-commute-2}
If $V(s,t)$ is any smooth vector field depending on the two parameters so that $V(s,t)\in T_{\Gamma(s,t)}M$, then
\begin{equation}
[D_s,D_t]V
=
R(\partial_s\Gamma,\partial_t\Gamma)V.
\end{equation}
\end{lemma}

The proof of this lemma is a computation similar to that of exercise~\ref{ex:R-components}.

\subsection{Jacobi fields}

As mentioned in section~\ref{sec:pre-JF}, we will study variation fields of families of geodesics.
It is important that all the curves are geodesics; otherwise there is no structure.

\begin{ex}
Show that for any vector field $V$ along a curve $\gamma$ there is a family of curves $\Gamma(\dum,s)$ so that the variation field of section~\ref{sec:pre-JF} is $V$.
Feel free to work in a single coordinate patch if it helps.\footnote{You have this liberty throughout the course.}
\end{ex}

When the family consists of geodesics, the variation field has special properties.
It will be what we shall call a Jacobi field.

\begin{ex}
\label{ex:Eucl-JF}
A Euclidean geodesic is of the form $\gamma_{x,v}(t)=x+tv$, parametrized by $x,v\in\R^n$.
Find all the possible variation fields along a Euclidean geodesic when all curves in the family are geodesics.
For any geodesic there should be a $2n$-dimensional space of such fields along it.
\end{ex}

\begin{definition}
\label{def:Rg}
The curvature operator along a geodesic $\gamma$ is a linear map $T_{\gamma(t)}M\to T_{\gamma(t)}M$ given by
\begin{equation}
\label{eq:Rg-def}
R_\gamma V
=
R(V,\dot\gamma)\dot\gamma.
\end{equation}
\end{definition}

This is in fact a $(1,1)$-tensor along the geodesic; such concepts can be defined by analogy to what we have done.

\begin{lemma}
\label{lma:Rg-normal}
We always have $\ip{\dot\gamma}{R_\gamma V}=0$.
\end{lemma}

\begin{proof}
This follows from a symmetry property of the Riemann curvature tensor, namely $\ip{W}{R(X,Y)Z}=-\ip{Z}{R(X,Y)W}$.
\end{proof}

\begin{lemma}
\label{lma:R-symmetric}
The curvature operator along a geodesic from definition~\ref{def:Rg} is symmetric: $\ip{V}{R_\gamma W}=\ip{R_\gamma V}{W}$.
\end{lemma}

\begin{proof}
This follows from a symmetry property of the Riemann curvature tensor, namely $\ip{W}{R(X,Y)Z}=\ip{X}{R(W,Z)Y}$.
\end{proof}

The operator $R_\gamma$ is symmetric, the operator $R(X,Y)$ is antisymmetric.

\begin{definition}
\label{def:JF}
Let $\gamma$ be a geodesic.
A vector field $J$ along $\gamma$ is called a Jacobi field if it satisfies the Jacobi equation
\begin{equation}
\label{eq:JE}
D_t^2J+R_\gamma J=0.
\end{equation}
\end{definition}

\begin{ex}
\label{ex:JF-unique}
Explain why a Jacobi field exists uniquely for all times, given $J$ and $D_tJ$ at one time.
\end{ex}

\begin{theorem}
\label{thm:JF=geodetic-variation}
The variation field of a family of geodesics is a Jacobi field.
Conversely, for every Jacobi field there is a family of geodesics whose variation field is the Jacobi field.
\end{theorem}

\begin{remark}
It is actually important for theorem~\ref{thm:JF=geodetic-variation} that a family of geodesics is a function $[0,1]\times(-\eps,\eps)\to M$, not $(0,1)\times(-\eps,\eps)\to M$.
The open intervals are harmless if the limit points still belong to the manifold, which is always true on a geodesically complete manifold.
If an endpoint is just outside the manifold, the family of geodesics might fail to exist as some of the geodesics can be forced to ``drop out''.
Feel free to assume geodesic completeness in this course when technical issues seem to arise.
\end{remark}

\begin{iex}
Prove the first half of the theorem as follows:
The fact that each $\Gamma(\dum,s)$ is a geodesic can be rewritten as $D_t^2\Gamma=0$.
Take $D_s$ of this equation and commute the derivatives using lemmas~\ref{lma:cd-commute-1} and~\ref{lma:cd-commute-2}.
Evaluate at $s=0$ to get a vector field along $\gamma=\Gamma(\dum,0)$.
\end{iex}

\begin{ex}
To prove the second half, proceed as follows:
You are given a Jacobi field $J(t)$ along a geodesic $\gamma(t)$, and you must find a family $\Gamma(t,s)$ with the correct variation field.
Let $a$ be a short curve on $M$ satisfying $a(0)=\gamma(0)$ and $\dot a(0)=J(0)$.
Argue why such an $a$ exists.
Let $b(s)$ be any vector field along $a(s)$ so that $D_sb(s)|_{s=0}=D_tJ(0)$ and $b(0)=\dot\gamma(0)$.
Argue why such a $b$ exists.
Now let $\Gamma(\dum,s)$ be the geodesic starting at $a(s)$ in the direction $b(s)$.
(Smoothness of $\Gamma$ follows from smoothness of the geodesic flow, to be established later.)
Let $V$ be the variation field of this family.
Use exercise~\ref{ex:JF-unique} to argue that $J=V$.
\end{ex}

\subsection{Parallel and normal Jacobi fields}
\label{sec:p-n-JF}

Let $\gamma$ be a geodesic throughout this subsection.
There are some special Jacobi fields, and we should understand them and the corresponding families of geodesics.

Reparametrization of geodesics produces more geodesics.
Consider the family $\Gamma(t,s)=\gamma(as+(1+bs)t)$.
The parameter $a$ describes the shift in the parametrization and $b$ describes the change in speed.
Every geodesic has constant speed, but that speed can vary with $s$.
The corresponding Jacobi field is
\begin{equation}
J(t)
=
(a+bt)\dot\gamma(t).
\end{equation}
Let us also verify using the Jacobi equation that this is indeed a Jacobi field.

It follows from lemma~\ref{lma:cd-commute-2} that $R(\dot\gamma,\lambda\dot\gamma)=0$ for any $\lambda\in\R$.
Therefore $R_\gamma\dot\gamma=0$.
The geodesic equation is $D_t\dot\gamma=0$, and so $D_t^2(a+bt)\dot\gamma(t)=0$.
Thus the Jacobi equation~\eqref{eq:JE} is satisfied.

Jacobi fields of this form are called parallel Jacobi fields.
They are somewhat uninteresting, as they reveal nothing about the behaviour of other geodesics than $\gamma$ itself.

For a general Jacobi field the inner product $\ip{\dot\gamma}{J}$ measures heuristically how much the varied geodesic gets ahead of $\gamma(t)$.
This inner product has a very rigid behaviour:

\begin{iex}
\label{ex:ip-J-g}
Let $J$ be a Jacobi field along a geodesic $\gamma$.
Show that\footnote{Using $t=0$ as the reference time is unimportant but convenient.}
\begin{equation}
\ip{\dot\gamma(t)}{J(t)}
=
\ip{\dot\gamma(0)}{J(0)}
+
t\ip{\dot\gamma(0)}{D_tJ(0)}
.
\end{equation}
The easiest way to do this is to compute the second covariant derivative of the inner product.
\end{iex}

Thus if both $J$ and $D_tJ$ are normal to $\dot\gamma$ at some point, then they both remain normal at all times.
Such Jacobi fields are called normal Jacobi fields.

The parallel component of a Jacobi field is
\begin{equation}
\begin{split}
J_p(t)
&=
\ip{\dot\gamma(t)}{J(t)}\dot\gamma(t)
\\&=
\ip{\dot\gamma(0)}{J(0)}\dot\gamma(t)
+
t\ip{\dot\gamma(0)}{D_tJ(0)}\dot\gamma(t).
\end{split}
\end{equation}
This is indeed a Jacobi field as verified above, and it is clearly parallel to $\dot\gamma$ at all times.
The normal component is
\begin{equation}
J_n(t)
=
J(t)-J_p(t).
\end{equation}
Exercise~\ref{ex:ip-J-g} shows that the Jacobi fields $J$ and $J_p$ have the same inner product against $\dot\gamma$ at all times.
Therefore $J_n(t)$ is indeed normal to $\dot\gamma$.
As the Jacobi equation is linear, $J_n$ is a Jacobi field.

It is not generally true that if a vector field satisfies an equation, then its parallel and normal components will as well.
This is a special feature of the Jacobi equation.

The parallel component of a Jacobi field describes how the parametrization of the family of geodesics varies.
The normal component describes how the geodesics as unparametrized curves or sets vary.
If a family of geodesics is reparametrized so that every geodesic has unit speed, then $\ip{\dot\gamma}{J}$ is constant.
The parameters can then be shifted to make this inner product vanish, making the corresponding Jacobi field normal.
Therefore it is often reasonable to restrict one's attention to only normal Jacobi fields, as they describe the ``true variations'' of geodesics.

\subsection{Spaces of constant curvature}
\label{sec:JF-const-curv}

Let us then take a brief look at Jacobi fields in some example spaces.

A space of constant (sectional) curvature $k$ looks locally like a Euclidean space ($k=0$), a hyperbolic space ($k<0$), or a sphere ($k>0$).
On such manifolds the curvature operator along a geodesic is given by
\begin{equation}
R_\gamma V
=
k(\abs{\dot\gamma}^2V-\ip{V}{\dot\gamma}\dot\gamma).
\end{equation}
The Jacobi equation for a normal Jacobi field along a unit speed geodesic becomes
\begin{equation}
D_t^2J+kJ=0.
\end{equation}
As $k$ is just a constant, this can be solved explicitly.

Let $e_1,\dots,e_{n-1},\dot\gamma$ be an orthonormal parallel frame along $\gamma$.
We can write our normal Jacobi field as
\begin{equation}
J(t)
=
\sum_{\alpha=1}^{n-1}J_\alpha(t)e_\alpha(t).
\end{equation}
As $D_te_\alpha=0$ and the frame is linearly independent at each point, we get the equation
\begin{equation}
J_\alpha''(t)+kJ_\alpha(t)=0.
\end{equation}
This is a constant coefficient ODE for a scalar function and can be solved explicitly:
\begin{equation}
J_\alpha(t)
=
\begin{cases}
J_\alpha(t)=a\sin(\sqrt{k}\,t)+b\cos(\sqrt{k}\,t) & \text{when } k>0, \\
J_\alpha(t)=at+b & \text{when } k=0, \\
J_\alpha(t)=ae^{\sqrt{-k}\,t}+be^{-\sqrt{-k}\,t} & \text{when } k<0.
\end{cases}
\end{equation}
The parameters $a,b\in\R$ can of course be different for different indices $\alpha$.

The flat case ($k=0$) should be familiar from exercise~\ref{ex:Eucl-JF}.
In positive curvature the Jacobi fields oscillate; consider variations of great circles on $S^2$.
\nts{picture?}
In negative curvature the behaviour is exponential; unless very carefully aimed, a Jacobi field grows exponentially when $t\to\pm\infty$.

The basic message is valid even when curvature is not constant:
In negative curvature nearby geodesics diverge, in positive curvature they converge.

\qa

\section{The exponential map}
\label{sec:exp}

In this section we will study all geodesics starting from a single point and collect all of them into a single object.

\subsection{Definitions}

If $x\in M$ and $v\in T_xM$, we denote by $\gamma_{x,v}$ the unique maximal\footnote{Defined on as long an interval as possible, containing zero.} geodesic for which $\gamma_{x,v}(0)=x$ and $\dot\gamma_{x,v}(0)=v$.
Exercise~\ref{ex:ge-e/u} provides the existence and uniqueness of such geodesics.

We would like to define the exponential map at $x$ to be $\exp_x\colon T_xM\to M$,
\begin{equation}
\exp_x(v)
=
\gamma_{x,v}(1).
\end{equation}
However, this does not necessarily make sense, as geodesics might not be defined all the way up to time $t=1$.
The definition is sensible as given if all geodesics through $x$ can be parametrized by the whole $\R$.
In other cases it needs to be defined on a subset of $T_xM$; as a small enough neighborhood of $0\in T_xM$ will be mapped nicely to points near $x$.

A calculation verifies the scaling law $\gamma_{x,\lambda v}(t)=\gamma_{x,v}(\lambda t)$ for any $\lambda\in\R$ for which everything is defined.
Therefore when $v\in T_xM$ is not zero, we can write $\exp_x(v)=\gamma_{x,v/\abs{v}}(\abs{v})$.
That is, the norm of the tangent vector gives the travel time.

As we can think of $T_xM$ as $\R^n$ upon fixing a basis, it makes sense to ask whether the exponential map is smooth.
It is.

\begin{ex}
\label{ex:exp-smooth}
Smoothness of the exponential map boils down to a general smoothness result for ODEs:
\begin{quote}
Suppose $F\colon\R^N\to\R^N$ is smooth.
Let $u(v,t)$ be defined so that $u(v,\dum)$ solves the ODE $\partial_tu(v,t)=F(u(v,t))$ and $u(v,0)=v$.
If $u$ is defined in an open set $\Omega\subset\R^N\times\R$, then $u$ is smooth in $\Omega$.
\end{quote}
Use this to prove that the exponential map is smooth where it is defined.
(Existence and uniqueness of $u$ was proven in exercise~\ref{ex:ge-e/u}.
Smoothness in time was proven in exercise~\ref{ex:geod-smooth}, but this is not enough.)
\end{ex}

There are different versions of the exponential map defined on different spaces.
The most immediate example is $\exp\colon TM\to M$ defined by $\exp(v)=\exp_x(v)$ when $v\in T_xM$.

\begin{iex}
Describe all unit speed geodesics through $x\in M$ using the exponential map.
\end{iex}

\begin{ex}
What is the exponential map of the Euclidean space $\R^n$ at a point $x\in\R^n$?
\end{ex}

\begin{ex}
On the smooth manifold $\R$ or a subset thereof a Riemannian metric is just a smooth function $g=g_{11}\colon\R\to(0,\infty)$.
The geodesic equation is $\ddot\gamma(t)+\frac12g'(\gamma(t))g^{-1}(\gamma(t))\dot\gamma(t)^2=0$.

Consider the metric $g(x)=x^{-2}$ on the manifold $M=(0,\infty)$.
What is the exponential map $\exp_1\colon T_1M\to M$?
\end{ex}

\subsection{Normal coordinates}

Let us fix $x\in M$.
We have learned that there is a neighborhood $\Omega\subset T_xM$ of the origin so that $\exp_x\colon\Omega\to M$ is well defined and smooth.
Since it can be differentiated, let us do so.

In general, the differential of a smooth map $f\colon N\to M$ at $y\in N$ is a map $\der f(y)\colon T_yN\to T_{f(y)}M$.
Using curves, it can be seen as the unique map for which any smooth curve on $N$ with $\gamma(0)=y$ satisfies $\partial_t(f(\gamma(t))|_{t=0}=\der f(y)\dot\gamma(0)$.
The curve-based definition is convenient as we may choose any curve with the correct $\dot\gamma(0)$.

\begin{ex}
Given a smooth map $f\colon\R^m\to\R^n$ and a point $y\in\R^m$, show that there exists a unique matrix $A$ for which $\partial_t(f(\gamma(t))|_{t=0}=A\dot\gamma(0)$ for any smooth curve $\gamma$ with $\gamma(0)=y$.
What is this $A$?
\end{ex}

The differential of the exponential map at the origin should be a map $\der\exp_x(0)\colon T_0(T_xM)\to T_xM$.
But as $T_xM$ is just a vector space (isometric to $\R^n$), we can naturally identify $T_0(T_xM)=T_xM$.

\begin{lemma}
\label{lma:dexp-id}
The differential $\der\exp_x(0)\colon T_xM\to T_xM$ of the exponential map is the identity map.
\end{lemma}

\begin{proof}
We use the curve definition of the differential.
Let $v\in T_xM$ be any vector.
We need a curve $\gamma\colon(-\eps,\eps)\to T_xM$ with $\gamma(0)=0$ and $\dot\gamma(0)=v$.
We choose $\gamma(t)=tv$.

Then we need to know what $\sigma(t)\coloneqq\exp_x(\gamma(t))$ is, because $\der\exp_x(0)v=\dot\sigma(0)$.
Now $\sigma(t)=\exp_x(tv)=\gamma_{x,tv}(1)=\gamma_{x,v}(t)$.
That is, $\sigma$ coincides with the geodesic $\gamma_{x,v}$.
This geodesic satisfies $\dot\gamma_{x,v}(0)=v\in T_xM$, so $\dot\sigma(0)=v$.

We have thus found that $\der\exp_x(0)v=v$.
\end{proof}

The exponential map maps radial lines in $T_xM$ into geodesics of $M$.
This is not generally true of lines that do not meet the origin.

\begin{iex}
\label{ex:exp-diffeo}
Show that there is a neighborhood $\Omega\subset T_xM$ of the origin and a neighborhood $U\subset M$ of $x$ so that $\exp_x\colon\Omega\to U$ is a diffeomorphism.
\end{iex}

If the inverse of the restricted $\exp_x$ of the exercise is called $\phi\colon U\to\Omega$ and $T_xM$ is identified with $\R^n$ using an orthonormal basis, we have a diffeomorphism $\phi\colon U\to\phi(U)\subset\R^n$.
In light of remark~\ref{rmk:coordinates-after-structure} this means that $\phi$ is a coordinate chart.
These coordinates are called the geodesic normal coordinates or Gaussian normal coordinates or just normal coordinates at $x$.

\begin{ex}
Given a point $x$ on a Riemannian manifold, how unique are the normal coordinates at it?
\end{ex}

\begin{ex}
\label{ex:CS-normal}
Study the geodesic equation~\eqref{eq:ge-2} in the normal coordinates at $x$.
Consider a geodesic passing through $x$ with velocity $v\in T_xM$.
Show that $\Cs ijkv^jv^k=0$ at $x$.
Use this information to conclude that $\Cs ijk=0$ at $x$.

In terms of the pseudoforce description of Christoffel symbols, this means that the system of coordinates can be chosen to be inertial (no Christoffel symbol, no pseudoforce) at a single point.
The normal coordinates do precisely this, but the symbol cannot be typically made vanish in an open set.
\end{ex}

\subsection{Differential of the exponential map}

We saw in lemma~\ref{lma:dexp-id} that the differential of the exponential map $\exp_x$ is the identity map on $T_xM$.
But it is smooth everywhere, so what is the derivative elsewhere?

Consider $0\neq v\in T_xM$ so that $\exp_x(v)$ is defined.
We would like to differentiate $\exp_x$ at $v$ in the direction of any $w\in T_xM$.
Therefore we study $\exp_x(v+sw)$ for some parameter $s\in(-\eps,\eps)$.

This gives rise to a family of geodesics defined by $\Gamma(t,s)=\exp_x(t(v+sw))$.
The derivative of $\exp_x$ at $v$ in the direction $w$ is
\begin{equation}
\der\exp_x(v)w
=
\partial_s\exp_x(v+sw)
=
\partial_s\Gamma(1,s)|_{s=0}.
\end{equation}
Let us denote $J_w(t)=\partial_s\Gamma(t,0)$.
This is a Jacobi field along $\gamma_{x,v}$.
The derivative is the value of this Jacobi field at $t=1$.

\begin{ex}
Let us find the initial conditions of the Jacobi field.
Verify that $\Gamma(0,s)=x$ and $\partial_t\Gamma(t,s)|_{t=0}=v+sw$ for all $s$.
Find $J_w(0)$ and $D_tJ_w(0)$.
\end{ex}

We have found that $\der\exp_x(v)$ maps a vector $w$ into the value of a Jacobi field along the geodesic $\gamma_{x,v}$ at $t=1$ with initial conditions $J_w(0)=0$ and $D_tJ_w(0)=w$.
One can therefore reasonably say that Jacobi fields vanishing at $x$ are the derivative of $\exp_x$.

\begin{ex}
This description is in fact valid for $v=0$ as well --- a constant curve is a geodesic..
Use this description in terms of Jacobi fields to find the differential of the exponential map at the origin.
\end{ex}

The derivatives satisfy an orthogonality condition named after Gauss:

\begin{theorem}[The Gauss lemma]
\label{thm:Gauss-lma-1}
Take any $v,w \in T_xM$ so that $\exp_x(v)$ is defined.
Then
\begin{equation}
\ip{\der\exp_x(v)v}{\der\exp_x(v)w}
=
\ip{v}{w}.
\end{equation}
\end{theorem}

Observe that the first inner product is on $T_{\exp_x(v)}M$ and the second one on $T_xM$.
Also notice that one of the two compared vectors has to be the direction of the corresponding geodesic.

\begin{proof}
The differential of the exponential is given by Jacobi fields.
We have $\der\exp_x(v)v=J_1(1)$ for the Jacobi field $J_1$ along $\gamma_{x,v}$ with the initial conditions $J_1(0)=0$ and $D_tJ_1(0)=v$.
But this Jacobi field is just $J_1(t)=t\dot\gamma_{x,v}(t)$.
(Recall that this is a Jacobi field with the correct initial condition and that solutions to the Jacobi equation are unique.)
Therefore $\der\exp_x(v)v=\dot\gamma_{x,v}(1)$.

Similarly, $\der\exp_x(v)w=J_2(1)$ for the Jacobi field $J_2$ along $\gamma_{x,v}$ with the initial conditions $J_2(0)=0$ and $D_tJ_2(0)=w$.
Exercise~\ref{ex:ip-J-g} gives
\begin{equation}
\ip{\der\exp_x(v)v}{\der\exp_x(v)w}
=
\ip{\dot\gamma_{x,v}(1)}{J_2(1)}
=
\ip{v}{J_2(0)}
+
1\ip{v}{D_tJ_2(0)}.
\end{equation}
Using the initial conditions of $J_2$ gives the claim.
\end{proof}

There is a more geometric version of the lemma, but that requires some setting up.

\begin{remark}
\label{rmk:dexp-block}
Take any non-zero $v\in T_xM$ and denote the corresponding unit vector by $\hat v=v/\abs{v}$.
We can complete $\{v\}$ into an orthonormal basis $\{e_1,e_2,\dots,e_{n-1},e_n=\hat v\}$ of $T_xM$.
When we parallel transport these vectors along $\gamma_{x,v}$, we get an orthonormal parallel frame along this geodesic.
The differential $\der\exp_x(v)$ of the exponential maps from $T_{\gamma_{x,v}(0)}M$ to $T_{\gamma_{x,v}(1)}M$.
Our frame gives a basis for both spaces.
Therefore in this frame we can write $\der\exp_x(v)$ as a matrix.
Let us write it in block form, separating the last component from the $n-1$ first ones:
\begin{equation}
\der\exp_x(v)
=
\begin{pmatrix}
A&b\\
c^T&d
\end{pmatrix}
,
\end{equation}
where $A$ is an $(n-1)\times(n-1)$ matrix, $b$ and $c$ are column vectors of dimension $n-1$, and $d\in\R$.
\end{remark}

\begin{ex}
\label{ex:dexp-block}
Use the results obtained so far to argue that
\begin{itemize}
\item
$b=0$,
\item
$d=\abs{v}$,
\item
$c=0$,
and
\item
$A$ is given by values of normal Jacobi fields along $\gamma_{x,v}$ that vanish at $t=0$.
\end{itemize}
No new proofs should be required here, just recollection and perhaps recontextualization of what has already been done.
\end{ex}

\subsection{Submanifolds}

When it comes to submanifolds, geometric intuition serves well for basics concepts and we will not need to go much beyond that.
We need to formalize a couple of concepts, but we will not attempt to build a complete theory or give all the details.

A subset $N\subset M$ is submanifold of dimension $k<n$ if near any point $x\in M$ in local coordinates it is a smooth $k$-dimensional surface in $\R^n$ in the usual sense.
A $k$-dimensional surface $\Sigma\subset\R^n$ can be defined, for example, as the image of a smooth map $\Omega\to\R^n$ from an open $\Omega\subset\R^k$ with an everywhere injective differential.
An alternative way is to require that $\Sigma$ is a level set of a function of a smooth function $\R^n\to\R^{n-k}$ with an everywhere surjective differential.
These definitions can be rephrased to work directly on manifolds as well, being careful to work locally.

An important property is that a $k$-dimensional submanifold $N\subset M$ is also a manifold in its own right.
It also inherits a Riemannian structure from the ambient space $M$.

For any $x\in N\subset M$ the tangent space of $N$ is a subspace of the tangent space of $M$.
That is, $T_xN\subset T_xM$.
There is a curve-based way to define this linear subspace: $T_xN$ consists of the velocities $\dot\gamma(0)$ of curves $\gamma\colon I\to N\subset M$ for which $\gamma(0)=x$.
That is, $T_xN$ consists of velocities of of curves staying in $N$.

A vector $v\in T_xM$ is said to be normal to a submanifold $N\subset M$ containing $x$ if $\ip{v}{w}=0$ for all $w\in T_xN$.
A basic argument in linear algebra shows that if $N$ has dimension $n-1$, then there is a unique unit normal vector to $N$ at $x$ up to sign.
One can locally define a smooth normal vector field on $N$.
We can say that a curve $\gamma$ meets $N$ orthogonally if at the intersection point $\dot\gamma$ is normal to $N$.

\subsection{Spheres}
\label{sec:sphere}

A geodesic sphere of radius $r>0$ centered at $x\in M$ is the set
\begin{equation}
\{
\exp_x(v);
v\in T_xM,
\abs{v}=r
\}.
\end{equation}
This is the image of the sphere $S(0,r)\subset T_xM$ under the exponential map.

The metric sphere of radius $r>0$ centered $x\in M$ is the set
\begin{equation}
\{
y\in M;
d(x,y)=r
\}.
\end{equation}
This is the set of points at distance $r$ from $x$.

These surfaces are closely related as we will soon see.
Notice that the geodesic sphere is the image of a smooth $(n-1)$-dimensional surface (a sphere of the tangent space) under a smooth map.
Therefore it is smooth at least when $\der\exp_x$ is bijective.
This happens at least near the origin by exercise~\ref{ex:exp-diffeo}.

\begin{theorem}[The Gauss lemma for spheres]
\label{thm:Gauss-lma-2}
Suppose that the geodesic sphere of radius $\abs{v}$ centered at $x\in M$ is a smooth submanifold near $\exp_x(v)$.
Then the geodesic $\gamma_{x,v}$ is normal to the geodesic sphere.
\end{theorem}

\begin{proof}
Let us take curves staying on the geodesic sphere.
These are best described as $\alpha(t)=\exp_x(\sigma(t))$, where $\sigma\colon(-\eps,\eps)\to S(0,\abs{v})\subset T_xM$ is a smooth curve with $\sigma(0)=v$.
Since $\sigma$ stays on the sphere, we have $0=\partial_t\abs{\sigma(t)}^2=2\ip{\sigma(t)}{\dot\sigma(t)}$ and so $\dot\sigma(0)$ is orthogonal to $v$.
A tangent vector to the geodesic sphere is then $\dot\alpha(0)=\der\exp_x(v)\dot\sigma(0)$, and by theorem~\ref{thm:Gauss-lma-1} this is orthogonal to $\dot\gamma_{x,v}(1)$.
\end{proof}

\qa

\section{Minimization of length}

\subsection{Short geodesics minimize length}

We are now ready to see why geodesics minimize length.
Before stating the theorem, we will need to recall the length of a geodesic.

\begin{iex}
\label{ex:geod-length}
Show that the length of the geodesic $\gamma_{x,v}\colon[0,1]\to M$ is $\abs{v}$ whenever the geodesic is defined on the whole interval.
\end{iex}

\begin{theorem}
\label{thm:exp-d}
Let $x\in M$ and let $r>0$ be such that $\exp_x\colon B(0,r)\to U\subset M$ is a diffeomorphism.
Then for any $v\in B(0,r)\subset T_xM$ the distance between the endpoints of the corresponding geodesic is
\begin{equation}
d(x,\exp_x(v))
=
\abs{v}.
\end{equation}
In fact, $\gamma_{x,v}|_{[0,1]}$ is the unique shortest curve between its endpoints.
\end{theorem}

\begin{proof}
The result is clear if $v=0$ so we assume $v\neq0$.
We will show that any curve from $x$ to the geodesic sphere of radius $\abs{v}$ centered at $x$ has at least length $r$.
Every curve from $x$ to $\exp_x(v)$ will have to meet this sphere.
It is enough to show that the segment of the curve until the first intersection with this sphere has at least length $r$.

We may also assume that the curve we compare to does not meet $x$ again after $t=0$.
Otherwise we could take the segment from a later intersection point to get an even shorter curve.

That is, we use a segment of the arbitrary curve and show that it has length $r$ or more, whence the original curve will have at least this length.

So, let $\gamma\colon[0,1]\to B(0,r)\subset T_xM$ be a smooth curve with $\abs{\gamma(1)}=\abs{v}$.
Then $\sigma=\exp_x\circ\gamma$ is a curve on $M$ from $x$ to the geodesic sphere of radius $\abs{v}$.
We have
\begin{equation}
\label{eq:chain-1}
\begin{split}
\abs{v}
&=
\abs{\gamma(1)}
\\&\stackrel{\text{(a)}}{=}
\int_0^1
\Der{t}\abs{\gamma(t)}
\dd t
\\&\stackrel{\text{(b)}}{=}
\int_0^1
\abs{\gamma(t)}^{-1}\ip{\gamma(t)}{\dot\gamma(t)}
\dd t
\\&\stackrel{\text{(c)}}{=}
\int_0^1
\abs{\gamma(t)}^{-1}\ip{\der\exp_x(\gamma(t))\gamma(t)}{\der\exp_x(\gamma(t))\dot\gamma(t)}
\dd t
\\&\stackrel{\text{(d)}}{\leq}
\int_0^1
\abs{\gamma(t)}^{-1}\abs{\der\exp_x(\gamma(t))\gamma(t)}\abs{\der\exp_x(\gamma(t))\dot\gamma(t)}
\dd t
\\&\stackrel{\text{(e)}}{=}
\int_0^1
\abs{\der\exp_x(\gamma(t))\dot\gamma(t)}
\dd t
\\&\stackrel{\text{(f)}}{=}
\int_0^1\abs{\dot\sigma(t)}\dd t
\\&\stackrel{\text{(g)}}{=}
\ell(\sigma).
\end{split}
\end{equation}
Justifying each step is an exercise.

By exercise~\ref{ex:geod-length} we have $\abs{v}=\ell(\gamma_{x,v}|_{[0,1]})$.
Therefore
\begin{equation}
\ell(\gamma_{x,v}|_{[0,1]})
\leq
\ell(\sigma).
\end{equation}
Thus the geodesic is indeed the shortest curve.

Let us then show that it is the unique one.
If equality holds throughout~\eqref{eq:chain-1}, the vectors
$\der\exp_x(\gamma(t))\gamma(t)$
and
$\der\exp_x(\gamma(t))\dot\gamma(t)$
must be parallel\footnote{This does not refer to parallel transport here, but to one vector being a scalar multiple of the other.} at all times.
By exercise~\ref{ex:dexp-parallel} this means that $\gamma(t)$ and $\dot\gamma(t)$ are parallel.

As we assumed that $\gamma(t)\neq0$ for $t>0$, this implies that $\gamma(t)=h(t)w$ for some increasing smooth surjection $h\colon[0,1]\to[0,1]$ and a constant vector $w\in T_xM$ with $\abs{w}=\abs{v}$.
Upon choosing constant speed parametrization --- which does not change length --- we have $\gamma(t)=tw$.

If $\sigma=\exp_x\circ\gamma$ is a shortest path from $x$ to $\exp_x(v)$, then $\sigma$ must be of the form $\sigma(t)=\exp_x(tw)$.
To get the end point right, we must have $\exp_x(w)=\exp_x(v)$.
The exponential map is diffeomorphic in the set we are in, so $w=v$.

Thus any minimizing curve between the same endpoints must indeed coincide with our geodesic up to reparamterization.
\end{proof}

\begin{ex}
Let us revisit the topological argument used in the proof.
We only wanted to work within the ball $B(0,r)$, so we argued that any curve not staying within it will have to meet the sphere.

Let $\gamma\colon[0,1]\to\R^n$ be continuous with $\gamma(0)=0$ and $\abs{\gamma(1)}>1$.
Show that $\abs{\gamma(t)}=1$ for some $t\in(0,1)$.
\end{ex}

\begin{ex}
Justify the named steps in~\eqref{eq:chain-1}.
\end{ex}

\begin{ex}
\label{ex:dexp-parallel}
Show using the Gauss lemma that
$\der\exp_x(v)v$
and
$\der\exp_x(v)w$
are parallel (so that one is a scalar multiple of the other) if and only if
$v$
and
$w$
are parallel.
\end{ex}

\begin{iex}
\label{ex:local-geodesic-short}
Show that every point $x\in M$ has a neighborhood $U$ so that for any $y\in U$ there is a unique shortest curve between $x$ and $y$ and it is a geodesic.
\end{iex}

\begin{ex}
Show that for small enough $r>0$ the metric sphere coincides with the geodesic sphere.
\end{ex}

\subsection{Conjugate points}

We now have a pretty good understanding of what happens when the exponential map is a diffeomorphism.
When we go far enough from the base point, it might stop being diffeomorphic.
We will now turn to studying that.

\begin{proposition}
\label{prop:dexp-singular-JF}
The exponential map $\exp_x\colon T_xM\to M$ has a bijective differential at $v\in T_xM\setminus0$ if and only if for any non-trivial Jacobi field $J$ along $\gamma_{x,v}$ that vanishes at $t=0$ is non-zero at $t=1$.
\end{proposition}

\begin{proof}
In remark~\ref{rmk:dexp-block} we write the differential as a matrix using a parallel orthonormal frame along the geodesic $\gamma_{x,v}$.
In exercise~\ref{ex:dexp-block} we saw that this matrix is of the form
$
\begin{pmatrix}
A&0\\
0&d
\end{pmatrix}
$
for some $d>0$.
Therefore the linear map $\der\exp_x(v)$ is bijective if and only the matrix $A$ is invertible.

The matrix $A$ was defined so that if a Jacobi field $J$ along the geodesic satisfies $J(0)=0$ and $D_tJ(0)=w$, then $J(1)=Aw$.
Notice that $D_tJ(0)\in T_xM$ and $J(1)\in T_{\exp_x(v)}M$, but the parallel frame gives a way to identify these two vector spaces.
The matrix $A$ only fails to be invertible when there is $w\neq 0$ so that $Aw=0$.
This is equivalent with the existence of a Jacobi field $J$ for which $J(0)=0$, $D_tJ(0)\neq0$, and $J(1)=0$.

By exercise~\ref{ex:JF-unique} a Jacobi field $J$ is uniquely determined by $J(0)$ and $D_tJ(0)$.
If we require $J(0)=0$, then the Jacobi field is non-trivial if and only if $D_tJ(0)\neq0$.
\end{proof}

\begin{ex}
Show that if a non-trivial Jacobi field vanishes at two different points, then it is normal.
\end{ex}

Proposition~\ref{prop:dexp-singular-JF} inspires us to give a name for the case when a non-trivial Jacobi field vanishes at two points.

\begin{definition}
Let $\gamma\colon I\to M$ be a geodesic and $a,b\in I$.
We say that the points $\gamma(a)$ and $\gamma(b)$ are conjugate along $\gamma$ if there is a non-trivial Jacobi field along $\gamma$ that vanishes at both $a$ and $b$.
\end{definition}

Just like parallel transport, conjugate points are a concept along a geodesic, not between a pair of points.

\begin{ex}
\label{ex:cp-char}
Let $\gamma\colon I\to M$ be a geodesic with non-zero speed and $a,b\in I$.
Show that the following are equivalent:
\begin{itemize}
\item
The points $\gamma(a)$ and $\gamma(b)$ are not conjugate along $\gamma$.
\item
The differential $\der\exp_{\gamma(a)}((b-a)\dot\gamma(a))$ is a bijection.
\item
If a Jacobi field $J$ along $\gamma$ vanishes at both $a$ and $b$, it is identically zero.
\end{itemize}
The last point can be understood as a Jacobi field being uniquely determined by its values at two non-conjugate points.
If the two points are conjugate, setting these two values is (somewhat) redundant.
\end{ex}

\begin{remark}
Yet another equivalent condition is that the geodesic sphere is smooth at that point.
This is very plausible, but it is possible for a smooth map with a non-invertible differential to map a smooth manifold into a smooth manifold.
For the exponential map this cannot happen, but studying the details would be a digression.
\end{remark}

\begin{ex}
Give an example of a map $f\colon\R^2\to\R^3$ for which $f(\R^2)$ is a smooth surface and the derivative matrix of $f$ is invertible almost everywhere but not everywhere.
\end{ex}

\subsection{Second variation of length}

The way we first found the geodesic equation was to study variations of the length of a curve.
We essentially defined geodesics to be critical points of the length functional --- with constant speed.

In general there is no guarantee that a critical point is a local minimum.
We just showed that short enough geodesics are globally minimal.
To study minimality locally, we need to calculate second derivative and see whether it is positive definite.

The second variation is most interesting when the reference curve is a geodesic, a critical point.
This will also simplify matters considerably.

We will consider again a family of curves $\Gamma(t,s)$.
We now assume that $\Gamma(\dum,0)$ is a geodesic and we assume that each $\Gamma(\dum,s)$ has constant speed.

\begin{proposition}
\label{prop:2nd-variation}
Let $\Gamma\colon[0,1]\times(-\eps,\eps)\to M$ be a smooth map so that
\begin{itemize}
\item
$\abs{\partial_t\Gamma(t,s)}=c_s$, a constant depending on $s$ but not $t$,
\item
$\Gamma(0,s)=p$ for all $s$, and
\item
$\Gamma(1,s)=q$ for all $s$.
\end{itemize}
Denoting $\gamma(t)=\Gamma(t,0)$ and $V(t)=\partial_s\Gamma(t,s)|_{s=0}$, we have
\begin{equation}
\partial_s^2\ell(\Gamma(\dum,s))|_{s=0}
=
\frac1{\ell(\gamma)}
\int_0^1
\left(
\abs{D_tV}^2
-
\ip{V}{R_\gamma V}
\right)
\dd t.
\end{equation}
\end{proposition}

Here $R_\gamma$ is the curvature operator along $\gamma$ from definition~\ref{def:Rg}.
Notice that as $\dot\gamma\neq0$, we have $\partial_t\Gamma(t,s)\neq0$ everywhere if $\eps>0$ is small enough --- therefore constant speed parametrization is legitimate.

\begin{proof}
Proposition~\ref{prop:1st-variation} (the first variation) was phrased and proven in local coordinates.
Now we will do things invariantly.

Let us denote $\ell(\Gamma(\dum,s))=\ell(s)$.
First we observe that since each $\Gamma(\dum,s)$ has constant speed and is defined on $[0,1]$, we have $\ell(s)=c_s$.
In fact, as $\gamma$ is a geodesic, $\ell'(0)=0$.

To get started, we use the reformulation~\eqref{eq:1st-variation-invariantly} of the first variation formula.
Now that the constant speed condition is satisfied for all $s$, the formula is valid for all $s$.
We have
\begin{equation}
\ell'(s)
=
-\int_0^1
\frac1{\ell(s)}
\ip{\partial_s\Gamma}{D_t\partial_t\Gamma}
\dd t.
\end{equation}
We can now simply differentiate under the integral sign and evaluate at $s=0$ to get
\begin{equation}
\ell''(0)
=
-
\frac1{\ell(\gamma)}
\int_0^1
\partial_s\ip{\partial_s\Gamma}{D_t\partial_t\Gamma}|_{s=0}
\dd t.
\end{equation}
The derivatives $\partial_s$ and $D_s$ are derivatives along the curves $\Gamma(t,\dum)$ for fixed $t$.

Using exercise~\ref{ex:Dt-ip} we get
\begin{equation}
\partial_s\ip{\partial_s\Gamma}{D_t\partial_t\Gamma}|_{s=0}
=
\ip{D_s\partial_s\Gamma}{D_t\partial_t\Gamma}|_{s=0}
+
\ip{\partial_s\Gamma}{D_sD_t\partial_t\Gamma}|_{s=0}.
\end{equation}
The first term vanishes because $D_t\partial_t\Gamma(t,0)=0$ --- after all, $\gamma$ is a geodesic.
Exercise~\ref{ex:commute-1} gives that
\begin{equation}
D_sD_t\partial_t\Gamma|_{s=0}
=
D_t^2V+R_\gamma V.
\end{equation}
With these ingredients we can simplify our second derivative to
\begin{equation}
\ell''(0)
=
-
\frac1{\ell(\gamma)}
\int_0^1
\ip{V}{D_t^2V+R_\gamma V}
\dd t.
\end{equation}
Integration by parts in the first term gives the claim since $V$ vanishes at the endpoints.
(See exercise~\ref{ex:ibp-vf} for details on integration by parts.)
\end{proof}

We will study this formula in more detail in section~\ref{sec:if}.

\begin{iex}
\label{ex:commute-1}
Commute the derivatives to prove that
\begin{equation}
D_sD_t\partial_t\Gamma
=
D_t^2\partial_s\Gamma
+
R(\partial_s\Gamma,\partial_t\Gamma)\partial_t\Gamma.
\end{equation}
At $s=0$ this becomes $D_t^2V+R_\gamma V$.
\end{iex}

\begin{ex}
\label{ex:ibp-vf}
Let us justify integration by parts of vector fields.
Let $V$ and $W$ be two vector fields along a geodesic $\gamma\colon[a,b]\to M$.
Show that
\begin{equation}
\int_a^b
\ip{V}{D_tW}
\dd t
=
\ip{V(b)}{W(b)}
-
\ip{V(a)}{W(a)}
-\int_a^b
\ip{D_tV}{W}
\dd t.
\end{equation}
It may help to recall how the integration by parts formula for functions on the real line is proven.
\end{ex}

\begin{ex}
Show that it follows from the assumptions of proposition~\ref{prop:2nd-variation} that the variation field is normal to the geodesic $\gamma$ at all times.
It can help to show first that $2\partial_t\ip{\partial_t\Gamma}{\partial_s\Gamma}=\partial_s\ip{\partial_t\Gamma}{\partial_t\Gamma}$ at $s=0$ and to recall that $\ell'(0)=0$.
\end{ex}

As was mentioned in section~\ref{sec:p-n-JF}, only the normal component of the variation field is geometrically meaningful.
The parallel component corresponds to reparametrization.

\qa

\section{The index form}
\label{sec:if}

\subsection{Second variation of length}

Let us denote by $\NVF(\gamma)$ the space of normal vector fields along a geodesic $\gamma\colon[a,b]\to M$.
Let $\NVF_0(\gamma)\subset\NVF(\gamma)$ be the subspace of vector fields vanishing at the endpoints.
The space $\NVF_0(\gamma)$ describes proper first order variations of a geodesic $\gamma$ with fixed endpoints.
Since the first order variation of the length vanishes, the second order variation of length only depends on the first order variation of the curve itself.

We found a formula for the second variation of length in proposition~\ref{prop:2nd-variation}.
Inspired by that, we give a name to the gadget we found.

\begin{definition}
Let $\gamma\colon[a,b]\to M$ be a geodesic.
The index form $\IF=\IF_\gamma$ of $\gamma$ is a quadratic form on $\NVF(\gamma)$ defined by
\begin{equation}
\IF(V,W)
=
\int_a^b
\left(
\ip{D_tV}{D_tW}
-
\ip{V}{R_\gamma W}
\right)
\dd t.
\end{equation}
\end{definition}

It follows from lemma~\ref{lma:R-symmetric} that the index form is symmetric.

\begin{definition}
Let $E$ be a real vector space and $Q\colon E\times E\to\R$ a quadratic form\footnote{That, is $Q$ is a symmetric element of $E^*\otimes E^*$.}.
We say that
\begin{itemize}
\item
$Q$ is positive definite if $Q(v,v)>0$ for all $v\in E\setminus0$.
\item
$Q$ is positive semidefinite if $Q(v,v)\geq0$ for all $v\in E$.
\item
$Q$ is negative (semi)definite if $-Q$ is positive (semi)definite.
\item
$Q$ is indefinite if $Q(v,v)>0$ and $Q(w,w)<0$ for some $v,w\in E$.
\end{itemize}
\end{definition}

\begin{ex}
Let $\gamma\colon[a,b]\to M$ be a unit speed geodesic.
Show that the second variation of its length corresponding to a family of curves with a normal variation field $V\in\NVF(\gamma)$ is $\IF(V,V)$.
You only need to rescale proposition~\ref{prop:2nd-variation} to unit speed and a general interval.
\end{ex}

One should therefore think of the index form as the Hessian of the length functional.
Any geodesic can be made longer by adding wiggles, so the index form cannot be negative definite or semidefinite.
All other options are possible as we will see.

\begin{ex}
\label{ex:indef-not-minimal}
Show that if $\IF_\gamma$ is not positive semidefinite on $\NVF_0(\gamma)$, then $\gamma$ is not the shortest curve between its endpoints.
This together with theorem~\ref{thm:exp-d} implies that for any $x$ and $v\in T_xM$ there is $\delta>0$ so that $\IF_{\gamma_{x,v}|_{[0,\eps]}}$ is positive semidefinite on $\NVF_0(\gamma)$.
\end{ex}

A local minimum need not be a global one.
Even if the index form is positive definite, the geodesic can fail to be minimizing.
There can be a curve taking an entirely different route between the two endpoints.
No amount of local analysis along a curve can rule this out.

\subsection{Jacobi fields, conjugate points, and definiteness}

Integration by parts (exercise~\ref{ex:ibp-vf}) reveals a connection between the index form and Jacobi fields.

\begin{iex}
\label{ex:IF-JF}
Let $V\in\NVF(\gamma)$.
Show that the following are equivalent:
\begin{enumerate}
\item
$V$ is a Jacobi field.
\item
$\IF(V,W)=0$ for all $W\in\NVF_0(\gamma)$.
\end{enumerate}
Why is it important that $W$ vanishes at the endpoints?
\end{iex}

\begin{remark}
\label{rmk:IF-JF}
Exercise~\ref{ex:IF-JF} has an interesting implication if the endpoints of the geodesic are conjugate.
Then there is a Jacobi field $J\in\NVF_0(\gamma)\setminus0$, and by the exercise $\IF(J,J)=0$.
Therefore positive definiteness is impossible in this case.
This connection between conjugate points and the definiteness of the index form goes much further as we will see next.
\end{remark}

\begin{lemma}
\label{lma:IF<0}
Let $\gamma\colon[a,b]\to M$ be a geodesic.
If there are conjugate points $\gamma(a')$ and $\gamma(b')$ along $\gamma$ so that $0<b'-a'<b-a$, then there is $V\in\NVF_0(\gamma)$ so that $\IF(V,V)<0$.
\end{lemma}

\begin{proof}
There is a non-trivial Jacobi field along $\gamma$ satisfying $J(a')=0=J(b')$.
The piecewise smooth vector field $\bar J$ defined by
\begin{equation}
\bar J(t)
=
\begin{cases}
J(t), & a'<t<b'\\
0, & \text{otherwise}
\end{cases}
\end{equation}
describes, roughly, a piecewise geodesic curve with the same length as $\gamma$ and with corners at $a'$ and $b'$.
Once we cut the corners, we should get a curve shorter than $\gamma$.

We assume that $a<a'$ and $b'<b$.
At least one has to be true, and if the other is replaced by an equality, the analysis we will do can be restricted to the other point.
It is enough to find a normal $C^1$ vector field $V$ with the desired property; see exercise~\ref{ex:IF<0-mollify}.

Let us denote $\zeta=D_tJ(a')$.
We can then parallel transport it as a vector field $\zeta(t)$ with $\zeta(a')=\zeta$.
This vector is normal to $\dot\gamma$ at all times.
Notice that since $J(a')=0$ but $J$ is not identically zero, $\zeta\neq0$.
For small $\eps>0$ we define a normal vector field $Z$ along $\gamma$ as
\begin{equation}
Z(t)
=
\begin{cases}
C\eps^{-1}(\abs{t-a'}-\eps)^2\zeta(t), & \abs{t-a'}<\eps\\
0, & \text{otherwise}
\end{cases}
\end{equation}
with some positive constant $C>0$.

Similarly, if $\eta=D_tJ(b')$, we define a parallel transport $\eta(t)$ and let\footnote{Capital $\zeta$ is $Z$, capital $\eta$ is $H$.}
\begin{equation}
H(t)
=
\begin{cases}
-C\eps^{-1}(\abs{t-b'}-\eps)^2\eta(t), & \abs{t-b'}<\eps\\
0, & \text{otherwise}
\end{cases}
\end{equation}
with the same constant $C>0$.
These two vector fields ``cut the corners'' as explained above.

We define $V(t)=\bar J(t)+Z(t)+H(t)$.
As a sum of three normal vector fields it is a normal vector field.
With a suitable choice of $C>0$ this vector field is $C^1$; see exercise~\ref{ex:IF<0-C1}.
Now it remains to show that $\IF(V,V)<0$ when $\eps>0$ is small enough.
We have
\begin{equation}
\label{eq:chain-2}
\begin{split}
\IF(V,V)
&=
\int_a^b
\left(
\abs{D_tV}^2
-
\ip{R_\gamma V}{V}
\right)
\dd t
\\&=
\int_a^b
\big(
\abs{D_t\bar J}^2
+
2\ip{D_t\bar J}{D_t(Z+H)}
+
\abs{D_t(Z+H)}^2
\\&\quad
-
\ip{R_\gamma \bar J}{\bar J}
-
2\ip{R_\gamma \bar J}{Z+H}
-
\ip{R_\gamma (Z+H)}{Z+H}
\big)
\dd t
\\&=
\int_{a'}^{b'}
\left(
\abs{D_t\bar J}^2
-
\ip{R_\gamma \bar J}{\bar J}
\right)
\dd t
\\&\quad+
\int_{a'-\eps}^{a'+\eps}
\left(
2\ip{D_t\bar J}{D_tZ}
+
\abs{D_tZ}^2
-
2\ip{R_\gamma \bar J}{Z}
-
\ip{R_\gamma Z}{Z}
\right)
\dd t
\\&\quad+
\int_{b'-\eps}^{b'+\eps}
\left(
2\ip{D_t\bar J}{D_tH}
+
\abs{D_tH}^2
-
2\ip{R_\gamma \bar J}{H}
-
\ip{R_\gamma H}{H}
\right)
\dd t.
\end{split}
\end{equation}
If we use exercise~\ref{ex:IF-JF} or remark~\ref{rmk:IF-JF} on the geodesic segment $\gamma|_{a',b'}$, we see that
\begin{equation}
\int_{a'}^{b'}
\left(
\abs{D_t\bar J}^2
-
\ip{R_\gamma \bar J}{\bar J}
\right)
\dd t
=
0.
\end{equation}
Since $\bar J$ is Lipschitz and vanishes at $a'$ and $b'$, we have $\abs{\bar J}=\Order(\eps)$ in the last two integrals of~\eqref{eq:chain-2}.
We also have $\abs{Z}=\Order(\eps)$ and $\abs{H}=\Order(\eps)$.
Exercise~\ref{ex:W12-Z-H} gives the other two integrals with contain only $Z$ and $H$.
As the intervals of integration have length $2\eps$, we have
\begin{equation}
\begin{split}
\IF(V,V)
&=
2\int_{a'}^{a'+\eps}
\ip{D_t J}{D_tZ}
\dd t
+
2\int_{b'-\eps}^{b'}
\ip{D_t J}{D_tH}
\dd t
\\&\quad+
\frac{8}{3}C^2\abs{\zeta}^2\eps
+
\frac{8}{3}C^2\abs{\eta}^2\eps
+
\Order(\eps^3).
\end{split}
\end{equation}
Let us study the first remaining integral.
In it $D_tJ(t)=\zeta(t)+\Order(\eps)$.
Using exercise~\ref{ex:W12-Z-H} gives thus
\begin{equation}
2\int_{a'}^{a'+\eps}
\ip{D_t J}{D_tZ}
\dd t
=
-4C\abs{\zeta}^2\eps
+
\Order(\eps^2).
\end{equation}
The other integral gives a similar negative leading order term.

We have arrived at
\begin{equation}
\IF(V,V)
=
-4C\abs{\zeta}^2\eps
-4C\abs{\eta}^2\eps
+\frac{8}{3}C^2\abs{\zeta}^2\eps
+\frac{8}{3}C^2\abs{\eta}^2\eps
+\Order(\eps^2).
\end{equation}
With our $C>0$ we have $4C>8C^2/3$, whence
\begin{equation}
\IF(V,V)
=
-\abs{\zeta}^2\left(4C-\frac83C^2\right)\eps
-\abs{\eta}^2\left(4C-\frac83C^2\right)\eps
+\Order(\eps^2)
\end{equation}
is indeed negative for $\eps>0$ small enough.
\end{proof}

\begin{ex}
\label{ex:IF<0-mollify}
Polish the proof by showing that if there is a compactly supported normal vector field $V$ with $C^1$ regularity so that $\IF_\gamma(V,V)<0$, then there is a smooth one as well.
\end{ex}

\begin{ex}
\label{ex:IF<0-C1}
Choose $C>0$ so that the vector field $V(t)$ of the proof is actually $C^1$.
What is the value of the constant and why is the resulting vector field $C^1$?
Verify that $4C>8C^2/3$.
\end{ex}

\begin{ex}
\label{ex:W12-Z-H}
Show that
\begin{equation}
\int_{a'-\eps}^{a+\eps}\abs{D_tZ}\dd t=\frac{8}{3}C^2\abs{\zeta}^2\eps
\end{equation}
and
\begin{equation}
\int_{a'}^{a'+\eps}\ip{\zeta(t)}{D_tZ(t)}\dd t=-2C\abs{\zeta}^2\eps.
\end{equation}
Similar formulas hold for $H$ with the norm of $\eta$.
\end{ex}

\begin{lemma}
\label{lma:IF>0}
Let $\gamma\colon[a,b]\to M$ be a geodesic.
If there are no conjugate points along $\gamma$, then $\IF(V,V)>0$ for all $V\in\NVF_0(\gamma)\setminus0$.
\end{lemma}

\begin{proof}
Let $\zeta_1,\dots,\zeta_{n-1},\dot\gamma(a)$ be an orthonormal basis of $T_{\gamma(a)}M$.
We can extend these into an orthonormal parallel frame with the transported vectors $\zeta_\alpha(t)$.
For $\alpha\in\{1,\dots,n-1\}$ let $J_\alpha$ be the Jacobi field with $J_\alpha(a)=0$ and $D_tJ_\alpha(a)=\zeta_\alpha$.
Near the initial point we have $J_\alpha(t)=t\zeta_\alpha(t)+\Order(t^2)$.

When $t_0\in(a,b]$, the vectors $J_\alpha(t_0)$ are linearly independent.
To see this, suppose that there are coefficients $\lambda_\alpha$ so that
\begin{equation}
\sum_{\alpha}\lambda_\alpha J_\alpha(t_0)
=
0.
\end{equation}
Then the $J=\sum_{\alpha}\lambda_\alpha J_\alpha$ is a Jacobi field which vanishes at $t=a$ and $t=t_0$.
As there are no conjugate points by assumption, $J$ must vanish identically.
Therefore
\begin{equation}
0
=
D_tJ(a)
=
\sum_{\alpha}\lambda_\alpha \zeta_\alpha.
\end{equation}
The vectors $\zeta_\alpha$ are linearly independent, so every $\lambda_\alpha$ vanishes.
This proves the linear independence.\footnote{One could say that the vectors $J_\alpha(t)$ form a ``Jacobi frame'' along $\gamma$. This provides a valid basis in every tangent space due to the lack of conjugate points.}
The Jacobi fields $J_\alpha(t)$ therefore constitute a basis for the orthogonal complement of $\dot\gamma(t)$ in $T_{\gamma(t)}M$ for any $t>a$.

We can thus write our normal vector field $V\in\NVF_0(\gamma)$ in this basis:
\begin{equation}
V(t)
=
\sum_\alpha V_\alpha(t)J_\alpha(t).
\end{equation}
Here $V_\alpha(t)$ are real-valued functions.
As $V(a)=0$, the functions $V_\alpha(t)$ are smooth up to $t=a$; see exercise~\ref{ex:little-Bezout}.

Let us denote
\begin{equation}
A(t)
=
\sum_\alpha \dot V_\alpha(t)J_\alpha(t)
\end{equation}
and
\begin{equation}
B(t)
=
\sum_\alpha V_\alpha(t)D_tJ_\alpha(t).
\end{equation}
With this notation we have $D_tV=A+B$.

Let us compute $\partial_t\ip{V}{B}$ --- this turns out to simplify matters greatly.
At first we get
\begin{equation}
\partial_t\ip{V}{B}
=
\ip{D_tV}{B}
+
\ip{V}{D_tB}.
\end{equation}
We already know that $D_tV=A+B$, so let us find $D_tB$.
The Leibniz rule and the Jacobi equation give
\begin{equation}
\begin{split}
D_tB
&=
\sum_\alpha
\left[
\dot V_\alpha(t)D_tJ_\alpha(t)
+
V_\alpha(t)D_t^2J_\alpha(t)
\right]
\\&=
\sum_\alpha
\left[
\dot V_\alpha(t)D_tJ_\alpha(t)
-
V_\alpha(t)R_\gamma J_\alpha(t)
\right]
\\&=
-R_\gamma V(t)
+\sum_\alpha
\dot V_\alpha(t)D_tJ_\alpha(t)
.
\end{split}
\end{equation}
Using this with exercise~\ref{ex:JF-symmetry} leads to
\begin{equation}
\begin{split}
\ip{V}{D_tB}
&=
-\ip{V}{R_\gamma V}
+\sum_{\alpha}
\ip{V}{\dot V_\alpha D_tJ_\alpha}
\\&=
-\ip{V}{R_\gamma V}
+\sum_{\alpha,\beta}
\ip{V_\beta J_\beta}{\dot V_\alpha D_tJ_\alpha}
\\&=
-\ip{V}{R_\gamma V}
+\sum_{\alpha,\beta}
V_\beta\dot V_\alpha
\ip{J_\beta}{D_tJ_\alpha}
\\&=
-\ip{V}{R_\gamma V}
+\sum_{\alpha,\beta}
V_\beta\dot V_\alpha
\ip{D_tJ_\beta}{J_\alpha}
\\&=
-\ip{V}{R_\gamma V}
+\sum_{\alpha,\beta}
\ip{V_\beta D_tJ_\beta}{\dot V_\alpha J_\alpha}
\\&=
-\ip{V}{R_\gamma V}
+\ip{B}{A}.
\end{split}
\end{equation}
Putting all of this together gives
\begin{equation}
\begin{split}
\partial_t\ip{V}{B}
&=
\ip{A+B}{B}
-\ip{V}{R_\gamma V}
+\ip{B}{A}
\\&=
\abs{D_tV}^2-\abs{A}^2
-\ip{V}{R_\gamma V}.
\end{split}
\end{equation}
Now we can finally turn to the index form.
With these preparations it becomes easy to analyze.

Because $V(a)=0=V(b)$, we have
\begin{equation}
\begin{split}
\IF(V,V)
&=
\int_a^b
\left(
\abs{D_tV}^2
-\ip{R_\gamma V}{V}
\right)
\dd t
\\&=
\int_a^b
\left(
\partial_t\ip{V}{B}
+\abs{A}^2
\right)
\dd t
\\&=
\int_a^b
\abs{A}^2
\dd t
\geq
0.
\end{split}
\end{equation}
If equality holds, then $A=0$, which means that $\dot V_\alpha=0$ and thus each coefficient $V_\alpha(t)$ is constant.
But every $V_\alpha(t)$ vanishes at $t=b$, so $V_\alpha=0$.
This means that $V=0$, so $\IF(V,V)=0$ is only possible when $V=0$.
\end{proof}

\begin{remark}
If there are conjugate points, the ``Jacobi frame'' used above only fails to be a frame at conjugate points.
This makes one think that perhaps the Hessian only has very few negative eigenvalues and that they should correspond to conjugate points.
This is indeed true but is beyond the scope of this course.
The maximal dimension of a subspace of $\NVF_0(\gamma)$ on which the index form is negative definite is called the index of the geodesic.
This index is finite and is indeed equal to the number of interior conjugate points, as long as one counts with multiplicity.
\end{remark}

\begin{iex}
\label{ex:JF-symmetry}
Let $J_1$ and $J_2$ be two Jacobi fields along the same geodesic.
Show that
\begin{equation}
\partial_t
\left(
\ip{D_tJ_1}{J_2}
-
\ip{J_1}{D_tJ_2}
\right)
=
0.
\end{equation}
Conclude that if $J_1$ and $J_2$ both vanish at the same point, then $\ip{D_tJ_1}{J_2}=\ip{J_1}{D_tJ_2}$ at all times.
\end{iex}

\begin{ex}
\label{ex:little-Bezout}
Little B\'ezout's theorem concerns polynomials: If $r$ is a root of a polynomial $p$, then $p(x)=(x-r)q(x)$ for some polynomial $q$.

Show that a similar result holds for smooth functions.
That is, show that if $f\in C^\infty(\R)$ and $f(0)=0$, then $f(t)=tg(t)$ for some smooth function $g$.
A neat way to do this is to compute $\int_0^1\Der{t}f(tx)\dd t$ in two ways.
This gives an explicit formula for $g$ as an integral, and smoothness is far easier to see than by studying $g(t)=f(t)/t$.
\end{ex}

\begin{theorem}
\label{thm:IF-definite}
Let $\gamma\colon[a,b]\to M$ be a geodesic.
Consider the index form $\IF_\gamma$ along it on $\NVF_0$.
\begin{enumerate}
\item
If there are no conjugate points along $\gamma$, then it is positive definite.
\item
If the endpoints are conjugate but there are no other conjugate points, then it is positive semidefinite.
\item
If an interior point is conjugate to another point, then it is indefinite.
\end{enumerate}
\end{theorem}

\begin{proof}
This follows from remark~\ref{rmk:IF-JF}, lemma~\ref{lma:IF<0}, and lemma~\ref{lma:IF>0}.
Recall that there are always vector fields $V\in\NVF_0(\gamma)$ with positive index form.
\end{proof}

\subsection{The index form in constant curvature}
\label{sec:IF-cc}

For a somewhat concrete example, let us take another look at space of constant curvature.
See section~\ref{sec:JF-const-curv}.
In this setting the index form on normal vector fields takes the form
\begin{equation}
\IF(V,W)
=
\int_a^b
\left(
\ip{D_tV}{D_tW}
-
k\ip{V}{W}
\right)
\dd t.
\end{equation}
When $k\leq0$, this is positive definite, and more strongly so when $k<0$.

Indeed, if one studies the forms of Jacobi fields in constant curvature as given in section~\ref{sec:JF-const-curv}, one sees that there are no conjugate points when $k\leq0$.
Theorem~\ref{thm:IF-definite} predicts exactly this behaviour.

If $k>0$ definiteness depends on length.
As we saw in exercise~\ref{ex:indef-not-minimal}, the index form is positive semidefinite (and in fact positive definite) when the geodesic is short enough.
Conjugate points in constant curvature $k>0$ are distance $\pi/\sqrt{k}$ apart.
If the geodesic is longer, then the index form becomes indefinite.

One way to interpret this is to consider the Poincar\'e inequality
\begin{equation}
\int_a^b
\abs{V}^2
\dd t
\leq
C
\int_a^b
\abs{D_tV}^2
\dd t,
\end{equation}
valid for all $V\in\NVF_0(\gamma)$.
If $C$ is small enough, this ensures that the index form is positive.
The constant $C$ becomes bigger when the interval $[a,b]$ gets longer.
At $b-a=\pi/\sqrt{k}$ the optimal Poincar\'e constant $C$ becomes exactly $1/k$, making the index form barely positive semidefinite.

\qa

\section{The tangent bundle}
\label{sec:TM}

\subsection{The tangent bundle as a manifold}

Previously, we have considered the tangent bundle as the disjoint union of tangent spaces:
\begin{equation}
TM
=
\coprod_{x\in M}T_xM.
\end{equation}
While this is correct as a set, there is more structure.
The tangent bundle is a manifold.

It is often convenient to write a tangent vector as a pair $(x,v)$, where $x\in M$ and $v\in T_xM$.
The tangent bundle is the set of all such pairs.
Sometimes the base point $x$ is left implicit.
When $U\subset M$ is open, we denote $TU=\{(x,v)\in TM;x\in U\}$.

Consider an open subset $U\subset M$ and a diffeomorphism $\phi\colon U\to\phi(U)\subset\R^n$.
The coordinate maps of a coordinate chart are often denoted by $x^i$, so that each $x^i\colon U\to\R$ is a smooth function and its differential is the familiar basis covector field $\der x^i$.
That is, at any point $x$ the differential $\der x^i\colon T_xM\to\R$ is a linear map.

Combining the components together, we have the map $\der\phi(x)\colon T_xM\to\R^n$ given by
\begin{equation}
\der\phi(x)v
=
(\der x^1(x)v,\der x^2(x)v,\dots,\der x^n(x)v)
=
(v^1,v^2,\dots,v^n).
\end{equation}
This map is a linear bijection since $\der\phi(x)v$ expresses $v$ in a basis.

We have a map on each tangent space, and we can promote it to a map $\der\phi$ on the whole bundle.
We define $\der\phi\colon TU\to\R^n\times\R^n$ so that
\begin{equation}
\der\phi(x,v)
=
(\phi(x),\der\phi(x)v).
\end{equation}
The base point $x$ is mapped with the coordinate map $\phi$ itself, whereas the tangent vector $v$ is mapped by its differential $\der\phi(x)$.

We want to use $\der\phi$ as a coordinate chart on $TM$ to make it into a manifold.
This chart makes $TU$ look like the product $U\times\R^n$.
However, the tangent bundle is not always a product globally although; this only works for open sets $U$ diffeomorphic to an open Euclidean set.

\begin{ex}
Let $(U_\alpha,\phi_\alpha)_{\alpha\in A}$ be a smooth atlas of $M$.
We defined a topology on $TM$ by saying that $V\subset TM$ is open if and only if $\der\phi_\alpha(TU_\alpha\cap V)\subset\R^{2n}$ is open for all $\alpha\in A$.
Show that this is a topology.
\end{ex}

\begin{ex}
A chart $\phi_\alpha\colon U_\alpha\to\phi_\alpha(U_\alpha)\subset\R^n$ induces a map $\der\phi_\alpha\colon TU_\alpha\to V_\alpha\subset\R^{2n}$ as described above.
Consider two of these, $\alpha=1,2$.
Given the diffeomorphic transition function $\psi$ between $\phi_1$ and $\phi_2$, write down the transition function $\Psi$ between $\der\phi_1$ and $\der\phi_2$.
Prove that it is a diffeomorphism.

This shows that a smooth atlas $(U_\alpha,\phi_\alpha)_{\alpha\in A}$ on $M$ induces a smooth atlas $(TU_\alpha,\der\phi_\alpha)_{\alpha\in A}$ on $TM$.
In particular, $TM$ is a smooth manifold of dimension $2n$.
\end{ex}

\begin{ex}
Is the smooth atlas induced by a maximal smooth atlas maximal?
\end{ex}

A chart $\phi\colon U\to\R^n$ gives local coordinates on $M$.
The map $\der\phi\colon TU\to\R^{2n}$ gives the induced coordinates on $TM$.

\begin{iex}
There is a canonical projection $\pi\colon TM\to M$ given by $\pi(x,v)=x$.
Show that this is a smooth map between the smooth manifolds $TM$ and $M$.
\end{iex}

\begin{iex}
\label{ex:HV-orientation}
Draw a picture of the tangent bundle so that $M$ is horizontal and the fibers are vertical.
Indicate $M$, a point $x$, and a fiber $T_xM$ on it.
It is important to draw the picture in this orientation.
\end{iex}

\subsection{Tensor bundles}

Fix some local coordinates on $U\subset M$.
We saw above that the linear maps $\der x^i\colon T_xM\to\R$ produced a map $T_xM\to\R^n$ and thus local coordinates $TU\to\R^{2n}$.

Recall that $T_xM$ is the dual of $T_x^*M$.
We can use the linear maps $\partial_i\colon T_x^*M\to\R$ to produce a map $T_x^*M\to\R^n$ and thus coordinates on $T^*U$.
A similar construction turns $T^*M$ into a smooth manifold of dimension $2n$.

\begin{remark}
The musical isomorphisms of a Riemannian manifold are diffeomorphisms between $TM$ and $T^*M$.
\end{remark}

Any tensor bundles can be treated in a similar fashion.
For example, consider $TM\otimes TM$.
The basis elements of $T_xM\otimes T_xM$ are $\partial_i\otimes\partial_j$.
The dual basis consists of $\der x^i\otimes\der x^j$ given by
\begin{equation}
\der x^i\otimes\der x^j
(a)
=
a^{ij}
\end{equation}
for $a\in T_xM\otimes T_xM$.
Equivalenty, if we expand $a$ in terms of basis elements as
\begin{equation}
a
=
a^{ij}
\partial_i\otimes\partial_j,
\end{equation}
we can describe the property as
\begin{equation}
\der x^i\otimes\der x^j
(
\partial_k\otimes\partial_l
)
=
\delta_{ik}\delta_{jl}.
\end{equation}
Using the maps $\der x^i\otimes\der x^j\colon TU\otimes TU\to\R$ we get coordinate charts $TU\otimes TU\to\R^{3n}$ on $TM\otimes TM$.
These make the tensor bundle $TM\otimes TM$ into a smooth manifold.

If $E$ is any tensor bundle (like $TM$ or $T^*M\otimes TM$), we denote the projection $\pi\colon E\to M$ by the same symbol.
In general, a bundle is a local product that comes with a global projection.

The preimage $\pi^{-1}(x)$ of a singleton is called a fiber of the bundle.
The fibers of the tangent bundle are the tangent spaces $T_xM=\pi^{-1}(x)$.

\subsection{Tensor fields}

\begin{definition}
A smooth section of a tensor bundle $E$ is a smooth map $f\colon M\to E$ for which $\pi(f(x))=x$ for all $x\in M$.
(In other words, it is a smooth right inverse of the projection $\pi$.)
\end{definition}

A smooth section of the tangent bundle is also called a smooth vector field.
We defined this concept earlier in a different fashion.
Sections of general tensor bundles are called tensor fields.

\begin{ex}
Show that a vector field is smooth if and only if all its components are smooth real-valued functions in any local coordinate system.
This shows that our two definitions of a smooth vector field agree.
The same holds true for tensor fields of any type.
\end{ex}

\subsection{The sphere bundle}

In all of our examples so far the fiber of a bundle is a vector space.
Such bundles are called vector bundles.
There are other kinds of bundles as well, and many interesting ones are obtained by subbundles of vector bundles.
A subbundle is, informally, a subset of a bundle that looks locally like a product.
A subbundle is a submanifold of the bundle.

The most important example to us is the sphere bundle of a Riemannian manifold
\begin{equation}
SM
=
\{(x,v)\in TM;\abs{v}=1\}.
\end{equation}
The fibers $S_xM$ of $SM$ are unit spheres in the tangent spaces $T_xM$.

\begin{ex}
If $f\colon N\to\R$ is a smooth function on a smooth manifold, then the level set $f^{-1}(0)$ is a smooth submanifold if $\der f\neq0$ on this set.
This smoothness follows from the implicit function theorem.
Use this to show that the sphere bundle is a smooth submanifold of $TM$.
\end{ex}

The tensor bundles work on a smooth manifold, but the sphere bundle requires a metric.

\subsection{Directions and iterated bundles}

We can think that $T_xM$ is, informally, the set of all directions one could move on $M$ from $x$.
Thinking of $TM$ as the set of all possible directions of motion is sometimes useful.

The tangent bundle $TM$ is a smooth manifold.
The possible directions on it are described by its tangent bundle, the double tangent bundle $TTM=T(TM)=T^2M$.

The fiber at $(x,y)\in TM$, the space $T_{(x,y)}TM$, describes all the possible directions one can move in from $(x,y)$.
Heuristically, one should be able to move in two kinds of directions: on the base or on the fiber.
This is indeed true invariantly and usefully, but formalizing it is postponed to the next section.

We can, however, describe the tangent vectors in local coordinates.
A local coordinate chart $\phi\colon U\to\R^n$ induces local coordinates $\der\phi\colon TU\to\R^n\times\R^n$ as described above.
Let us denote these coordinates by $x^i$ and $y^i$ --- it makes sense to divide coordinates in two halves for base and fiber.
The natural basis of $T_{(x,y)}TM$ is given by the vectors
\begin{equation}
\partial_{x^1},\dots,\partial_{x^n},\partial_{y^1},\dots,\partial_{y^n}.
\end{equation}
The dual basis on $T^*_{(x,y)}TM$ is given by $\der x^i$ and $\der y^i$.

One can take the tangent bundle of any smooth manifold whatsoever.
A very natural space for us will be $TSM$.

\begin{iex}
Let $M$ have dimension $n$ as always.
What are the dimensions of the smooth manifolds $TM$, $T^*M$, $SM$, $TTM$, and $TSM$?
\end{iex}

Manifolds can be embedded in Euclidean spaces and this can give a way to visualize matters.
But when it comes to the tangent bundle or especially the double tangent bundle, it is far more transparent to work with abstract manifolds.

\subsection{Lifts and geodesics}

Many things can be lifted from manifolds to their tangent bundles.

\begin{ex}
Promoting a smooth function into a function between the bundles is often useful.
We defined earlier the differential of a smooth function $f\colon M\to N$ at $x\in M$ as a linear map $\der f(x)\colon T_xM\to T_{f(x)}N$.
This induces a map $\der f\colon TM\to TN$.
Show that $\der f$ is a bijection if and only if $f$ is a diffeomorphism.
\end{ex}

The lift of a smooth curve $\gamma\colon\R\to M$ is the curve $\sigma\colon\R\to TM$ given by $\sigma(t)=(\gamma(t),\dot\gamma(t))$.
The second order geodesic equation for $\gamma$ is a first order equation for the lift $\sigma$.
We used this to prove existence, uniqueness, and smoothness of geodesics; see exercises~\ref{ex:ge-e/u} and~\ref{ex:exp-smooth}.

Writing a curve $\sigma$ on $TM$ in terms of the local coordinates on $TM$ gives $\der x^i(\sigma(t))=\gamma^i(t)$ and $\der y^i(\sigma(t))=\dot\gamma^i(t)$.
If $\sigma$ is the lift of a geodesic $\gamma$, then $\partial_t\gamma^i=\dot\partial^i$ and $\partial_t\dot\gamma^i=-\Cs ijk\dot\gamma^j\dot\gamma^k$.
In other words,
\begin{equation}
\begin{split}
\partial_t\der x^i(\sigma)
&=
\partial_t\gamma^i
\\&=
\dot\gamma^i
\\&=
\der y^i(\sigma).
\end{split}
\end{equation}
Similarly, $=-\Cs ijk\dot\gamma^j\dot\gamma^k$
\begin{equation}
\begin{split}
\partial_t\der y^i(\sigma)
&=
\partial_t\dot\gamma^i
\\&=
-\Cs ijk\dot\gamma^j\dot\gamma^k
\\&=
-\Cs ijk\der y^j(\sigma)\der y^k(\sigma).
\end{split}
\end{equation}
That is, $\sigma$ satisfies
\begin{equation}
\partial_t\sigma(t)
=
X(\sigma(t)),
\end{equation}
for all $t$,
where $X$ is a vector field on $TM$ given in local coordinates by
\begin{equation}
\label{eq:X-def}
X
=
y^i\partial_{x^i}
-
\Cs ijk y^jy^k\partial_{y^i}.
\end{equation}
This is called the geodesic vector field.

This should be interpreted so that if $\sigma=(x,v)\in TM$ is some initial for a geodesic, $X(\sigma)\in T_{(x,v)}TM$ tells which way the lift of the geodesic $\gamma_{x,v}$ will start moving.
The $x$-component of $\sigma$ moves in the direction of $y$ and the $y$-component moves in a direction depending on the Christoffel symbol.

Let us recall that an integral $\gamma$ of a vector field $V$ on a smooth manifold $N$ is a smooth curve on $N$ satisfying $\dot\gamma(t)=V(\gamma(t))$.

\begin{ex}
\label{ex:integral-curve-geodesic}
Show that if a smooth curve $\sigma\colon\R\to TM$ is the integral curve of the geodesic vector field, then it is a lift of a geodesic.
The opposite conclusion was obtained above.

We have found a new description of geodesics:
A curve is a geodesic if and only if its lift is an integral curve of the geodesic vector field.
Another way to phrase it is that a geodesic is a projection of an integral curve of the geodesic vector field.
\end{ex}

We will study this idea further, but we will first need to split $T_{(x,y)}TM$ into ``base directions'' and ``fiber directions'' invariantly.
The span of the vectors $\partial_{y^i}$ depends on the choice of coordinates.

\qa

\section{Horizontal and vertical subbundles}

As we discussed above, the double tangent bundle $TTM$ describes the directions of motion on $TM$.
There are two basic ways to move: along a fiber or in the base.
As per exercise~\ref{ex:HV-orientation}, directions along the fiber are called vertical and those in the base horizontal.

\subsection{The vertical subbundle}

Consider a curve $\sigma\colon(-\eps,\eps)\to TM$ with $\sigma(0)=\theta=(x,y)\in TM$.
If $\sigma(t)$ stays on the fiber $T_xM$, it makes sense to consider $\dot\sigma(0)\in T_\theta TM$ vertical.
We can describe $\sigma$ staying on the same fiber by saying that $\pi(\sigma(t))$ stays constant.
Differentiating this with respect to $t$ at $t=0$ leads to $\der\pi(\theta)\dot\sigma(0)=0$.

\begin{definition}
\label{def:V-theta}
The vertical fiber at $\theta\in TM$ is
\begin{equation}
V(\theta)
=
\ker(\der\pi(\theta))
\subset
T_\theta TM.
\end{equation}
\end{definition}

Observe that this definition does not depend on the Riemannian metric $g$ and can thus be defined on any smooth manifold.

\subsection{The horizontal subbundle}

Consider again a curve $\sigma$ through $\theta\in TM$.
The velocity of the curve should be considered horizontal if only the base point moves but the tangent vector does not.
But that does not directly make sense; as $x$ changes, we cannot keep $v\in T_xM$ constant.
Fortunately, there is a way to make sense of this through covariant derivatives.

Consider the curve $\gamma=\pi\circ\sigma\colon(-\eps,\eps)\to M$ projected to the base manifold $M$.
Now $\sigma(t)\in T_{\gamma(t)}M$ for all $t$, so $\sigma$ can be regarded as a vector field along the curve $\gamma$.
To make this more precise, we write $\sigma(t)=(\gamma(t),\Sigma(t))\in TM$.
It makes sense to say that the curve $\sigma$ goes in a horizontal direction if the covariant derivative along $\gamma$ vanishes.

To make this more precise, we define a map $K_\theta\colon T_\theta TM\to T_xM$ by requiring that
\begin{equation}
K_\theta\dot\sigma(0)
=
D_t\Sigma(0),
\end{equation}
where $D_t$ is the covariant derivative along the curve $\gamma$.

To make $K_\theta$ a well-defined map, we need to check two properties:
\begin{enumerate}
\item
The map is defined everywhere:
For every $\xi\in T_\theta TM$ there is a curve $\sigma$ through $\theta$ with $\dot\sigma(0)=\xi$.
\item
The map has a unique value everywhere:
If $\sigma_1$ and $\sigma_2$ are two curves through $\theta$ in the direction $\xi\in T_\theta TM$, then $D_t\Sigma_1(0)=D_t\Sigma_2(0)$.
\end{enumerate}

\begin{ex}
Explain why these properties hold and why $K_\theta$ is linear.
\end{ex}

Informally, we can think of a vector $\xi\in T_\theta TM$ as $(A,B)$, where $A$ points along the base and $B$ along the fiber.
In this view $K_\theta\xi$ is the covariant derivative of $B$ in the direction $A$.

We can promote these maps $K_\theta$ into a global connection map
\begin{equation}
K\colon TTM\to TM
\end{equation}
given by
\begin{equation}
K(\theta,\xi)
=
(\pi(\theta),K_\theta\xi).
\end{equation}

\begin{definition}
\label{def:H-theta}
The horizontal fiber at $\theta\in TM$ is
\begin{equation}
H(\theta)
=
\ker(K_\theta)
\subset
T_\theta TM.
\end{equation}
\end{definition}

An alternative way to describe horizontal directions is to require that parallel transported objects are horizontal.
To achieve this, we define a horizontal lift $L_\theta\colon T_xM\to T_\theta TM$ at $\theta=(x,v)$ which describes the ways parallel transports of $v$ evolve in different directions.
Given any $w\in T_xM$, let $\gamma_w$ be a curve through $x$ with $\dot\gamma_w(0)=w$.
Let $P_w^v(t)$ be the parallel transport of $v$ along $\gamma_w$.
We get a curve $\sigma_w^v(t)=(\gamma_v(t),P_w^v(t))$ on $TM$.
We now define
\begin{equation}
L_\theta w
=
\dot\sigma_w^v(0).
\end{equation}
Checking that this map is well defined (independent of the choice of the curve $\gamma_w$) and linear is similar to the check of $K_\theta$.

\begin{ex}
Show that $\ker(K_\theta)=\im(L_\theta)$.
This gives a different way to view $H(\theta)$.
\end{ex}

Notice that $H(\theta)$ does depend on the notion of parallel transport and therefore on $g$.
Vertical directions are smooth concept, horizontal ones are a metric one.

\subsection{Properties of the vertical and horizontal bundles}

The vertical subbundle of $TTM$ is the has the fiber $V(\theta)$ at $\theta\in TM$.
Similarly, the fibers of the horizontal subbundle are $H(\theta)$.
The vertical subbundle gives all the directions along the fibers and the horizontal ones all the directions ``along the base''.

The various maps we have seen so far have several interesting properties.

\begin{iex}
Show that $\der\pi(\theta)\circ L_\theta=\id$ on $T_xM$.
\end{iex}

\begin{ex}
\label{ex:H-proj-bij}
Show that $\der\pi(\theta)|_{H(\theta)}\colon H(\theta)\to T_xM$ is a linear bijection.
\end{ex}

\begin{ex}
\label{ex:V-proj-bij}
Show that $K_\theta|_{V(\theta)}\colon V(\theta)\to T_xM$ is a linear bijection.
\end{ex}

\begin{iex}
\label{ex:H+V}
Show that $T_\theta TM=H(\theta)\oplus V(\theta)$.
That is, show that the horizontal and vertical fibers together span $T_\theta TM$ and they only intersect at the origin.
\end{iex}

In conclusion, $T_\theta TM$ is can be seen as a product of the horizonal and the vertical fiber.
Both $H(\theta)$ and $V(\theta)$ can be identified with $T_xM$.
The projections from $T_\theta TM$ to these two components are $\der\pi(\theta)$ and $K_\theta$.
Indeed, the map
\begin{equation}
\label{eq:j-def}
j_\theta
\colon
T_\theta TM
\to
T_xM\times T_xM
\end{equation}
given by
\begin{equation}
j_\theta(\xi)
=
(\der\pi(\theta)\xi,K_\theta\xi)
\end{equation}
is a linear bijection.

It is useful to denote the horizontal and vertical parts of $\xi\in T_\theta TM$ as $\xi^h=\der\pi(\theta)\xi\in T_xM$ and $\xi^v=K_\theta\xi\in T_xM$.
Identifying with $j_\theta$, we can write $\xi=(\xi^h,\xi^v)$.

\subsection{The Sasaki metric}

The space $T_xM$ has an inner product given by the metric tensor.
The product of two inner product spaces $A$ and $B$ is an inner product space in a natural way:
\begin{equation}
\ip{(a,b)}{(a',b')}_{A\times B}
=
\ip{a}{a'}_A+\ip{b}{b'}_B.
\end{equation}

\begin{definition}
The Sasaki metric on $TM$ is defined so that for each $\theta\in TM$ the map $j_\theta\colon T_\theta TM\to T_xM\times T_xM$ of~\eqref{eq:j-def} is a linear isometry.
\end{definition}

In other words, the Sasaki metric is a metric tensor on $TM$ --- a section of $T^*(TM)\otimes T^*(TM)$ --- defined so that
\begin{itemize}
\item
$H(\theta)$ is orthogonal to $V(\theta)$,
\item
$\der\pi(\theta)|_{H(\theta)}\colon H(\theta)\to T_xM$ is isometric, and
\item
$K_\theta|_{V(\theta)}\colon V(\theta)\to T_xM$ is isometric.
\end{itemize}
For any $\xi,\xi'\in T_\theta TM$ we have
\begin{equation}
\begin{split}
\ip{\xi}{\xi'}
&=
\ip{\der\pi(\theta)\xi}{\der\pi(\theta)\xi'}
+
\ip{K_\theta\xi}{K_\theta\xi'}
\\&=
\ip{\xi_h}{\xi_h'}
+
\ip{\xi_v}{\xi_v'}.
\end{split}
\end{equation}

\subsection{Coordinate expressions}
\label{sec:HV-coords}

Suppose we are given some coordinates $x$ on an open set $U\subset M$.
That is, we have a map $x\colon U\to\R^n$ whose coordinates are $x^i$.
We freely identify the point with its coordinates, so we have dropped the chart $\phi$ altogether from notation.

The local coordinates $x$ on $M$ induce local coordinate $(x,y)$ on $TM$.
Informally, ``$y=\der x$'' since the induced coordinates are given by differentials of the original coordinates.
That is, a vector $v\in T_xM$ can be written as
\begin{equation}
v
=
y^i\partial_{x^i}.
\end{equation}
If we stay on the same fiber, only the variable $y$ changes.
The coordinates $y^i$ on the fiber are simply the components of the tangent vector in the coordinates $x$.

Similarly, the local coordinates $(x,y)$ on $TM$ induce local coordinates $(x,y,X,Y)$ on $TTM$.
A vector $\xi\in T_\theta TM$ at $\theta=(x,y)\in TM$ can be written as
\begin{equation}
\xi
=
X^i \partial_{x^i} + Y^i \partial_{y^i}.
\end{equation}
The vectors $\partial_{x^i}$ and $\partial_{y^i}$ form a basis for $T_\theta TM$, but this basis does not go well together with the decomposition to horizontal and vertical directions.
The vertical part behaves better.

\begin{lemma}
\label{lma:V-basis}
The vector fields $\partial_{y^i}$ are a basis for $V(\theta)$.
\end{lemma}

\begin{proof}
Recall definition~\ref{def:V-theta} and exercise~\ref{ex:H-proj-bij}.
The claim of the lemma follows from $\der\pi(\theta)\partial_{y^i}=0$ (so that the vectors are in the right space) and $K_\theta\partial_{y ^k}=\partial_{x^k}$ (so that the isomorphism maps them to a known basis).

To show the first property, consider $\der\pi(\theta)\partial_{y^i}\in T_xM$ as a derivation.
To that end, let $f\colon M\to\R$ be smooth.
We have
\begin{equation}
\begin{split}
\der\pi(\theta)\partial_{y^i}f|_x
&=
\der f(x)[\der\pi(\theta)\partial_{y^i}]
\\&=
\der(f\circ\pi)(\theta)\partial_{y^i}
\\&=
\partial_{y^i}(f\circ\pi)|_\theta.
\end{split}
\end{equation}
The function $f\circ\pi\colon TM\to\R$ is constant on fibers, so $\partial_{y^i}(f\circ\pi)=0$.

Let us then move to the second claim.
To use the definition (or defining property) of $K_\theta$, we need a curve $\sigma(t)$ on $TM$ for which $\sigma(0)=\theta$ and $\dot\sigma(0)=\partial_{y^i}$.
In local coordinates this can be achieved with $\sigma(t)=(\gamma(t),\Sigma(t))=(x,v+t\partial_{x^i})$.
This curve stays on the fiber $T_xM$ and its time derivative is the $i$th basis vector on $T_xM$.
Since $\dot\gamma=0$, the covariant derivative is simply
\begin{equation}
D_t\Sigma(t)|_{t=0}
=
\partial_t(v+t\partial_{x^i})|_{t=0}
=
\partial_{x^i}
\end{equation}
as required.
\end{proof}

Let us define new vector fields $\delta_{x^i}=\partial_{x^i}-\Cs jiky^k\partial_{y^j}$.

\begin{lemma}
\label{lma:H-basis}
The vector fields $\delta_{x^i}$ are a basis for $H(\theta)$.
\end{lemma}

The proof consists of two steps:

\begin{ex}
Prove that $K_\theta\delta_{x^i}=0$.
\end{ex}

\begin{ex}
Prove that $\der\pi(\theta)\delta_{x^i}=\partial_{x^i}$.
\end{ex}

Using the bases given above, any vector $\xi\in T_\theta TM$ can be written as
\begin{equation}
\xi
=
X^i\delta_{x^i}
+
Y^i\partial_{y^i}
\end{equation}
and the horizontal and vertical components are
\begin{equation}
\xi^h
=
\der\pi(\theta)\xi
=
X^i\partial_{x^i}
\end{equation}
and
\begin{equation}
\xi^v
=
K_\theta\xi
=
Y^i\partial_{x^i}.
\end{equation}
The components stay the same but the basis changes, as one might expect of a natural isomorphism.

The inner product in the Sasaki metric between two vectors $\xi,\tilde\xi\in T_\theta TM$ expressed like so is given by
\begin{equation}
\begin{split}
\ip{\xi}{\tilde\xi}
&=
\ip{\der\pi(\theta)\xi}{\der\pi(\theta)\tilde\xi}
+
\ip{K_\theta\xi}{K_\theta\tilde\xi}
\\&=
\ip{X^i\partial_{x^i}}{\tilde X^j\partial_{x^j}}
+
\ip{Y^i\partial_{x^i}}{\tilde Y^j\partial_{x^j}}
\\&=
X^i\tilde X^j\ip{\partial_{x^i}}{\partial_{x^j}}
+
Y^i\tilde Y^j\ip{\partial_{x^i}}{\partial_{x^j}}
\\&=
X^i\tilde X^jg_{ij}
+
Y^i\tilde Y^jg_{ij}.
\end{split}
\end{equation}
This basis makes the structure of the Sasaki metric more transparent.

Let us then consider what happens on the dual side.

\begin{ex}
\label{ex:dual-basis-change}
Let $e_1,\dots,e_k$ be a basis of a vector space.
Suppose another basis is given by $f_i=\sum_jA_{ij}e_j$.
If the dual basis of the original one is given by $e_i^*$, the new dual basis is of the form $f_i^*=\sum_jB_{ij}e_j^*$.
Show that the matrix $B$ is the inverse transpose of $A$.
\end{ex}

The change of basis from $\partial_{x^i}$ and $\partial_{y^i}$ to $\delta_{x^i}$ and $\partial_{y^i}$ is given by a matrix of the form
\begin{equation}
A
=
\begin{pmatrix}
I&-G\\
0&I
\end{pmatrix},
\end{equation}
where $G_{i}^{\phantom{i}j}=\Cs jiky^k$.
Therefore the change of basis for the dual basis is given by the matrix
\begin{equation}
B
=
\begin{pmatrix}
I&0\\
G^T&I
\end{pmatrix}.
\end{equation}
That is, the dual basis is given by $\der x^i$ and $\delta y^i=\der y^i+\Cs ijky^k\der x^j$.

\begin{ex}
This requires that
\begin{equation}
\begin{split}
\der x^i(\delta_{x^j})
&=
\delta^i_j,
\\
\der x^i(\partial_{y^j})
&=
0,
\\
\delta y^i(\delta_{x^j})
&=
0,\quad\text{and}
\\
\delta y^i(\partial_{y^j})
&=
\delta^i_j.
\end{split}
\end{equation}
This ensures that we have indeed a dual basis to $\delta_{x^i}$ and $\partial_{y^i}$.
This follows from the general observation of exercise~\ref{ex:dual-basis-change} and the considerations after it, but it is worthwhile to verify by hand.
\end{ex}

Observe that we needed to fix the base components $\partial_{x^i}$ to get a nice basis for $T_\theta TM$ but the fiber components $\der y^i$ to get a nice basis on $T_\theta^*TM$.

If needed, these new basis vectors can be used to span the cohorizontal and covertical subspaces of $T_\theta^*TM$.
This will rarely be needed, as the Sasaki metric gives a way to identify $T_\theta^*TM$ with $T_\theta TM$.

\qa

\section{The geodesic flow}

\subsection{Smooth dynamical systems}
\label{sec:dyn-sys}

A smooth dynamical system or a flow on a smooth manifold $N$ is a smooth action of the group $(\R,{+})$ on the diffeomorphism group of $N$.
More concretely, it is a family of smooth maps $\phi_t\colon N\to N$ so that
\begin{itemize}
\item
$\phi_t$ depends smoothly on $t$,
\item
$\phi_0=\id$, and
\item
$\phi_t\circ\phi_s=\phi_{t+s}$.
\end{itemize}
We often speak of such systems so that a point $x\in N$ flows to the point $\phi_t(x)\in N$ in time $t$.
The curves $t\mapsto\phi_t(x)$ are called trajectories.

\begin{ex}
\label{ex:union-traject}
Show that $N$ is a disjoint union of trajectories.
\end{ex}

\begin{ex}
Show that each $\phi_t\colon N\to N$ is a diffeomorphism.
\end{ex}

The dynamical system gives rise to a vector field $G$ on $N$.
It can be defined as a velocity of a trajectory (vector as the velocity of a curve) or as a differential operator (vector as a derivation).
The first point of view says that
\begin{equation}
G(x)
=
\partial_t\phi(x)|_{t=0}.
\end{equation}
Taking the second point of view, we can differentiate $f\colon N\to\R$ along the flow by
\begin{equation}
Gf(x)
=
\partial_tf(\phi_t(x))|_{t=0}.
\end{equation}
As the derivative can be written as $\der f(x)(G)$, the two descriptions agree.

\begin{ex}
Show that a trajectory of the flow $\phi_t$ is an integral curve of $G$.
That is, show that a trajectory $\sigma(t)=\phi_t(x)$ satisfies $\dot\sigma(t)=G(\sigma(t))$.
\end{ex}

The exercise shows that the vector field $G$ determines the flow uniquely.
Therefore $G$ is called the generator of the flow.

\subsection{The geodesic flow}

We defined the geodesic vector field in~\eqref{eq:X-def}.
It is a vector field on $TM$ and therefore a section of $TTM$.

\begin{definition}
The geodesic flow is the flow on the tangent bundle $TM$ of a Riemannian manifold $M$ generated by the geodesic vector field $X$.
\end{definition}

We saw in exercise~\ref{ex:integral-curve-geodesic} that trajectories of the geodesic flow are exactly the lifts of geodesics.

We should hurry to mention that this definition only makes sense as is if $M$ is geodesically complete.
Otherwise some geodesics are not defined for all times.
If $M$ is incomplete, the geodesic flow $\phi\colon\R\times TM\to TM$ is only defined on some open subset of $\R\times TM$.
All our considerations will be local, so it does not matter whether the flow is globally defined or not.
To keep things simple, we assume $M$ to be geodesically complete, but the assumption is unimportant.

If $(x,v)\in TM$ and $t\in\R$, then $\phi_t(x,v)=(\gamma_{x,v}(t),\dot\gamma_{x,v}(t))$ gives the position and direction of the geodesic starting at $(x,v)$ after time $t$.
Exercise~\ref{ex:exp-smooth} proves the smoothness of the geodesic flow, although in that context we only argued that exponential maps are smooth.

On $T\R^n=\R^n\times\R^n$ the geodesic flow is given simply by $\phi_t(x,v)=(x+tv,v)$.

\begin{ex}
Like any vector field on $TM$, the geodesic vector field can be decomposed to horizontal and vertical components.
Verify that at $(x,v)$ we have $X_h=v$ and $X_v=0$.

That is, the geodesic flow heuristically changes the point on the base but keeps the direction fixed.
This corresponds to geodesics parallel transporting their velocity.
\end{ex}

\begin{iex}
If we want to encode all geodesics on $M$ into a dynamical system, why does it have to be a system over $TM$ instead of just $M$?
\end{iex}

\subsection{The differential of the geodesic flow}

The geodesic flow is a smooth map $\phi\colon\R\times TM\to TM$, and for each $t\in\R$ the map $\phi_t\colon TM\to TM$ is a diffeomorphism.
The time derivative is given by the geodesic vector field.
Let us therefore study the derivative of $\phi_t$ for a fixed $t\in\R$.

To this end, consider a smooth curve $\sigma\colon(-\eps,\eps)\to TM$ through $\theta\in TM$.
We would like to find $\partial_s\phi_t(\sigma(s))|_{s=0}$ in terms of $\sigma'(0)$.
The mapping from the latter to the former is $\der\phi_t(\theta)\colon T_\theta TM\to T_{\phi_t(\theta)}TM$.

For each $s\in(-\eps,\eps)$ the curve $t\mapsto\phi_t(\sigma(s))$ is the lift of a geodesic.
Therefore $\Gamma(t,s)=\pi(\phi_t(\sigma(s)))$ is a family of geodesics.
Thus we are led to study the Jacobi field $J(t)=\partial_s\Gamma(t,s)|_{s=0}$ along $\gamma(t)=\Gamma(t,0)$.

If we denote $\sigma'(s)=\partial_s\sigma(s)$ (so that the dot refers to a derivative in $t$ but not in $s$), we have
\begin{equation}
\begin{split}
J(t)
&=
\partial_s\Gamma(t,s)|_{s=0}
\\&=
\partial_s\pi(\phi_t(\sigma(s)))|_{s=0}
\\&=
\der\pi(\phi_t\sigma(0))\der\phi_t(\sigma(0))\sigma'(0)
\\&=
[\der\phi_t(\theta)\sigma'(0)]_h.
\end{split}
\end{equation}
The Jacobi field gives the horizontal part of the differential.

Let us then find the covariant derivative of this Jacobi field.
To that end, we write $\phi_t(\sigma(s))=(\alpha(t,s),\beta(t,s))\in TM$.
We find
\begin{equation}
\label{eq:chain-4}
\begin{split}
D_tJ(t)
&=
D_t\partial_s\Gamma(t,s)|_{s=0}
\\&\stackrel{\text(a)}{=}
D_s\partial_t\Gamma(t,s)|_{s=0}
\\&\stackrel{\text(b)}{=}
D_s\partial_t\pi(\phi_t(\sigma(s)))|_{s=0}
\\&\stackrel{\text(c)}{=}
D_s\partial_t\alpha(t,s)|_{s=0}
\\&\stackrel{\text(d)}{=}
D_s\beta(t,s)|_{s=0}
\\&\stackrel{\text(e)}{=}
K_{\phi_t(\sigma(0))}\partial_s\phi_t(\sigma(s))|_{s=0}
\\&\stackrel{\text(f)}{=}
K_{\phi_t(\theta)}\der\phi_t(\sigma(0))\sigma'(0)
\\&\stackrel{\text(g)}{=}
[\der\phi_t(\theta)\sigma'(0)]_v.
\end{split}
\end{equation}
That is, the covariant derivative of the Jacobi field gives the vertical part of the differential.

\begin{ex}
Explain the named steps in~\eqref{eq:chain-4}.
\end{ex}

To get the initial conditions of the Jacobi field, we study what happens at $t=0$.
There $\phi_0=\id$ and $\der\phi_0=\id$, so $J(0)=[\sigma'(0)]_h$ and $D_tJ(0)=[\sigma'(0)]_v$.

\begin{theorem}
\label{thm:d-flow}
Consider the differential of $\phi_t$ at $\theta\in TM$.
Choose any $\xi\in T_\theta TM$ and denote $\eta\coloneqq\der\phi_t(\theta)\xi\in T_{\phi_t(\theta)}TM$.
If these are decomposed in horizontal and vertical parts as
$\xi=(\xi_h,\xi_v)$
and
$\eta=(\eta_h,\eta_v)$,
then
\begin{equation}
\begin{split}
\eta_h
&=
J_\xi(t)
\quad\text{and}\\
\eta_v
&=
D_tJ_\xi(t),
\end{split}
\end{equation}
where $J_\xi$ is the Jacobi field along the geodesic $t\mapsto\pi(\phi_t(\theta))$ with initial conditions
\begin{equation}
\begin{split}
J_\xi(0)
&=
\xi_h
\quad\text{and}\\
D_tJ_\xi(0)
&=
\xi_v.
\end{split}
\end{equation}
\end{theorem}

\begin{ex}
Prove theorem~\ref{thm:d-flow}.
\end{ex}

Jacobi fields describe perturbations in position (horizontal), whereas their covariant derivatives desribe perturbations in direction (vertical).

If we write the tangent space as $H\oplus V$ at both $\theta$ and $\phi_t(\theta)$, the differential $\der\phi_t\colon T_\theta TM\to T_{\phi_t(\theta)}TM$ of theorem~\ref{thm:d-flow} can be written in block form as
\begin{equation}
\label{eq:d-flow-block}
\der\phi_t(\theta)
=
\begin{pmatrix}
A_{hh} & A_{hv} \\
A_{vh} & A_{vv}
\end{pmatrix},
\end{equation}
where
\begin{equation}
\begin{split}
A_{hh}
&\colon
H(\theta)\to H(\phi_t(\theta)),
\\
A_{hv}
&\colon
V(\theta)\to H(\phi_t(\theta)),
\\
A_{vh}
&\colon
H(\theta)\to V(\phi_t(\theta)),
\quad\text{and}
\\
A_{vv}
&\colon
V(\theta)\to V(\phi_t(\theta))
\end{split}
\end{equation}
are linear maps.
That is,
\begin{equation}
\begin{pmatrix}
\eta_h\\
\eta_v
\end{pmatrix}
=
\begin{pmatrix}
A_{hh} & A_{hv} \\
A_{vh} & A_{vv}
\end{pmatrix}
\begin{pmatrix}
\xi_h\\
\xi_v
\end{pmatrix}.
\end{equation}

\subsection{The exponential map}
\label{sec:exp-flow}

Let us return to the exponential map from section~\ref{sec:exp} and see it from the point of view of the geodesic flow.
The geodesic flow contains the lifts of all geodesics for all times.
The exponential map only contains the geodesics starting from a single point.

\begin{iex}
Show that $\exp_x=\pi\circ\phi_1|_{T_xM}$.
\nts{Diagram?}
\end{iex}

One could say that the exponential map maps directions to points.
Indeed, its differential is indeed a vertical-to-horizontal map.

\begin{ex}
Consider the block structure of $\der\phi_t(\theta)$ at $\theta=(x,v)$ given in~\eqref{eq:d-flow-block}.
Show that $\der\exp_x(v)=A_{hv}$ when one identifies the horizontal and vertical fibers with tangent spaces of $M$ in the canonical way.
\end{ex}

In light of exercise~\ref{ex:cp-char}, the points $\pi(\theta)$ and $\pi(\phi_1(\theta))$ are conjugate along the geodesic $t\mapsto\pi(\phi_t(\theta))$ if and only if $A_{hv}$ is singular.
The whole block matrix of~\eqref{eq:d-flow-block} is always invertible because $\phi_t$ is a diffeomorphism, but the individual blocks can fail to be invertible.

\subsection{The flow on the sphere bundle}

Recall that the unit sphere bundle is the set of those $(x,v)\in TM$ for which $\abs{v}=1$.

Because the speed of a geodesic is constant, the geodesic flow preserves the norm of a vector.
Therefore we may restrict the diffeomorphism
\begin{equation}
\phi_t\colon TM\to TM
\end{equation}
to
\begin{equation}
\phi_t\colon SM\to SM.
\end{equation}
This is a dynamical system on the sphere bundle $SM$, and its generator is still called the geodesic vector field although it is a slightly different object due to the different ambient manifold.
If clarity is required, we will decorate objects with ``$TM$'' or ``$SM$''.

Since all geodesics have unit speed in the flow on $SM$, we miss some directions on the bundle.
The only direction missing on $T_{(x,v)}TM$ is the vertical direction of $v$.
Indeed, we have
\begin{equation}
TSM
=
\{((x,v),\xi)\in TTM;\abs{v}=1,\xi_v\perp v\}.
\end{equation}
The missing direction corresponds to reparametrizations of geodesics, so no geometric information is lost in studying the flow on $SM$.

Consequently, the Jacobi field $t\dot\gamma(t)$ does not appear in the differential of the geodesic flow on $SM$.
We can also further restrict directions so that $\dot\gamma(t)$ does not appear either.
After we have done this in the next section, all Jacobi fields are normal.

\qa

\section{Derivatives on the unit sphere bundle}


\subsection{Horizontal and vertical bundles on $SM$}

The sphere bundle is a level set of the function $f\colon TM\to\R$, $f(x,y)=g_{ij}(x)y^iy^j$.
Given $\theta=(x,y)\in TM$ and $\eta\in T_\theta TM$, let us compute $\der f(\theta)\eta$.
Take any curve $\alpha(t)=(a(t),A(t))$ on $TM$ so that $\alpha(0)=\theta$ and $\dot\alpha(0)=\eta$.
Now $f(\alpha(t))=g_{ij}(a(t))A^i(t)A^j(t)=\ip{A(t)}{A(t)}$ and  using the covariant derivative along $a$ gives $\partial_tf(\alpha(t))=2\ip{A(t)}{D_tA(t)}$.

Let us write $\eta$ in the new basis we found in section~\ref{sec:HV-coords}:

\begin{lemma}
\label{lma:velocity-HV}
Let $\alpha(t)=(a(t),A(t))$ be a curve on $TM$.
Its derivative is
\begin{equation}
\dot\alpha
=
\dot a^i\partial_{x^i}
+
\dot A^i\partial_{y^i}
=
\dot a^i\delta_{x^i}
+
(D_tA)^i\partial_{y^i},
\end{equation}
where $D_t$ is the covariant derivative along $a$.
\end{lemma}

\begin{ex}
Prove the lemma.
\end{ex}

If we write our $\eta\in T_\theta TM$ as
\begin{equation}
\eta
=
X^i\delta_{x^i}
+
Y^i\partial_{y^i},
\end{equation}
lemma~\ref{lma:velocity-HV} gives us $\der f(\theta)\eta=\partial_tf(\alpha(t))|_{t=0}=2g_{ij}(x)y^iY^j$.

Because $SM=f^{-1}(1)\subset TM$ is a level set of $f$, we have
\begin{equation}
TSM
=
\{
X^i\delta_{x^i}
+
Y^i\partial_{y^i}
\in TTM;
(x,y)\in SM,
g_{ij}(x)y^iY^j=0
\}.
\end{equation}
There is no restriction in $X$ but $Y$ cannot have anything in the direction of $y$.
We remedy this asymmetry with the following definition using the decomposition on $TTM$.

\begin{definition}
Let $\theta=(x,y)\in SM$.
We define the horizontal and vertical fibers of $TSM$ at $\theta$ to be
\begin{equation}
H^{SM}(\theta)
=
\{X^i\delta_{x^i}\in TSM;g_{ij}(x)y^iX^j=0\}
\end{equation}
and
\begin{equation}
V^{SM}(\theta)
=
T_\theta SM
\cap
V^{TM}(\theta).
\end{equation}
\end{definition}

In the definition of $H^{SM}(\theta)$ we set $Y^i=0$.
This ensures that the vector is horizontal; see lemmas~\ref{lma:V-basis} and~\ref{lma:H-basis}.
Similarly, in the vertical fiber we have $X^i=0$, so the point on $TSM$ can be written as $Y^i\partial_{y^i}$ with the constraint $g_{ij}(x)y^iY^j=0$.

In words:
\begin{itemize}
\item
At $\theta=(x,v)\in SM$ the horizontal subspace $H^{SM}(\theta)\subset T_\theta SM$ consists of vectors that are purely horizontal and whose horizontal component (a vector on $T_{x}M$) is orthogonal to $v$.
\item
At $\theta=(x,v)\in SM$ the vertical subspace $V^{SM}(\theta)\subset T_\theta SM$ consists of vectors that are purely vertical and whose vertical component (a vector on $T_{x}M$) is orthogonal to $v$.
\end{itemize}

\begin{ex}
We took out one direction from $H^{SM}$, and that is the direction of the geodesic vector field $X$.
Show that this one-dimensional subspace is all that is not horizontal or vertical, that is,
\begin{equation}
\label{eq:TSM=HVX1}
T_\theta SM
=
H^{SM}(\theta)
\oplus
V^{SM}(\theta)
\oplus
\R X(\theta).
\end{equation}
This gives a decomposition of $TSM$ into horizontal, vertical, and geodesic directions.
\end{ex}

Let $N$ be a bundle over $SM$ whose fiber at $(x,y)\in SM$ is
\begin{equation}
N_{(x,y)}
=
\{
v\in T_xM;
\ip{v}{y}=0
\}.
\end{equation}
This bundle gives us a way to formalize tangent spaces of $T_xM$ with the direction of $y$ removed.

The sphere $S_xM$ is a manifold, and so it has a tangent space at every point.
It is well justified to think that $N_{(x,y)}=T_yS_xM$.

\begin{ex}
\label{ex:TSM-K}
Let $\theta\in SM$.
Recall that $T_\theta SM\subset T_\theta TM$.
Show that $K_\theta\colon V^{SM}(\theta)\to N_\theta$ is a linear bijection.
\end{ex}

Because $T_\theta SM\subset T_\theta TM$, the Sasaki metric gives an inner product on $T_\theta SM$.
This is the Riemannian metric the submanifold $SM$ inherits from $TM$.

\begin{ex}
\label{ex:H=H+X}
We defined things so that $H^{TM}(\theta)=H^{SM}(\theta)\oplus\R X(\theta)$.
Show that this direct sum is orthogonal.
Is the decomposition~\eqref{eq:TSM=HVX1} orthogonal as well?
\end{ex}

As in exercise~\ref{ex:TSM-K}, $\der\pi(\theta)\colon H^{SM}(\theta)\to N_\theta$ is a linear bijection.
The Sasaki metric was defined so that our maps $V^{SM}(\theta)\to N_\theta$ and $H^{SM}(\theta)\to N_\theta$ are isometries.

\begin{remark}
\label{rmk:N-NVF}
If we have a section $W$ of the bundle $N$ and a unit speed geodesic $\gamma$, we get a natural normal vector field along $\gamma$ as follows.
The lift of $\gamma$ is the curve $(\gamma,\dot\gamma)$ on $SM$.
At every point $W_\gamma(t)\coloneqq W(\gamma(t),\dot\gamma(t))\in T_{\gamma(t)}M$.
Because the fiber $N_{\gamma(t),\dot\gamma(t)}$ is the orthogonal complement of $\dot\gamma(t)$, the vector field $W_\gamma(t)$ is indeed orthogonal to $\dot\gamma(t)$.
\end{remark}

\subsection{Horizontal and vertical gradients on $SM$}


As we will now work mostly on the sphere bundle, let us drop the decorations and write $H^{SM}(\theta)=H(\theta)$ and $V^{SM}(\theta)=V(\theta)$.

Now that we have a handle on different directions on $SM$, let us differentiate functions.
Consider a function $u\colon SM\to\R$.
As the Sasaki metric makes $SM$ into a Riemannian manifold, $u$ has a gradient $\nabla u(\theta)\in T_\theta SM$ at $\theta\in SM$.
Using the decomposition (see~\eqref{eq:TSM=HVX1})
\begin{equation}
\label{eq:TSM=HVX2}
T_\theta SM
=
H(\theta)
\oplus
V(\theta)
\oplus
\R X(\theta),
\end{equation}
we may decompose the gradient these three parts.
The last element in the decomposition is one-dimensional, so it makes sense to treat the third component of the full gradient as a scalar.

Once we identify $H(\theta)$ and $V(\theta)$ with $N_\theta$, we have
\begin{equation}
T_\theta SM
=
N_\theta
\times
N_\theta
\times
\R.
\end{equation}
In this decomposition
\begin{equation}
\label{eq:SM-grad}
\nabla u(\theta)
=
(
\hd u(\theta)
,
\vd u(\theta)
,
Xu(\theta)
),
\end{equation}
where $\hd u$ and $\vd u$ are the horizontal and vertical gradients of $u$.
Both $\hd u$ and $\vd u$ are sections of the bundle $N$ because at every point $\theta\in SM$ they take values in $N_\theta$.

Let us write these derivatives in terms of coordinates.
On $TM$ we have the basic derivatives $\delta_{x^i}$ and $\partial_{y^i}$.
If we want to express derivatives on $SM$ using these, we need to extend functions from $SM$ to $TM$ to differentiate there.
A natural extension is given by the scaling map $s\colon TM\setminus 0\to SM$, $s(x,y)=(x,y/\abs{y})$.
Now that $u$ is a function on $SM$, the scaled map $u\circ s$ is a smooth function in a neighborhood of $SM$ on $TM$.
We thus define basic derivatives of $u$ as
\begin{equation}
\begin{split}
\delta_iu
&=
\delta_{x^i}(u\circ s)|_{SM}
\quad\text{and}
\\
\partial_iu
&=
\partial_{y^i}(u\circ s)|_{SM}.
\end{split}
\end{equation}
These operators can be used to write the differential of a function.
To get the gradient, we use the musical isomorphisms to write $\delta^i=g^{ij}\delta_j$ and $\partial^i=g^{ij}\partial_j$.

Let us consider the restriction $u_x=u|_{S_xM}$.
The gradient\footnote{The gradient in the inner product space $T_xM$ or its  subset $S_xM$. That is, this is a Euclidean gradient.} of $u_x$ should correspond to $\vd u$.
If we extend $u_x$ to a neighborhood of $S_xM$ in $T_xM$ as $u_x\circ s$ (but evaluate everything on $SM$), then the differential is
\begin{equation}
\der(u_x\circ s)
=
\partial_iu\der y^i.
\end{equation}
The gradient is obtained by musical isomorphism:
\begin{equation}
\nabla(u_x\circ s)
=
\partial^iu\partial_ {y^i}.
\end{equation}
The scaling ensures that the radial derivative of $u_x\circ s$ vanishes, and so the gradient is orthogonal to the radial vector on $TTM$ and the gradient belongs to $V(\theta)$.
Identifying $V(\theta)$ with $N_\theta$ via $K_\theta$, we find the vertical gradient of $u$ to be
\begin{equation}
\vd u
=
\partial^iu\partial_ {x^i}.
\end{equation}
Notice that the natural isomorphism $K_\theta$ changes the basis, not the components.

\begin{ex}
Show that the geodesic vector field operates on $u$ in local coordinates as $Xu(x,v)=v^i\delta_iu(x,v)$ at any $(x,v)\in SM$.
This justifies thinking of the geodesic vector field as ``$X=v\cdot\nabla_x$'', the $x$-derivative in the direction of $v$.
In $\R^n$ we have $S\R^n=\R^n\times S^{n-1}$ and indeed $X=v\cdot\nabla_x$.
\end{ex}

To find the horizontal gradient at $\theta=(x,v)$, we can proceed similarly and differentiate $u\circ s$ on $TM$ using the basis elements $\delta_{x^i}$.
The full horizontal gradient on $H^{TM}(\theta)$ is
\begin{equation}
\delta^iu\delta_{x^i}.
\end{equation}
The component in the direction of $v$ should be projected out, as that is already contained in $Xu$.
Recall exercise~\ref{ex:H=H+X}.
Once we project this geodesic direction out and apply the isomorphism $\der\pi(\theta)\colon H^{SM}(\theta)\to N_\theta$, we find that the horizontal gradient is
\begin{equation}
\hd u
=
(\delta^iu-(Xu)v^i)\partial_{x^i}.
\end{equation}
Now we have found coordinate expressions for the decomposition~\eqref{eq:SM-grad}.

\subsection{Derivatives of sections of $N$}

There is a natural way to integrate on a Riemannian manifold $M$.
The divergence $\dive V$ of a vector field $V$ is defined so that
\begin{equation}
\int_M \ip{V}{\nabla f}
=
-
\int_M f\dive V
\end{equation}
for all smooth compactly supported $f\colon M\to\R$.
In other words, the divergence is the negative formal transpose of the gradient: ``$\dive=-\nabla^T$''.
The divergence is a first order differential operator given in local coordinates as $\dive V=V^i_{\phantom{i};i}$.
It is the trace of the covariant derivative $\nabla V$.

Similarly, we may integrate over the Riemannian manifold $SM$.
The horizontal and vertical divergences $\hdiv u$ and $\vdiv u$ of $u$ are defined similarly through transposes by requiring that
\begin{equation}
\int_{SM}\ip{V}{\hd u}
=
-
\int_{SM} u\hdiv V
\end{equation}
and similarly for $\vdiv$.
The geodesic vector field $X$ is skew-adjoint: $X^T=-X$.

The horizontal and vertical divergences map smooth sections of $N$ into smooth functions on $SM$.

\begin{iex}
The geodesic vector field also operates on sections of $N$.
If $V$ is a section, we define
\begin{equation}
XV(\theta)
=
D_t V(\phi_t(\theta))|_{t=0}.
\end{equation}
This is the same formula as for scalar differentiation, but the derivative is now covariant.
Show that $XV$ is a section of $N$.

It follows that if we restrict $V$ to a normal vector field $V_\gamma$ along $\gamma$ as in remark~\ref{rmk:N-NVF}, then the geodesic vector field corresponds to the covariant derivative along the geodesics.
That is, $(XV)_\gamma=D_tV_\gamma$.
\end{iex}

Unfortunately more complete details are beyond the scope of this course.

\subsection{Commutator relations}

Now that we can differentiate, the question arises whether the various differential operators commute.
This is easiest to study on $TM$ first.
The coordinate derivatives $\partial_{x^i}$ and $\partial_{y^i}$ all commute with each other.

\begin{iex}
Show that $[\delta_{x^i},\partial_{y^j}]=\Cs kij\partial_{y^k}$.
\end{iex}

The commutator $[\delta_{x^i},\delta_{x^j}]$ will involve derivatives of Christoffel symbols.
Alternatively, it can be seen as a commutator of covariant derivatives.
Either way, it should be no surprise that the commutator contains the curvature operator.
Using the commutator relations for the basis elements on $TTM$ allows one to compute the commutators for the various derivatives on $SM$.

Recall the curvature operator along a geodesic $\gamma$ from definition~\ref{def:Rg}.
It is an operator depending on $\gamma(t)$ and $\dot\gamma(t)$ and maps $T_{\gamma(t)}M\to T_{\gamma(t)}M$.
By lemma~\ref{lma:Rg-normal} it also maps $N_{(\gamma(t),\dot\gamma(t))}$ to itself.
Therefore the curvature operator induces a map $R$ that maps sections of $N$ to sections of $N$.

\begin{proposition}
\label{prop:commutators}
The differential operators on $SM$ satisfy the following commutator relations:
\begin{equation}
\begin{split}
[X,\vd]
&=
-\hd,
\\
[X,\hd]
&=
R\vd,
\\
\hdiv\vd-\vdiv\hd
&=
(n-1)X,
\\
[X,\vdiv]
&=
-\hdiv,
\\
[X,\hdiv]
&=
\vdiv R.
\end{split}
\end{equation}
\end{proposition}

We will not prove this proposition\footnote{See~\cite{PSU-paper} for a proof and more details.}, but we observe a redundancy.

\begin{ex}
Prove the formula $[X,\vdiv]=-\hdiv$ assuming $[X,\vd]=-\hd$ using the definitions by formal transposes.
(A similar argument works for the two commutator relations involving the curvature operator.)
\end{ex}

\subsection{The Santal\'o formula}

Let us return briefly to integration over $SM$.
While the exact proofs would consume too much time, there is an important idea that we need to discuss: a change of variables associated with the geodesic flow.

To make everything well defined, we have to impose restrictions on the geometry.
First of all, we assume $M$ to be a compact Riemannian manifold with boundary.
One can define manifolds with boundary abstractly, but one can also think of $M$ as a compact subset with a smooth boundary inside a Riemannian manifold without boundary.

The manifold $M$ has a boundary $\partial M$, and so has its sphere bundle $SM$:
\begin{equation}
\partial(SM)
=
\{(x,v)\in SM;x\in\partial M\}.
\end{equation}
A vector at the boundary can point in three kinds of directions: inwards, tangentially to $\partial M$, or outwards.

Let $\nu(x)$ be the outer unit normal vector at $x\in\partial M$.
Then the tangential vectors at $x$ are precisely those that are normal to $\nu(x)$.
The inward pointing boundary is
\begin{equation}
\inwb
=
\{
(x,v)\in\partial(SM);
\ip{v}{\nu(x)}<0.
\}
\end{equation}
This set parametrizes all geodesics that start at the boundary and go inwards.

To describe how far the geodesic can be extended before falling out of the manifold, we define $\tau\colon SM\to\R$ to be the travel time function so that a geodesic starting at $(x,v)\in SM$ can be maximally extended to the future to be defined on $[0,\tau(x,v)]$.

We want to rule out two problems:
\begin{enumerate}
\item
There might be geodesics that do not meet the boundary and are thus not parametrized by $\inwb$.
\item
Some geodesics might start tangentially by still go inside the manifold.
\end{enumerate}
To rule out the first one, we assume that every maximal geodesic has finite length.
In other words, given any point and direction, the geodesic comes out in finite time.
To rule out the second one, we assume that the boundary is strictly convex in the sense that the second fundamental form of the boundary is positive definite.

As $\inwb\subset\partial(SM)\subset SM$ is a submanifold, it inherits a Riemannian metric.
Therefore one may integrate over it.
Let $\mu$ be the natural measure on $SM$ and $\lambda$ the one one $\partial(SM)$.
A measure more compatible with the geodesic flow is obtained by $\tilde\lambda=\abs{\ip{v}{\nu}}\lambda$.

\begin{proposition}[The Santal\'o formula]
\label{prop:Santalo}
Let $M$ be a compact Riemannian manifold with boundary so that every geodesic has finite length and the boundary is strictly convex.
Then for any smooth $u\colon SM\to\R$ we have
\begin{equation}
\begin{split}
&\int_{(x,v)\in SM}u(x,v)\dd\mu(x,v)
=
\\&\quad\int_{(x,v)\in\inwb}
\int_0^{\tau(x,v)}
u(\phi_t(x,v))
\dd t
\dd\tilde\lambda(x,v).
\end{split}
\end{equation}
\end{proposition}

We omit the proof.\footnote{See~\cite[lemma 3.3.2]{S}.}

So, to integrate over $SM$, one can integrate over the space of all geodesics ($\inwb$) and then over each geodesic.
Think of this as a Fubini-type theorem.
In the usual Fubini theorem, one can write the plane as a disjoint union of parallel lines and integrate first over each line and then integrate all those integrals together.
Now we just write $SM$ as a union of trajectories of the geodesic flow; see exercise~\ref{ex:union-traject}.

\qa

\section{Geodesic X-ray tomography}
\label{sec:xrt}

This section is devoted to a problem whose solution serves as a recap of the course and shows how to apply the tools.
The question is:
Is a function $f\colon M\to\R$ on a Riemannian manifold uniquely determined by its integrals over all geodesics?

\subsection{The geodesic X-ray transform}

To formalize the question, we define the geodesic X-ray transform.
If $\Gamma$ is the set of all maximal unit speed geodesics on $M$, the geodesic X-ray transform of $f\colon M\to\R$ is the function $\xrt f\colon\Gamma\to\R$ given by
\begin{equation}
\xrt f(\gamma)
=
\int_a^b f(\gamma(t))\dd t
\end{equation}
for a maximal geodesic $\gamma\colon(a,b)\to M$.

Even if $f$ is smooth and compactly supported, the integral might not exist over all geodesics.
Therefore we need to impose some restrictions on the geometry of $M$.
We assume $M$ to be compact and all maximal geodesics to have finite length.
The the operator $\xrt$ is well defined on $C^\infty(M;\R)$.

The problem is easiest to study when the space $\Gamma$ of all geodesics has a good structure.
To that end we require that the boundary is strictly convex.
This ensures that all geodesics are parametrized by the submanifold $\inwb\subset SM$.

\begin{ex}
We can always take $\Gamma$ to be the quotient of $SM$ by the geodesic flow.
That is, we can define an equivalence relation on $SM$ so that $\theta\sim\theta'$ if and only if $\theta'=\phi_t(\theta)$ for some $t\in\R$.
There are manifolds for which the geodesic flow has a dense trajectory on $SM$.
Show that in this case the quotient $SM/{\sim}$ is not a topological manifold.
\end{ex}

\begin{ex}
How would you parametrize geodesics in $\R^n$?
The parametrization can be redundant.
Give a formula for $\xrt f$ in $\R^n$, when $f$ is smooth and compactly supported.
\end{ex}

Furthermore, to avoid problems near the boundary, we only study functions that are compactly supported in the interior of the manifold $M$.
That is, there is a positive distance between $\partial M$ and $\spt(f)$.
In this case we obtain an operator $\xrt\colon C^\infty_c(M;\R)\to C^\infty_c(\inwb;\R)$.

The question is: Is this operator injective?
That is, do the integrals of $f$ over all geodesics $\gamma\in\Gamma$ determine $f$ uniquely?

We shall show that the operator is indeed injective.
To do so, we need to show that if $f\in\ker(\xrt)$, then $f=0$.

\subsection{The transport equation}

Take any smooth $f\colon M\to\R$.
We define its integral function $u^f\colon SM\to\R$ to be
\begin{equation}
\label{eq:u-def}
u^f(x,v)
=
\int_0^{\tau(x,v)}f(\gamma_{x,v}(t))\dd t.
\end{equation}
Recall that $\gamma_{x,v}$ is the geodesic starting at $(x,v)$ and $\tau(x,v)$ is the travel time function.
The integral is taken from any point all the way up to the boundary.

As geodesics are parametrized by their starting points at $\inwb$, we may actually write $\xrt f=u^f|_{\inwb}$.
The restriction $u^f|_{\partial(SM)\setminus\inwb}$ is always zero because the geodesics to be integrated over have zero length.

\begin{ex}
The manifold $M$ with boundary $\partial M$ can be thought of as follows.
Consider a Riemannian manifold $\tilde M$ without boundary and a smooth function $\rho\colon\tilde M\to\R$.
Suppose $M=\rho^{-1}([0,\infty))$, $\partial M=\rho^{-1}(0)$, and $\der\rho\neq0$ at $\partial M$.
(Smooth domains can always be defined in terms of a smooth defining function $\rho$ like this.)

Take any $(x,v)\in SM\setminus\partial(SM)$ so that the maximal geodesic starting there meets boundary in finite time and is not tangent to it at the exit point.
Use the implicit function theorem to show that the travel time function $\tau$ is smooth in a neighborhood of $(x,v)$.
It follows then from our assumptions that $\tau$ is smooth in all of $SM\setminus\partial(SM)$.
\end{ex}

\begin{lemma}
\label{lma:u-smooth}
If $f$ is a smooth compactly supported function in the kernel of $\xrt$, then $u^f$ is smooth and compactly supported.
\end{lemma}

\begin{proof}
Everything appearing in the defining integral~\eqref{eq:u-def} is smooth in the interior $SM\setminus\partial(SM)$, so $u^f$ is smooth in this set.

If $x$ is close enough to $\partial M$, then for any $v\in S_xM$ either $\gamma_{x,v}$ or $\gamma_{x,-v}$ will avoid the support of $f$ for all future times.\footnote{This is a little tricky to prove precisely, but the geometric intuition is hopefully clear enough.}
As $f\in\ker(\xrt)$, we have $u^f(x,v)+u^f(x,-v)=0$.
Thus $u^f(x,v)=0$ when $x$ is close enough to $\partial M$.
\end{proof}

\begin{ex}
Now that we have established that $u^f$ is regular, it remains to establish its crucial property.
Show that $Xu^f=-\pi^*f$.

Here the pullback $\pi^*f$ means the composition $f\circ\pi$.
As $f$ is a function of $x\in M$ only, and $\pi^*f$ promotes it into a function of $(x,v)\in SM$ which does not depend on $v$.
\end{ex}

We now know that if $\xrt f=0$, then $u^f$ is smooth (and compactly supported) and satisfies the transport equation
\begin{equation}
\begin{cases}
Xu^f=-\pi^*f & \text{in }SM,\\
u^f=0 & \text{on }\partial(SM).
\end{cases}
\end{equation}
We will show that this boundary value problem for a partial differential equation has the unique solution $u^f=0$.
It then follows that $\pi^*f=-Xu^f=0$ and so $f=0$.
This shows that $\xrt$ is injective.

To show the uniqueness of the solution of the transport equation, observe that the right-hand side of the transport equation $Xu^f=-\pi^*f$ is independent of direction.
Therefore is derivative with respect to the direction $v$ vanishes.
In other words,
\begin{equation}
0
=
\vd(-\pi^*f)
=
\vd X u^f.
\end{equation}
Now we have found a homogeneous second order equation for $u^f$.

\subsection{The Pestov identity}

To show uniqueness of solutions to the PDE $\vd X u=0$, we will use an energy identity known as a Pestov identity or Mukhometov--Pestov identity.
The identity is not hard to prove using our tools, but it can be hard to guess.

\begin{proposition}[Pestov identity]
\label{prop:Pestov}
If $u\colon SM\to\R$ is smooth and compactly supported, then
\begin{equation}
\int_{SM}\abs{\vd Xu}^2
=
\int_{SM}\abs{X\vd u}^2
-
\int_{SM}\ip{\vd u}{R\vd u}
+
\int_{SM}\abs{Xu}^2.
\end{equation}
\end{proposition}

\begin{proof}
We will write the various integrals as norms and inner products in $L^2(SM)$.
Compact support allows us to integrate by parts without boundary terms.
We want to compute
\begin{equation}
\begin{split}
\aabs{\vd Xu}^2
-\aabs{X\vd u}^2
&=
\iip{\vd Xu}{\vd Xu}
-\iip{X\vd u}{X\vd u}
\\&=
-\iip{\vdiv\vd Xu}{Xu}
+\iip{XX\vd u}{\vd u}
\\&=
\iip{X\vdiv\vd Xu}{u}
-\iip{\vdiv XX\vd u}{u}
\\&=
\iip{(X\vdiv\vd X-\vdiv XX\vd)u}{u}.
\end{split}
\end{equation}
To simplify this, we apply the commutator rules of proposition~\ref{prop:commutators} to find
\begin{equation}
\begin{split}
X\vdiv\vd X-\vdiv XX\vd
=
&=
(\vdiv X-\hdiv)\vd X-\vdiv X(\vd X-\hd)
\\&=
-\hdiv\vd X+\vdiv X\hd
\\&=
-\hdiv\vd X+\vdiv(\hd X+R\vd)
\\&=
(\vdiv\hd-\hdiv\vd)X
+\vdiv R\vd
\\&=
-(n-1)XX
+\vdiv R\vd.
\end{split}
\end{equation}
Therefore
\begin{equation}
\begin{split}
\aabs{\vd Xu}^2
-\aabs{X\vd u}^2
&=
\iip{(X\vdiv\vd X-\vdiv XX\vd)u}{u}
\\&=
-(n-1)\iip{XXu}{u}
+\iip{\vdiv R\vd u}{u}
\\&=
(n-1)\iip{Xu}{Xu}
-\iip{R\vd u}{\vd u}.
\end{split}
\end{equation}
This is the claimed identity.
\end{proof}

\begin{ex}
What is the commutator $[\vdiv\vd,X]$?
\end{ex}

The Pestov identity is easy to use when $\vd Xu=0$ as in our case.
Let us try to understand the structure of the right-hand side better.
The first two terms only depend on $u$ through $V\coloneqq\vd u$, which is a smooth section of the bundle $N$.
Per remark~\ref{rmk:N-NVF} this section gives rise to a normal vector field $V_\gamma\in\NVF_0(\gamma)$ along any maximal geodesic $\gamma$.
The zero boundary values are due to compact support.

\begin{lemma}
\label{lma:Pestov-IF}
If $V$ is a smooth section of the bundle $N$, then
\begin{equation}
\begin{split}
&\int_{SM}\left(\abs{XV}^2-\ip{V}{RV}\right)
\\&\quad=
\int_{(x,v)\in\inwb}
\IF_{\gamma_{x,v}}(V_{\gamma_{x,v}},V_{\gamma_{x,v}})
\dd\tilde\lambda(x,v),
\end{split}
\end{equation}
where $\lambda$ is the Riemannian volume measure on $\inwb$, $\IF_\gamma$ is the index form along $\gamma$, and $V_\gamma$ is the normal vector field along $\gamma$ arising from the section $V$ of $N$.
\end{lemma}

\begin{proof}
We begin by applying the Santal\'o formula of proposition~\ref{prop:Santalo} to our integral over $SM$.
In the notation of the proposition, $u(x,v)=\abs{XV}^2-\ip{V}{RV}|_{(x,v)}$.
Santal\'o gives an integral over the inward pointing boundary $\inwb$ and over each geodesic $\gamma=\gamma_{x,v}$ we end up with the integral
\begin{equation}
\begin{split}
\int_0^{\tau(x,v)}
u(\phi_t(x,v))
\dd t
&=
\int_0^{\tau(x,v)}
\left(\abs{(XV)_\gamma}^2-\ip{V_\gamma}{RV_\gamma}\right)
\dd t
\\&=
\int_0^{\tau(x,v)}
\left(\abs{D_tV_\gamma}^2-\ip{V_\gamma}{RV_\gamma}\right)
\dd t
\\&=
\IF_\gamma(V_\gamma,V_\gamma).
\end{split}
\end{equation}
This completes the proof.
\end{proof}

To prove uniqueness, we want the right-hand side of our Pestov identity to be positive.
We shall see how to do so soon, but we need to make the right assumption to guarantee positivity.

\subsection{Injectivity on simple manifolds}

\begin{definition}
A simple Riemannian manifold is a compact Riemannian manifold with strictly convex boundary so that each maximal geodesic has finite length and there are no conjugate points.
\end{definition}

For example, the closed Euclidean ball is simple.
We can think of simple manifolds as ``ball-like'', but they are quite a bit more flexible.

\begin{theorem}
\label{thm:xrt}
The geodesic X-ray transform is injective on smooth compactly supported functions on a simple Riemannian manifold of any dimension $n\geq2$.
\end{theorem}

\begin{proof}
Let us take a smooth and compactly supported $f\colon M\to\R$.
We assume that $\xrt f=0$ and aim to show that $f=0$.

The integral function $u^f\colon SM\to\R$ is also smooth and compactly supported by lemma~\ref{lma:u-smooth}.
As we found, this function satisfies $\vd Xu^f=0$.

Let us then turn to the Pestov identity of proposition~\ref{prop:Pestov}:
\begin{equation}
\aabs{\vd Xu^f}^2
=
\aabs{X\vd u^f}^2
-\iip{\vd u^f}{R\vd u^f}
+(n-1)\aabs{Xu^f}^2.
\end{equation}
The left-hand side vanishes.

By theorem~\ref{thm:IF-definite} the index form is positive definite in the absence of conjugate points.
Combining this with lemma~\ref{lma:Pestov-IF} shows that 
\begin{equation}
\aabs{X\vd u^f}^2
-\iip{\vd u^f}{R\vd u^f}
\geq
0.
\end{equation}
Therefore our energy identity reduces to\footnote{This estimate hints at things failing when $n=1$. And injectivity does indeed fail.}
\begin{equation}
0
\geq
(n-1)\aabs{Xu^f}^2.
\end{equation}
This can only hold if $Xu^f=0$.
Therefore $\pi^*f=-Xu^f=0$ and so $f=0$.
\end{proof}

\begin{remark}
One would obtain positivity in the Pestov identity more directly if each of the three terms on the right-hand side is positive.
This is the case if $R\leq0$ in the appropriate sense.
This brings us back to sections~\ref{sec:JF-const-curv} and~\ref{sec:IF-cc}, where we saw that there are no conjugate points in non-positive curvature.
\end{remark}

\begin{iex}
To summarize, list which tools developed through this course were used to prove theorem~\ref{thm:xrt}.
\end{iex}

\subsection{Applications}

The geodesic X-ray transform appears frequently in the theory of inverse problems.
It arises in the study of many inverse boundary value problems for PDEs and as a linearization of non-linear geometric problems.
For example, the derivative of the distance between two points with respect to the Riemannian metric is an X-ray transform of the variation of the metric tensor.
This makes the geodesic X-ray transform appear in linearized travel time tomography in non-Euclidean geometry, which is useful for global seismology and ultrasound imaging.
In $\R^3$ and $\R^2$ the transform has direct medical applications, as computerized tomography (CT) is based on it.

\qa

\section{Looking back and forward}

\subsection{Ways to view geodesics}

We found a number of different ways to see geodesics, as
\begin{enumerate}
\item
critical points of the length functional,
\item
minimizers of length (at least locally),
\item
solutions to the geodesic equation,
\item
as curves that parallel transport their velocity,
\item
as projections of trajectories of the geodesic flow, and
\item
as curves that lift to integral curves of the geodesic vector field.
\end{enumerate}

\begin{ex}
Where were these different aspects discussed in the notes?
Give, briefly and in your own words, each definition of a geodesic from the list above.
\end{ex}

\begin{ex}
How are the different definitions linked to each other?
After all, they define the same concept.
\end{ex}

In addition to the length functional $\ell(\gamma)=\int\abs{\dot\gamma}\dd t$ one can also study the energy functional $E(\gamma)=\frac12\int\abs{\dot\gamma}^2\dd t$.
We left it aside as it does not have an equally clear geometric interpretation.
It has the nice property that all critical points are constant speed geodesics, so it leads to the geodesic equation more directly.

One way to view geodesics that we ignored is to realize the geodesic flow as a Hamiltonian flow on the cotangent bundle with its natural symplectic structure.
This topic is highly recommended to readers with any familiarity with Hamiltonian mechanics --- and those without any.

\subsection{Families of geodesics}

The course had two main goals: to understand individual geodesics and families of geodesics.
There were several different objects that collected or compared various geodesics:
\begin{enumerate}
\item
Jacobi fields,
\item
the exponential map, and
\item
the geodesic flow.
\end{enumerate}

\begin{ex}
Summarize what geodesics are described by each of the three objects above.
\end{ex}

Section~\ref{sec:exp-flow} compared the exponential map to the geodesic flow.
Differentiating either one leads to Jacobi fields.

The flow lives on $SM$, so its differential lives on $TSM$.
This was split in three directions: geodesic, horizontal, and vertical.
The horizontal and vertical components on $TSM$ correspond to Jacobi fields and their covariant derivatives.

The geodesic flow is always a diffeomorphism, but the exponential map as its restriction can fail to be so.
This failure happens locally at conjugate points.
Points are conjugate along a geodesic if a Jacobi field vanishes at both points but not identically.

\subsection{More aspects of geodesics}

We studied local minimization of length and its connections to the index form and conjugate points.
All conjugate points to a given point can be collected in a so-called conjugate locus of that point.
Theorem~\ref{thm:IF-definite} can be rephrased so that geodesics are locally minimal only up to the conjugate locus but not beyond.

The corresponding global minimization works up to the so-called cut locus --- this is not a theorem but a definition.
The conjugate locus is further away than the cut locus by theorem~\ref{thm:IF-definite}.
We did not study global minimization properties of geodesics.

In addition to distances between points, one can study distance between general submanifolds.
Zero-dimensional submanifolds are points.
Existence and uniqueness of minimizing curves between two submanifolds depends on the geometry of the submanifolds in addition to that of the whole manifold.
A minimizing curve is always a geodesic, but the boundary conditions are different.

The endpoint is not fixed, but the direction must be normal to the submanifold.
The nature of this condition depends on the codimension of the submanifold.
The first variation formula has a boundary term that forces this.
The second variation formula has a more complicated boundary term depending on the curvature of the submanifold.

When we studied minimization between points, we used Jacobi fields that correspond to families of geodesics between the two points.
Now those have to be replaced by families of geodesics that are normal to the submanifolds at the endpoints.
This leads to conditions on vanishing Jacobi fields but the initial conditions are different.
Something critical happened when the two points were conjugate.
Similarly, something critical happens with the distance from a hypersurface to a point when the point is a focal point.
Focal points are analogous to conjugate points, but one endpoint has to be replaced by a hypersurface.

We briefly touched upon geodesic spheres in section~\ref{sec:sphere}.
The shapes of these spheres have interesting properties as one varies the radius and the center.
There is an evolution equation for the shape operator of the geodesic sphere that corresponds to the Jacobi equation.

In general, our endeavors have been very local in nature, but there is a substantial amount of global geometry of geodesics to be studied.

\subsection{General geometry}

If you have not read Riemannian geometry before this course, perhaps you have now found a reason to look into the fundamentals of the theory.
Lee's book~\cite{L} is highly recommended for that purpose.

The theory of Riemannian geometry branches out quickly, and we have only focused on the branch along a geodesic.
Matters like integration, curvature, submanifolds, general fiber bundles, and global geometry deserve a look.

Differential geometry does not end with Riemannian geometry.
Pseudo-Riemannian manifolds are very similar to Riemannian manifolds.
The metric tensor is not assumed to be a positive definite and symmetric matrix in local coordinates, but only invertible and symmetric.
With positivity we lose a sense of distance, but many of the considerations do not really rely on distance.
The geodesic equation, parallel transport, the exponential map, Jacobi fields, and the flow work just as well.
If one wants to introduce geodesics as critical points of a functional, energy is better than length.
Pseudo-Riemannian, especially Lorentzian, manifolds are heavily used in general relativity.

A step in a different direction can be taken by throwing away not positivity but the existence of a quadratic form.
If we only require that every tangent space has a (smooth and strictly convex) norm, we end up with Finsler geometry.
Many of our considerations generalize to Finsler geometry, but the details are more technical.
A Finsler manifold has a natural Riemannian metric on each tangent space, but this metric depends on a reference direction.
Therefore Finsler geometry can be seen as ``anisotropic Riemannian geometry''.

Of course, one can drop all metric properties altogether and only study a smooth manifold or perhaps introduce another kind of additional structure.
Or one can keep a metric structure but lose the smooth one and study metric geometry.

The rabbit hole is deep and branches indefinitely.
Nevertheless, the reader is invited to enter.

\subsection{Geodesic flows}

We have only scratched the surface of the theory of geodesic flows.
The global and local geometry of the manifold influence the behavior of the flow.
For example, curvature has an effect on ergodicity.
For a detailed and deep exposition of geodesic flow, see the book~\cite{P} by Paternain.

\subsection{Integral geometry}

In section~\ref{sec:xrt} we studied whether a function is determined by its integrals over geodesics.
This is an example of an inverse problem in integral geometry.
There are a number of different problems in this spirit.
The object to be determined can be a tensor field or a connection, for example.
The problems can also be non-linear, and the task can be to determine the whole manifold from some kind of data.

For integral geometry on manifolds, we recommend the books by Sharafutdinov~\cite{S} and Paternain--Salo--Uhlmann~\cite{PSU-book}.
For the details we omitted in section~\ref{sec:xrt}, the article~\cite{PSU-paper} and its appendices are a good reference.
If you want a big picture of the current state of research on such problems, the review~\cite{IM} and references therein can get you started.

These problems are interesting and highly non-trivial already in Euclidean geometry.
For an overview of the various tools and ideas in Euclidean X-ray tomography, we refer the reader to~\cite{I}.

Another reference on these topics is the email address given on the cover page of these notes.

\subsection{Feedback}

\begin{iex}
Which results or ideas did you find most interesting in this course?
\end{iex}

\begin{iex}
At times, this course had more focus on ideas than technical details than usual.
How did you find this kind of a course?
\end{iex}

\begin{iex}
Do you feel that something was left out?
Is there something --- perhaps some of the further study directions mentioned above --- that you would like to have seen covered?
\end{iex}

\qa

Previous feedback has been of great help in improving these notes.
Many thanks for all students who contributed!

\nts{Vektorikentän komponentit ovat koordinaattifunktioiden derivaattoja.}

\end{document}